\documentclass[reqno,centertags,12pt]{amsart}
\usepackage{amsmath,amsthm,amscd,amssymb,latexsym,verbatim}

\usepackage{graphicx,epsf,cite}

-\marginparwidth \oddsidemargin -4mm \evensidemargin \oddsidemargin

\newtheorem{theorem}{Theorem}[section]

\newtheorem{lemma}[theorem]{Lemma}
\newtheorem{corollary}[theorem]{Corollary}
\theoremstyle{definition}
\newtheorem{definition}[theorem]{Definition}

\theoremstyle{remark}

\newcounter{smalllist}

\DeclareMathOperator*{\arrow}{\rightarrow}

\allowdisplaybreaks
\numberwithin{equation}{section}

\newcommand{\lb}{\label}

\newcommand{\beq}{\begin{equation}}
\newcommand{\eeq}{\end{equation}}

\newcommand{\bal}{\begin{align}}
\newcommand{\eal}{\end{align}}
\newcommand{\bals}{\begin{align*}}
\newcommand{\eals}{\end{align*}}

\newcommand{\bbN}{{\mathbb{N}}}
\newcommand{\bbR}{{\mathbb{R}}}

\newcommand{\bbP}{{\mathbb{P}}}
\newcommand{\bbE}{{\mathbb{E}}}
\newcommand{\bbZ}{{\mathbb{Z}}}
\newcommand{\bbC}{{\mathbb{C}}}

\newcommand{\bbS}{{\mathbb{S}}}

\newcommand{\calC}{{\mathcal C}}
\newcommand{\calF}{{\mathcal F}}

\newcommand{\eps}{\varepsilon}
\newcommand{\del}{\delta}
\newcommand{\tht}{\theta}

\newcommand{\til}{\tilde}

\begin{document}
\title[Propagation of Reactions  in Inhomogeneous Media]
{Propagation of Reactions 
in Inhomogeneous Media}

\author{Andrej Zlato\v s}

\address{\noindent Department of Mathematics \\ University of
Wisconsin \\ Madison, WI 53706, USA \newline Email: \tt
zlatos@math.wisc.edu}


\begin{abstract} Consider reaction-diffusion equation $u_t=\Delta u + f(x,u)$ with $x\in\bbR^d$ and general inhomogeneous ignition reaction $f\ge 0$ vanishing at $u=0,1$.  Typical solutions $0\le u\le 1$ transition from $0$ to $1$ as time progresses, and we study them in the region where this transition occurs.  Under fairly general qualitative hypotheses on $f$ we show that in dimensions $d\le 3$,  the Hausdorff distance of the super-level sets $\{u\ge\eps\}$ and $\{u\ge 1-\eps\}$ remains uniformly bounded in time for each $\eps\in(0,1)$.
Thus, $u$ remains uniformly in time close to the characteristic function of $\{u\ge\tfrac 12\}$ in the sense of Hausdorff distance of super-level sets.
We also show that $\{u\ge\tfrac 12\}$ expands with average speed (over any long enough time interval)  between the two spreading speeds corresponding to any $x$-independent lower and upper bounds on $f$.  On the other hand,  these results turn out to be false in dimensions $d\ge 4$, at least without further quantitative hypotheses on $f$.  The proof for $d\le 3$ is based on showing that as the solution propagates, 
small values of $u$ cannot escape far ahead of values close to 1.  The proof for $d\ge 4$ involves construction of a counter-example for which this fails.

Such results were before known for $d=1$ but are new for general non-periodic media in dimensions $d\ge 2$ (some are also new  for homogeneous and periodic media).  They extend in a somewhat weaker sense to  monostable, bistable, and mixed reaction types, as well as to transitions between general equilibria $u^-<u^+$ of the PDE, and to solutions not necessarily satisfying $u^-\le u\le u^+$.  
\end{abstract}

\maketitle

\section{Introduction and Motivation} \lb{S1}

Reaction-diffusion equations are used to model a host of natural processes such as combustion, chemical reactions, or population dynamics.  The baseline model, which already captures a lot of the properties of the dynamics involved, is the parabolic PDE
\beq \lb{1.1}
u_t = \Delta u + f(x,u)
\eeq
for $u:(t_0,\infty)\times\bbR^d \to \bbR$, where $t_0\in[-\infty,\infty)$ and $d\ge 1$.  If $t_0>-\infty$, then we also let
\beq \lb{1.2}
u(t_0,x)=u_0(x)  \qquad\text{for $x\in\bbR^d$.}  
\eeq
 The Lipschitz {\it reaction function} $f$ is such that there exist two ordered {\it equilibria} (time-independent solutions) 
$u^-<u^+$ for \eqref{1.1}, and one is usually interested in studying the transition of general solutions of \eqref{1.1} from one to the other as $t\to\infty$.

A prototypical situation is when $u^-\equiv 0$ and $u^+\equiv 1$, with $f\ge 0$ vanishing at $u=0,1$.
Here $u\in[0,1]$ is the (normalized) temperature of fuel, concentration of a reactant, or population density.  Depending on the application, $f$ may be either an {\it ignition reaction} (vanishing near $u=0$) in combustion models; or a {\it monostable reaction} (positive for $u\in (0,1)$) 
such as {\it Zeldovich} and {\it Arrhenius reactions} with $f_u(x,0)\equiv0$ in models of chemical reactions and {\it KPP reaction} with $f(x,u)\le f_u(x,0)u$ for all $u\ge 0$ in population dynamics models. 

For the sake of clarity of presentation, we will first study this scenario, where our main results are Theorems \ref{T.1.2} and \ref{T.1.3}.   Later we will extend these to more general situations: with general $u^-<u^+$, different types of reactions, including mixtures of ignition, monostable, and {\it bistable reactions} (the latter have $[f(x, u)-f(x,u^\pm(x))][u-u^\pm(x)]<0$ for $u$ near $u^\pm(x)$), and for solutions not necessarily satisfying $u^-\le u\le u^+$ (see Theorems \ref{T.1.11} and \ref{T.1.12}).
However, in order to minimize technicalities, our first result will be stated in the even simpler setting of ignition reactions with a constant ignition temperature (see Theorem \ref{T.1.0} below).  

The study of transitions between equilibria of reaction-diffusion equations has seen a lot of activity since the seminal papers of Kolmogorov, Petrovskii, Piskunov  \cite{KPP} and Fisher \cite{Fisher}.  Of central interest has been long time propagation of solutions with ``typical'' initial data, and the related questions about traveling fronts.  The first type of such initial data are {\it spark-like data} --- compactly supported, such as in \eqref{1.5a} below. The second are {\it front-like data} ---  vanishing on a half-space $\{x\cdot e\ge R\}$ for some unit vector $e\in\bbR^d$ and with $\liminf_{x\cdot e\to-\infty} u_0(x)$ close enough to 1, such as in \eqref{1.6a}. For ignition reaction one can also allow rapid decay to 0 as $|x|\to\infty$ or $x\cdot e\to\infty$, such as in \eqref{1.5} and \eqref{1.6}. 
We will for now discuss these data (and also call the corresponding solutions front-like and spark-like), but later we will turn to more general ones (see, e.g.,  Theorem \ref{T.1.3}).

In both cases it was proved, first for homogeneous ($x$-independent) reactions in several dimensions by Aronson, Weinberger \cite{AW} and then for $x$-periodic ones by Freidlin, G\" artner \cite{Freidlin, GF} and Weinberger \cite{Weinberger}, that for typical solutions, the state $u=1$ invades $u=0$ with a speed that is asymptotically constant (in each direction for spark-like data) as $t\to\infty$.  Specifically, that for each unit $e\in\bbR^d$ there is a ({\it front speed}) $c_e>0$ such that for any $\del>0$,
\beq \lb{0.1}
\lim_{t\to\infty} \inf_{x\cdot e\le (c_e-\del)t} u(x,t) = 1 \qquad\text{and}\qquad  \lim_{t\to\infty} \sup_{x\cdot e\ge (c_e+\del)t} u(x,t) = 0
\eeq
for front-like initial data; and there is a ({\it spreading speed}) $s_e\in(0,c_e]$ such that  for any $\del>0$,
\beq \lb{0.2}
\lim_{t\to\infty} \inf_{x\in (1-\del)tS} u(x,t) = 1 \qquad\text{and}\qquad  \lim_{t\to\infty} \sup_{x\notin (1+\del)tS} u(x,t) = 0
\eeq
for spark-like initial data, where  $S:=\{ se \,\big|\, \|e\|=1 \text{ and }  0\le s\le s_e  \}$ is the {\it Wulff shape} for $f$.  Of course, for homogeneous reactions there is $c>0$ such that $s_e=c_e=c$ for all unit $e\in\bbR^d$.

Closely related to this is the study of {\it traveling fronts} for $x$-independent $f$ and {\it pulsating fronts} for $x$-periodic $f$. Traveling fronts are front-like {\it entire} (with $t_0=-\infty$) solutions of \eqref{1.1} moving with a constant speed $c$ in a  unit direction $e\in\bbR^d$, of the form $u(t,x)=U(x\cdot e-ct)$ with $\lim_{s\to-\infty}U(s)=1$ and $\lim_{s\to\infty}U(s)=0$.  Pulsating fronts, first introduced by Shigesada, Kawasaki, Teramoto \cite{SKT} and proved to exist for general periodic $f$ as above 
by Xin \cite{Xin3} and Berestycki, Hamel \cite{BH}, are similar but $u(t,x)=U(x\cdot e-ct, x)$ and $U$ is periodic in the second argument.  
The minimal of the speeds for which such a front exists for a given unit $e\in\bbR^d$ is then precisely $c_e$, and we also have $s_e=\inf_{e'\cdot e>0} [c_{e'}/(e'\cdot e)]$.  

The above results hold for fairly general $f\ge 0$, and there is a vast literature on these and many other aspects of reaction-diffusion equations in homogeneous and periodic media.  Instead of a more comprehensive discussion, we refer to \cite{BH, Weinberger} and the excellent reviews by Berestycki \cite{Berrev} and Xin \cite{Xin2}.

Unsurprisingly, the picture becomes less satisfactory for non-periodic reactions, particularly in the several spatial dimensions case $d\ge 2$.  The above results and the comparison principle show that if $c_0$ and $c_1$ are the ($e$-independent) speeds associated with homogeneous reactions $f_0$ and $f_1$ such that $f_0\le f\le f_1$, 
then \eqref{0.1}, \eqref{0.2} hold with $c_e$ and $S$ replaced by $c_0$ and $B_{c_0}(0)$ in the first statements and by $c_1$ and $B_{c_1}(0)$ in the second ones.  That is, transition between $u\sim 0$ and $u\sim 1$ occurs inside a spatial strip or annulus whose width grows linearly in time with speed $c_1-c_0$ (while for homogeneous and $x$-periodic media it grows sub-linearly, by taking $\delta\to 0$ in \eqref{0.1}, \eqref{0.2}).  In the general inhomogeneous case, these estimates cannot be improved, unless one includes further restrictive hypotheses on $f$ or is willing to tolerate complicated formulas involving $f$.  

For stationary ergodic reactions, the results should hold as originally stated, but also with $|x-x\cdot e|\le Ct$ for any $C<\infty$ in \eqref{0.1}.  For $d\ge 2$  this {\it homogenization} result was  proved only in the KPP case, by Lions, Souganidis \cite{LioSou}.
(This case has an important advantage of a close relationship of the dynamics for \eqref{1.1} and for its linearization at $u=0$.  Other authors also exploited this link in the study of spreading for KPP reactions, e.g., Berestycki, Hamel, Nadin \cite{BHNad}.  However, results aiming to more precisely  locate  the transition region for non-stationary-ergodic reactions are somewhat restricted by the necessity of more complicated hypotheses involving the reaction.)   Results from the present paper can be used to approach this problem for ignition and non-KPP monostable reactions.  This will be done elsewhere.

The above results hold also in the case $d=1$, with the stationary ergodic KPP reaction result proved earlier in \cite{GF}.  However,  some recent developments have gone further, particularly for ignition reactions.  Mellet, Nolen, Roquejoffre, Ryzhik, Sire \cite{NolRyz, MRS, MNRR} proved for reactions of the form $f(x,u)=a(x)f_0(u)$ (with $f_0$ vanishing on $[0,\tht_0]\cup\{1\}$ and positive on $(\tht_0,1)$, and $a\ge 1$ bounded above),
and Zlato\v s \cite{ZlaGenfronts} for more general ignition reactions 
the following.  There is a unique right-moving (and a unique left-moving) transition front solution and as $t\to\infty$, each front-like solution with $e=1$ ($e=-1$) converges in $L^\infty_x$ to its time-translate.  A similar result holds for spark-like solutions, when restricted to $\bbR^+$ ($\bbR^-$).  
Moreover, if $f$ is stationary ergodic, then \eqref{0.1}, \eqref{0.2} hold with some $c_e=s_e$ for $e=\pm 1$.

The {\it transition fronts} appearing here are a generalization of the concepts of traveling and pulsating fronts to disordered (non-periodic) media.  In the one-dimensional setting they are  entire solutions   
of \eqref{1.1} satisfying
\beq \lb{0.3}
\lim_{x\to\mp\infty} u(t,x)=1 \qquad\text{and}\qquad  \lim_{x\to\pm\infty} u(t,x)=0
\eeq
for each $t\in\bbR$ (with upper sign for right-moving fronts and lower sign for left-moving fronts), 
as well as $\sup_{t\in\bbR} L_{u,\eps}(t)$ for each $\eps\in(0,\tfrac 12)$, where $L_{u,\eps}(t)$ is the length of the shortest interval 
containing all $x\in\bbR$ with $u(t,x)\in[\eps,1-\eps]$.  This last property is called {\it bounded width} in \cite{ZlaGenfronts}.    The definition of transition fronts was first given in some specialized cases by Shen \cite{Shen} and Matano  \cite{Matano}, and then in a very general setting (including several dimensions) by Berestycki, Hamel in their fundamental papers \cite{BH2,BH3}.  Existence of  transition fronts in one-dimensional disordered media (but no long term asymptotics of general solutions) was also proved  for  bistable reactions which are small perturbations of homogeneous ones by Vakulenko, Volpert \cite{VakVol}, for KPP reactions which are (spatially) decaying  perturbations of homogeneous ones by Nolen, Roquejoffre, Ryzhik, Zlato\v s \cite{NRRZ}, for general KPP reactions by Zlato\v s \cite{ZlaInhomog}, and for general monostable reactions which are close to KPP reactions by Tao, Zhu, Zlato\v s \cite{TZZ}. 
We also mention results proving existence of a critical front, once some transition front exists,  by Shen \cite{Shen} and Nadin \cite{Nadin}.


While it is again not possible to improve the estimates on the length of the interval on which the transition between $u\sim 0$ and $u\sim 1$ is {\it guaranteed to happen}  (which again grows as $(c_1-c_0)t$ in time if $f_0\le f\le f_1$), bounded width of transition fronts and the convergence-to-fronts results in \cite{MNRR,ZlaGenfronts} show that for ignition reactions and typical solutions, the transition does occur within intervals whose lengths are uniformly bounded in time.  Moreover, this bound depends on some bounds on the reaction but neither on the reaction itself, nor on the initial condition.  In particular, this shows that after a uniform-in-$(f,u,t)$ scaling in space, each such solution becomes, in some sense, close to the {\it characteristic function of a time-dependent  spatial interval.}  Moreover, the convergence-to-fronts results can be  used to show that this interval grows in (equally scaled) time with speed within $[c_0,c_1]$. 

Experience from observation of natural processes modeled by \eqref{1.1} suggests that this picture should be also valid for  media in several spatial dimensions.  For instance, aerial footage of forest fires spreading through (spatially inhomogeneous) regions demonstrates variously curved but usually relatively narrow lines of fire separating burned ($u\sim 1$) and unburned ($u\sim 0$) areas.  However, results demonstrating such phenomena for typical solutions of \eqref{1.1} with {\it general inhomogeneous} reactions have not been previously obtained in dimensions $d\ge 2$.

It turns out that the multi-dimensional case is much more involved in this respect. 
The first issue is that it is not completely obvious how to extend the definition of bounded width of  solutions of \eqref{1.1}, \eqref{1.2} to the multi-dimensional setting, and some first instincts may lead to unsatisfactory results for general non-periodic media (see the discussion below).   
The extension we introduce here is motivated by the Berestycki-Hamel definition of transition fronts (which are entire solutions of \eqref{1.1}) in several dimensions \cite{BH2,BH3}.  However, there are a couple of differences, and
we discuss the relationship of the two concepts after stating our main results, at the end of the next section.

For solutions $u\in[0,1]$ of the Cauchy problem for \eqref{1.1}, our extension is as follows.  We let $\Omega_{u,\eps}(t):=\{x\in\bbR^d\,|\, u(t,x)\ge \eps\}$ be the $\eps$-super-level set of $u$ at time $t$ and define the {\it width of the transition zone} of $u$ from $\eps$ to $1-\eps$ (or to be more precise, from $[\eps,1-\eps)$ to $1-\eps$) to be
\beq \lb{1.3xx}
L_{u,\eps}(t) := \inf \left\{L>0 \,\big|\, \Omega_{u,\eps}(t)\subseteq B_L \left(\Omega_{u,1-\eps}(t)\right) \right\}
\eeq
for $\eps\in(0,\tfrac 12)$, with $B_r(A):=\bigcup_{x\in A} B_r(x)$ and $\inf\emptyset=\infty$. Notice that this is precisely the {\it Hausdorff distance} of the sets $\Omega_{u,\eps}(t)$ and $\Omega_{u,1-\eps}(t)$. We now say that $u$ has {\it bounded width} if
\beq \lb{1.4xx}
\lim_{t\to\infty} L_{u,\eps}(t) < \infty
\eeq
for each $\eps\in(0,\tfrac 12)$.  The limit is necessary in \eqref{1.4xx} because $\sup_{x} u(t,x)<1$ may hold for each $t$ (e.g., for spark-like solutions); it will be replaced by $\sup_{t\in\bbR}$ for entire solutions (see Definition~\ref{D.1.1}).
So by \eqref{1.4xx}, $u\in[0,1]$ has bounded width if and only if for any $0<\eps<\eps'<1$, super-level sets $\Omega_{u,\eps'}(t)\subseteq \Omega_{u,\eps}(t)$ have uniformly (in large time) bounded Hausdorff distance.  In particular, each of them is uniformly in time close to  $\Omega_{u,1/2}(t)$.

One may wonder why do we not treat the equilibria 0 and 1 in a symmetric fashion and include a similar definition involving the sub-level sets of $u$ as well. (We do so in \eqref{1.3a} and the related definition of {\it doubly-bounded width}, which is  for $u\in[0,1]$ equivalent to uniformly bounded Hausdorff distance of the {\it boundaries} $\partial\Omega_{u,\eps}(t)$ and $\partial\Omega_{u,\eps'}(t)$.)
While adding this requirement works well in one dimension \cite{MRS,MNRR,NolRyz,ZlaGenfronts}, it turns out to be too restrictive for the  treatment of sufficiently general (not necessarily periodic) reactions and solutions of the Cauchy problem \eqref{1.1}, \eqref{1.2} in dimensions $d\ge 2$.  This is due to $u\equiv 1$ being the invading equilibrium
and $u\equiv 0$ the invaded one, coupled with the possibility of arbitrary variations in the medium on arbitrarily large scales in two or more unbounded dimensions.  We discuss the issues involved after stating Theorem \ref{T.1.0} for reactions with a constant ignition temperature $\tht_0$, which is a special case of Theorem \ref{T.1.2} below.

Even with a suitable definition at hand, our results proving bounded widths of typical solutions, under quite general and physically natural {\it qualitative} hypotheses on the reaction, only hold in dimensions $d\le 3$.  Surprisingly, such results are in fact false for $d\ge 4$ without the addition of further {\it quantitative} hypotheses (e.g., $f$ being sufficiently close to a homogeneous reaction; 
see Remark 1 after Definition \ref{D.1.10} below).  The reason is that in $d\ge 4$, even in the constant ignition temperature case, intermediate values of $u$ may spread faster than values close to 1 (see the discussion after Definition \ref{D.1.1a} and Remark 1 after Theorem \ref{T.1.2} for the general ignition case).  This turns out to be related to the possibility of existence of non-constant stationary solutions $p\in(0,1)$ of \eqref{1.1} in $\bbR^{d-1}$
(see Section \ref{S8}).

\begin{theorem} \lb{T.1.0}
Let $f$ be Lipschitz (with constant $K$) and non-increasing in $u$ on $[1-\tht,1]$ for each $x\in\bbR^d$ (where $\tht>0$).   Assume that $f_0(u)\le f(x,u)\le f_1(u)$ for all $(x,u)\in \bbR^d\times[0,1]$, with $f_0,f_1:[0,1]\to[0,\infty)$ vanishing on $[0,\tht_0]\cup\{1\}$ and positive on $(\tht_0,1)$ (where $\tht_0>0$).  Let  $c_0$ and $c_1$ be the spreading speeds of $f_0$ and $f_1$.

(i) If $d\le 3$,
then the solution of \eqref{1.1}, \eqref{1.2} with any spark-like \eqref{1.5} or front-like \eqref{1.6} initial data $u_0\in[0,1]$ has bounded width \eqref{1.4xx}. In fact, for any $\eps\in(0,\tfrac 12)$ there are  $\ell_\eps, T_\eps$ such that $\sup_{t\ge T_\eps} L_{u,\eps}(t)\le \ell_\eps$, with $\ell_\eps$ depending only on $\eps,f_0,K$ ($T_\eps$ also depends on $u_0$).  Finally, $u$ propagates with global mean speed in $[c_0,c_1]$ in the sense of  Definition \ref{D.1.1b} below, with  $\tau_{\eps,\delta}$ in that definition depending only on  $\eps, f_0,K,\delta,f_1$.

(ii) If  $d\ge 4$, then there are $f,f_0,f_1$ as above such that no solution of \eqref{1.1}, \eqref{1.2} with compactly supported $u_0\in[0,1]$  
 and satisfying $\limsup_{t\to\infty}\|u(t,\cdot)\|_\infty>0$ has bounded width. 
\end{theorem}

{\it Remarks.} 1. Hence in dimensions $d\le 3$, each typical solution $u$ eventually becomes uniformly close (in the sense of Hausdorff distance of $\eps$-super-level sets for each $\eps\in(0,1)$)
 to the  characteristic function of  $\Omega_{u,1/2}(t)$, and the latter grows  with speed (averaged over long enough time intervals) essentially in $[c_0,c_1]$.
 So after a uniform-in-$(f,u,t)$ space-time scaling, typical solutions look like Figure~1, with the shaded area  expanding at speeds within $[c_0,c_1]$.  Since this also shows that an observer at any point $x\in\bbR^d$ at which $u(t,x)=\eps$ (for a large enough $t$) will see transition to the value $1-\eps$ within a uniformly bounded time interval,  this means that the {\it reaction zone} (where $u\sim \tfrac 12$)  is uniformly bounded in both space and time.
\smallskip

2. One can use \cite[Theorem 1.11]{BH3}  to prove (i) for homogeneous ignition reactions (see  \cite{Jones, Roussier} for bistable ones), and also for $x$-periodic ignition reactions and front-like solutions. 
However, 
besides disordered media, (i) is new for $x$-periodic media and spark-like solutions as well.  In fact, some of our results are new even for homogeneous media (e.g., Theorem \ref{T.1.3}).
\smallskip

3.  We will generalize Theorem \ref{T.1.0}
in several ways.  This will include a proof that solutions eventually increase in time on each interval of values $[\eps,1-\eps]$, extensions to other types of reactions (ignition with non-constant ignition temperature, monostable, bistable, and their mixtures) and to transitions between general equilibria $u^-<u^+$, as well as a treatment of more general types of solutions (trapped between time shifts of general time-increasing solutions, and not necessarily satisfying $u^-\le u\le u^+$).
\smallskip

The reasons for the new complications for $d\ge 2$, described above, are not just technical but stem from ``real world'' considerations in the case of {\it two or more unbounded dimensions.} (Note that the result in \cite{ZlaGenfronts} extends to the quasi-one-dimensional case of infinite cylinders in $\bbR^d$. The results described in Remark 2 after Theorem \ref{T.1.0} also have a (quasi-)one-dimensional nature due to either radial symmetry or periodicity.)  
%
%

First, one might think that the reaction zone will always coincide  with a bounded neighborhood of some time-dependent hypersurface.
This turns out to not be the case  in general in dimensions $d\ge 2$, since without some order in the medium (such as periodicity)  a fire will not always spread at roughly the same speed everywhere, so the initial spherical or hyperplanar shape of the reaction zone can become very distorted.  In fact, areas of slowly burning material in the medium may cause it to propagate {\it around them faster than through them,} resulting in pockets of temporarily unburned material behind  the leading edge of the fire.  
See  Figure \ref{fig} for an illustration of this phenomenon, and the proof of Theorem \ref{T.1.6}(ii) for an extreme example of it. 
While these pockets will eventually burn up, variations in the medium can create arbitrarily many or  even infinitely many of them (the latter for front-like solutions, although not spark-like) at a given (large) time, and they can be arbitrarily large and occur arbitrarily far behind the leading edge as $t\to\infty$.  As a result of this potentially complicated geometry of the reaction zones of general solutions of \eqref{1.1}, our definition of bounded width includes no requirements on the shape of the sets $\Omega_{u,\eps}(t)$ or their boundaries.

\begin{figure}[ht] 
\centering \scalebox{0.18}{\includegraphics{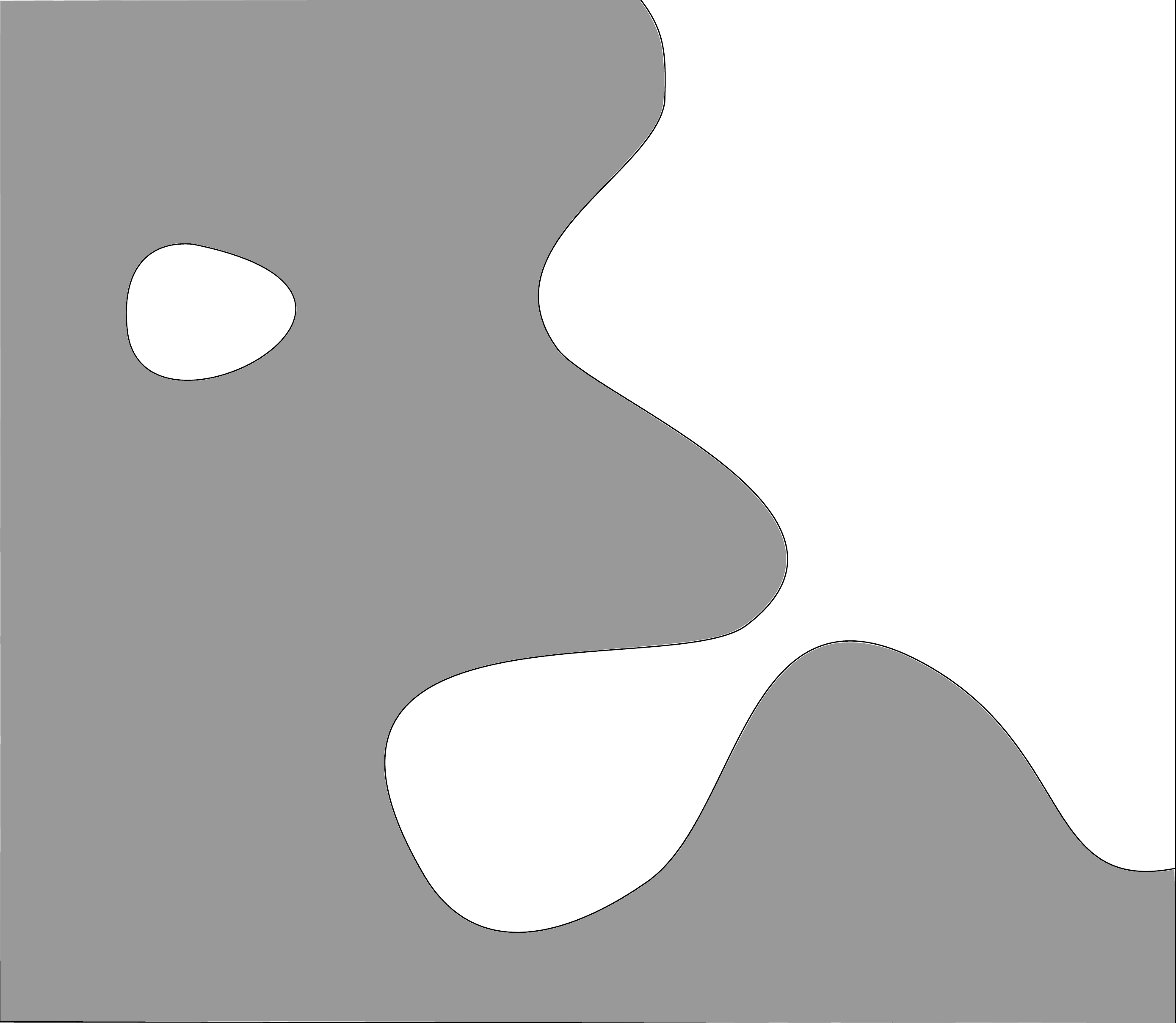}}
 \caption{On the shaded region $u\sim 1$, and on the white region $u\sim 0$.} 
 \label{fig} 
 \end{figure}
 
It is worth noting that while one might think that this issue can only arise if the medium has large variations in combustivity, this is not the case either.  In fact, as long as $f_0,f_1$  satisfy $c_0<c_1$, it is always possible to construct $f$ such that $f_0\le f\le f_1$ and the above situation (arbitrarily many unburned pockets which can be arbitrarily large and arbitrarily far behind the leading edge) does indeed occur for typical solutions $u$.  In particular, it happens almost surely for stationary ergodic media with short range correlations.

Another critical issue, related to this, arises from the consideration of what happens to such an unburned pocket far behind the leading edge of the fire.  It ``burns in'' from its perimeter and at the time  when it is just about to be burned up (say when the minimum of $u$ on it is $\tfrac 3{4}$), the nearest point where $u$ is close to 0 (say $\le \tfrac 1{4}$) may be very far from the pocket.  This shows that for general inhomogeneous media in dimensions $d\ge 2$, one may have {\it unbounded-in-time width of the transition zone from $u\sim 1$ to $u\sim 0$.} 

On the other hand, in the situation studied here when the invaded state $u= 0$ is either stable or relatively weakly unstable (the invading state $u=1$ clearly must be stable), pockets of burned material cannot form arbitrarily far ``ahead'' of 
the leading edge, unlike pockets of yet-unburned material ``behind'' the leading edge. (This is very different from the KPP case where $u=0$ is strongly unstable; see \cite{NRRZ} for examples of such media in one dimension, and the discussion after Definition \ref{D.1.1} for what may be done in that case.)   This means that typical solutions will be {\it pushed} (as opposed to {\it pulled}),  their propagation being driven by intermediate (rather than small) values of $u$. Thus one can still hope to see a {\it uniformly-in-time bounded width of the transition zone from $u\sim 0$ to $u\sim 1$}.  This lack of symmetry between the spatial  transitions $1\arrow 0$ and $0\arrow 1$ is why our definition of bounded width involves the Hausdorff distance of the super-level sets of $u$ but not of the sub-level sets (or of their boundaries).

Let us conclude this introduction with the discussion of  convergence of typical solutions of the Cauchy problem to entire solutions (such as transition fronts) of \eqref{1.1} in several dimensions.  In contrast to one dimension,  it is unlikely that any  general enough such results  can be obtained for disordered media.  Firstly, the disorder may result in reaction zones of solutions neither moving in a particular direction nor attaining a particular geometric shape.  Secondly, in Theorem \ref{T.1.6}(ii) we construct media where any entire solution with uniformly-in-time bounded width of the transition zone from $u\sim 0$ to $u\sim 1$ satisfies $\lim_{t\to\infty} \inf_{x} u(t,x)=1$ (while typical solutions of the Cauchy problem have $\inf_{x} u(t,x)=0$ for each $t\in\bbR$).

And thirdly, if such a result existed, one should also expect the following Liouville-type claim to hold:  If a solution $u$ is initially between two time translates of a front-like (or spark-like) solution $v$  (and so by the comparison principle, $v(\cdot,\cdot)\le u (T+\cdot,\cdot)\le v(2T+\cdot,\cdot)$ for some $T\ge 0$), then for any $\eps>0$ there is $T_\eps>0$ such that for any $(t,x)\in[T_\eps,\infty)\times \bbR^d$, 
\beq \lb{1.00}
\|u(t,\cdot)-v(t+\tau_{t,x},\cdot)\|_{L^\infty({B_{1/\eps}(x)})}<\eps
\eeq
for some $|\tau_{t,x}|\le T$.
That is,  $u$ should locally look more and more like a (possibly $(t,x)$-dependent) time translate of $v$ as time progresses. Somewhat surprisingly, this claim is false in general  in dimensions $d\ge 2$, even if $v$ is required to be an entire solution.  This is for non-pathological reasons and we discuss a counter-example in Section \ref{S7}.
 
Nevertheless, despite the likely lack of sufficiently general results on convergence to transition fronts or other entire solutions in general disordered media, these solutions will still play an important role in our analysis.  This is because one can use parabolic regularity to build entire solutions from those of the Cauchy problem sampled near any sequence of points $(t_n,x_n)$ with $t_n\to\infty$, so results for the former can be used in the analysis of the latter.

Finally, let us mention that our results can be extended to some more general PDEs, with $x$-dependent second order terms as well as first order terms with divergence-free coefficients.  This will be done elsewhere.  


\section{The Definition of Bounded Width and the Main Results} \lb{S1a}

Let us now turn to our main results. We will first assume that $u\in[0,1]$ and $f\ge 0$ is Lipschitz and bounded below by some homogeneous {\it pure ignition} reaction $f_0$.  (Later we will consider more general situations.) We will thus assume the following.

\medskip
{\it Hypothesis (H):  $f$ is Lipschitz with constant $K\ge 1$ and 
\beq\lb{0.ac}
f(x,0)=f(x,1)=0 \qquad \text{for $x\in\bbR^d$.}  
\eeq
There is also $\tht_0\in (0,1)$ and a Lipshitz function $f_0:[0,1]\to[0,\infty)$  with $f_0(u)=0$ for $u\in[0,\tht_0]\cup\{1\}$ and $f_0(u)>0$ for $u\in(\tht_0,1)$  such that 
\[
f(x,u)\ge f_0(u) \qquad \text{for $(x,u)\in \bbR^d\times [0,1]$}.
\]
Finally, there is $\tht\in[0,\tfrac 13]$ such that $f(x,u)=0$ for $(x,u)\in \bbR^d\times[0,\tht]$ and $f$ is non-increasing in $u$ on $[1-\tht,1]$ for each $x\in\bbR^d$.  If such $\tht>0$ exists, then $f$ is an {\it ignition reaction}, otherwise (in which case the last hypothesis is vacuous) $f$ is a {\it monostable reaction}.}
\medskip

{\it Remarks.}  1. The definition of ignition reactions sometimes also includes existence of $\til\tht(x)\in[\tht,\tht_0]$ such that $f(x,u)>0$ if and only if $u\in(\til\tht(x),1)$, which we call the {\it pure ignition} case.  We will not need this stronger hypothesis here.  
\smallskip

2. While the requirement of $f$ being non-increasing in $u$ on $[1-\tht,1]$ is not always included in the definition of ignition reactions, many results for them need to assume it. This includes our main results, although the hypothesis is not needed for their slightly weaker versions (specifically, not including those statements which use Theorem \ref{T.1.5}(ii) below). 
Notice also that we can assume without loss that $f_0$ is non-increasing on $[1-\delta,1]$ for some $\delta>0$ because this can be achieved after replacing $f_0(u)$ by $\min_{v\in[1-\delta,u]} f_0(v)$.  Thus $f_0$ is itself an ignition reaction according to the above definition.
\smallskip

For a set $A\subseteq\bbR^d$ and $r>0$, we let $B_r(A):=\bigcup_{x\in A} B_r(x)$ (for $r\le 0$ we define $B_r(A):=\emptyset$).
If  $u:(t_0,\infty)\times\bbR^d\to[0,1]$ and  $\eps\in[0,1]$, we let $\Omega_{u,\eps}(t):=\{x\in\bbR^d\,|\, u(t,x)\ge \eps\}$ for $t>t_0$. 
For $\eps\in(0,\tfrac 12)$, the {\it width of the transition zone} of $u$ from $\eps$ to $1-\eps$ at time $t>t_0$ is
\beq \lb{1.3}
L_{u,\eps}(t) := \inf \left\{L>0 \,\big|\, \Omega_{u,\eps}(t)\subseteq B_L \left(\Omega_{u,1-\eps}(t)\right) \right\},
\eeq
with the usual convention $\inf\emptyset=\infty$.  
For $\eps\in(\tfrac 12,1)$ the corresponding width is
\beq \lb{1.3a}
L_{u,\eps}(t) := \inf \left\{L>0 \,\big|\, \bbR^d\setminus 
\Omega_{u,\eps}(t) \subseteq B_L \left(\bbR^d\setminus
\Omega_{u,1-\eps}(t) \right) \right\}.
\eeq
Finally,  for $\eps\in(0,\tfrac 12)$ we also define the minimal length of transition from $(\eps,1-\eps)$ to either $\eps$ or $1-\eps$ to be
\beq \lb{1.3aa}
J_{u,\eps}(t) := \inf \left\{L>0 \,\big|\, \bbR^d = B_L \left( \Omega_{u,1-\eps}(t) \cup \left[\bbR^d\setminus
\Omega_{u,\eps}(t) \right] \right) \right\} .
\eeq
For the above to be perfectly symmetric, we could replace $\bbR^d\setminus\Omega_{u,\eps}(t)$ by $\bbR^d\setminus\bigcup_{\eps'>1-\eps}\Omega_{u,\eps'}(t)$, but as we mentioned in the introduction,  \eqref{1.3a} and \eqref{1.3aa} will not play a major role here.


\begin{definition} \lb{D.1.1}
Let $u:(t_0,\infty)\times\bbR^d\to[0,1]$ be a solution of \eqref{1.1} with $t_0\in[-\infty,\infty)$.  We say that $u$ has a {\it bounded width (with respect to 0 and 1)} if  for any $\eps\in(0,\tfrac 12)$ we have
\beq \lb{1.4}
L^{u,\eps} := \lim_{T\to\infty} \sup_{t>t_0+T} L_{u,\eps}(t) < \infty.
\eeq
We say that $u$ has a {\it doubly-bounded width} if \eqref{1.4} holds for any $\eps\in(0,\tfrac 12)\cup (\tfrac 12,1)$.
And we say that $u$ has a {\it semi-bounded width}  if  for any $\eps\in(0,\tfrac 12)$  we have
\beq \lb{1.4f}
J^{u,\eps} := \lim_{T\to\infty} \sup_{t>t_0+T} J_{u,\eps}(t) < \infty.
\eeq

\end{definition}

{\it Remarks.} 1. Notice that if $t_0=-\infty$, then $L^{u,\eps} =  \sup_{t\in\bbR} L_{u,\eps}(t)$ and $J^{u,\eps} =  \sup_{t\in\bbR} J_{u,\eps}(t)$.  For $t_0>-\infty$, however, these quantities are defined only asymptotically.  One reason for this is that if $\sup_{x\in\bbR^d} u_0(x)<1$, then $\sup_{x\in\bbR^d} u(t,x)<1$ for any $t>t_0$.  Thus for any $\eps\in(0,\tfrac 12)$, $L_{u,\eps}(t)$ will equal $\infty$ up to some time $t_\eps$ ($\to\infty$ as $\eps\to 0$).
\smallskip

2.  We trivially have that $L^{u,\eps}$ is non-increasing in  $\eps\in(0,\tfrac 12)$ (as is $J^{u,\eps}$) and non-decreasing in $\eps\in(\tfrac 12,1)$, so in fact the definition  only needs to involve $\eps$ close to 0 and 1.
\smallskip



While the definition of bounded width involves $\eps\in(0,\tfrac 12)$, we do not make one involving only $\eps\in(\tfrac 12,1)$.  This lack of symmetry was explained in the introduction, 
and is due to the possibility of existence of unburned pockets with $u\sim 0$ behind the leading edge of the reaction zone.  
Hence, typical solutions $u$ in general disordered media may have $L^{u,\eps}=\infty$  for  $\eps\in(\tfrac 12,1)$.  In particular, they would not have doubly-bounded widths, but may still have bounded widths, at least when $u\equiv 0$ is a stable equilibrium.

If the equilibrium $u\equiv 0$ is strongly unstable (such as for KPP $f$),  bounded width is also too much to hope for in some situations, even when $d=1$.  Indeed, an easy extension of the construction from \cite{NRRZ} yields media where burned pockets with $u\sim 1$ can form arbitrarily far ahead of the leading edge of the reaction zone.  While we do not study this case here,  we introduce the concept of semi-bounded width in Definition \ref{D.1.1} because it is likely to be relevant in such situations.


We next define the propagation speed of (the  reaction zone of) $u$ (cf. \cite{BH3}).

\begin{definition} \lb{D.1.1b}
Let  $u:(t_0,\infty)\times\bbR^d\to[0,1]$ be a solution of \eqref{1.1} with $t_0\in[-\infty,\infty)$,
and let $0<c\le c'\le\infty$.
We say that {\it $u$  propagates with global mean speed in $[c,c']$}  if for any $\eps\in(0,\tfrac 12)$ and $\delta>0$ there are $T_{\eps,\delta},\tau_{\eps,\delta}<\infty$ such that 
\beq\lb{1.3g}
B_{(c-\delta)\tau} \left(\Omega_{u,\eps}(t) \right)  \subseteq \Omega_{u,1-\eps}(t+\tau)
\qquad\text{and}\qquad
\Omega_{u,\eps}(t+\tau) \subseteq B_{(c'+\delta)\tau} \left(\Omega_{u,1-\eps}(t) \right)
\eeq
whenever $t> t_0+T_{\eps,\delta}$ and $\tau\ge \tau_{\eps,\delta}$.  If any such $0<c\le c'\le\infty$ exist, we also say that {\it $u$  propagates with a positive global mean speed.}
\end{definition}

{\it Remarks.}  1. If $t_0=-\infty$, then obviously $t\in\bbR$ above is arbitrary.  
\smallskip

2. Notice that the definition would be unchanged if we took $T_{\eps,\delta}=\tau_{\eps,\delta}$.  However, this formulation will be more convenient for us because we will show that under certain conditions, $\tau_{\eps,\delta}$ (but not necessarily $T_{\eps,\delta}$) will be independent of $f,u$.
\smallskip

We now let $c_0$ be the front/spreading speed associated with the homogeneous reaction $f_0$. That is, $c_0$ is the {\it unique} value such that \eqref{1.1} for $d=1$ and with $f_0(u)$ in place of $f(x,u)$ has a  traveling front solution $u(t,x)=U(x-c_0t)$ with $\lim_{s\to -\infty}U(s)=1$ and $\lim_{s\to \infty}U(s)=0$.  

We also   let $f_1:[0,1]\to[0,\infty)$ be any Lipschitz function with constant $K$ such that
\beq \lb{1.4e}
f(x,u)\le f_1(u) \qquad \text{for $(x,u)\in \bbR^d\times [0,1]$},
\eeq
which is also pure ignition if $\tht>0$ in (H) and pure monostable (i.e. $f(0)=f(1)=0$ and $f(u)>0$ for $u\in(0,1)$) otherwise.
For instance, we could pick $f_1(u):=\sup_{x\in\bbR^d} f(x,u)$, if this function is pure ignition/monostable.  We also let $c_1$ be the front/spreading speed associated with $f_1$ (which is again the unique traveling front speed if $f_1$ is ignition, and it is the {\it minimal} traveling front speed if $f_1$ is monostable).  
The existence of $c_0, c_1$ is well known, as well as that $f_0(u)\le f_1(u)\le Ku$ implies $c_0\le c_1\le 2\sqrt K$.  


Our main results below say that under appropriate (quite general and physically relevant) qualitative hypotheses on the reaction,  typical solutions of \eqref{1.1} have bounded widths and (their reaction zones) propagate with global mean speeds in the interval $[c_0,c_1]$.  They also eventually grow in time on any closed interval of values of $u$ contained in $(0,1)$.  Specifically,  we will prove the following conclusion for typical solutions $u$.

\medskip
{\it Conclusion (C):  For any $\eps\in(0,\tfrac 12)$, there are $\ell_\eps,m_\eps, T_\eps\in(0,\infty)$  such that  
\beq \lb{1.7}
\sup_{t> t_0+T_\eps} L_{u,\eps}(t)\le \ell_\eps \qquad\text{and}\qquad
\inf_{\substack{(t,x)\in(t_0+T_\eps,\infty) \times \bbR^d \\ u(t,x)\in[\eps,1-\eps]}} u_t(t,x) \ge m_\eps.
\eeq
In particular, $L^{u,\eps}\le \ell_\eps$, so $u$ has a bounded width.
Moreover, if a pure ignition $f_1$
 satisfies \eqref{1.4e}, then $u$ propagates with global mean speed in $[c_0,c_1]$.}
 \medskip

Moreover, $\ell_\eps,m_\eps$ as well as $\tau_{\eps,\delta}$ from Definition \ref{D.1.1b} will depend on some uniform bounds on the reaction, but neither on the reaction itself nor on the solution.  That is, the spatial scale on which the transition from $u\sim 0$ to $u\sim 1$ happens as well as the temporal scale on which the global mean speed of (the reaction zone of) $u$ is observed to be in $[c_0,c_1]$, will become independent of $f,u$ after an initial time interval.

Note that such expanding sets may also be weak solutions of appropriate Hamilton-Jacobi equations.
Connection of the two types of PDE is well-established in the homogenization theory for various types of media (e.g., periodic or  stationary ergodic), see for instance \cite{Freidlin, LioSou}.  It will be explored, via our results, for general disordered media elsewhere.

\bigskip
\noindent
{\bf Solutions of the Cauchy Problem with Bounded Widths}
\smallskip
\smallskip

We will first show that  for ignition reactions, (C) holds in dimensions $d\le 3$, but not in dimensions $d\ge 4$ (under the same qualitative hypotheses).

A crucial additional (and necessary)  hypothesis, which is automatically satisfied in the case of constant ignition temperature $\tht_0$, relates to the following definition (see Remarks 1 and 2 below).  It says that if for any  $x\in\bbR^d$ we increase  $u$ from 0 to 1, once $f(x,u)$ becomes large enough, it cannot become arbitrarily small until $u\sim 1$, as illustrated in Figure \ref{fig2}.

\begin{figure}[ht] 
\centering \scalebox{0.29}{\includegraphics{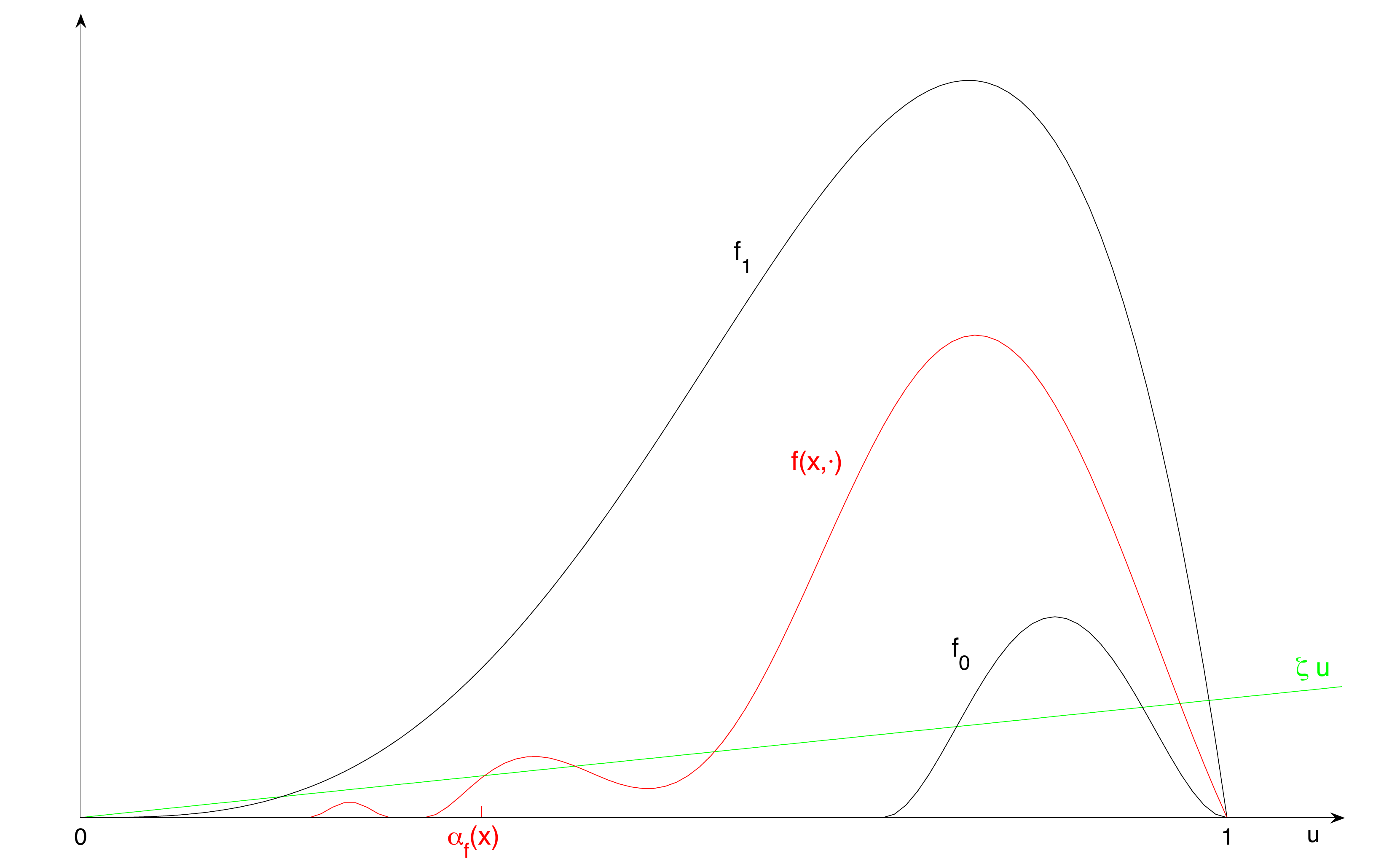}}
 \caption{Example of a reaction from Definition \ref{D.1.1a} (at some fixed $x\in\bbR^d$).}
 \label{fig2} 
 \end{figure}

\begin{definition} \lb{D.1.1a}
Let $f_0,K,\tht$ be as in (H) and let $\zeta, \eta>0$.  If $f$ satisfies (H), define
\beq \lb{1.4a}
\alpha_f(x) = \alpha_f(x;\zeta):= \inf \{ u\ge 0 \,|\, f(x,u)> \zeta u \},
\eeq
(with $\inf\emptyset=\infty$) and let $F(f_0,K,\tht,\zeta,\eta)$ be the set of all $f$ satisfying (H) such that
\beq \lb{1.4b}
 \inf_{\substack{x\in\bbR^d \\  u\in[\alpha_f(x),\tht_0]}} f(x,u)  \ge \eta.
\eeq
\end{definition}

{\it Remarks.} 1.  We will require that $f\in F(f_0,K,\tht,\zeta,\eta)$ for some {\it not too large} $\zeta>0$ and some $\eta>0$.  This assumption is physically relevant and encompasses a large class of functions.   A natural example is the pure ignition reaction from Remark 1 after (H),
when also
\beq \lb{1.4c}
 \til\eta:=\inf_{\substack{x\in\bbR^d \\  u\in[\til\tht(x)+\delta,\tht_0]}} f(x,u) >0
\eeq 
for some $\delta>0$  (in that case $f\in F(f_0,K,\tht,\zeta,\eta)$ for any $\zeta\ge \tfrac K\tht \delta$ and $\eta\in(0,\til\eta]$).  
\smallskip

2. Notice that this definition is not necessary when $f$ has a constant ignition temperature:  if $f(x,u)= 0$ for $(x,u)\in \bbR^d\times[0,\tht_0]$, then $f$ from (H) is in $F(f_0,K,\tht,\zeta,\eta)$ for any $\zeta,\eta>0$.
This is the case in Theorem \ref{T.1.0} (the special case $f(x,u)=a(x)f_0(u)$ was also considered in \cite{MRS,NolRyz,MNRR} in one spatial dimension).
\smallskip

3. Note that $F(f_0,K,\tht,\zeta,\eta)$ is spatially translation invariant and closed under locally uniform convergence of functions.  It is also decreasing in its odd arguments and increasing in the even ones.  In particular, $F(f_0,K,\tht,\zeta,\eta)\subseteq F(f_0,K,0,\zeta,\eta)$.   These facts will be useful later, as well as the obvious $\alpha_f(x)\ge \tfrac \eta K$ for $f\in F(f_0,K,\tht,\zeta,\eta)$. 
\smallskip

Without (some version of) the assumption from Remark 1, solutions of \eqref{1.1} need not have bounded widths even when $d=1$ and $f$ is a homogeneous ignition reaction!  Indeed, assume that $f:[0,1]\to[0,\infty)$ is such that $f(u)=0$ for $u\in[0,\tfrac 14]$, $f(u)>0$ for $u\in(\tfrac 14,\tfrac 12)$, and $f(u)= 2f(u-\tfrac 12)$ for $u\in [\tfrac 12,1]$.  Such $f$ vanishes on $[\tfrac 12,\tfrac 34]$ and so belongs to $F(f_0,K,\tht,\zeta,\eta)$ {\it only for large $\zeta$} (specifically $\zeta\ge\|f(u)/u\|_\infty$).  

For such $f$, there obviously is a traveling front solution  $u(t,x)=U(x-ct)$ of \eqref{1.1} connecting 0 and $\tfrac 12$ (i.e., such that $\lim_{s\to-\infty}U(s)=\tfrac 12$ and $\lim_{s\to\infty}U(s)=0$) and another $u(t,x)=\tfrac 12 + U(\sqrt 2(x-\sqrt 2ct))$ connecting $\tfrac 12$ and 1.  Their speeds are $0<c<\sqrt 2c$ and a simple comparison principle argument shows that all spark-like and front-like solutions have a linearly in time growing {\it propagating terrace}:
\beq \lb{1.4d}
\lim_{t\to\infty} \sup_{x\in [(c+\del)t, (\sqrt 2c-\del)t]} \left| u(x,t)-\frac 12 \right| = 0
\eeq
for any $\del>0$ (see \cite{DGM} for further results of this nature).  In particular, they do not have bounded widths.
Of course, for such solutions one can separately study the transition from 0 to $\tfrac 12$ and that from $\tfrac 12$ to 1, using our results.  Hence, the latter can also be applied in some situations when \eqref{1.4b} is not satisfied for any $\eta>0$ (and some not too large $\zeta>0$).

We are now ready to state our first main result, which applies to general {\it spark-like} and {\it front-like} initial data $u_0\in[0,1]$.  Specifically, we will assume that
either there are $x_0\in\bbR^d$,  $R_2\ge R_1>0$, and $\eps_1,\eps_2>0$ such that
\beq \lb{1.5}
(\tht_0+\eps_1) \chi_{B_{R_1}(x_0)}(x) \le u_0(x)\le e^{-\eps_2(|x-x_0|-R_2)},
\eeq
or there are $e\in\bbS^{n-1}$, $R_2\ge R_1$, and $\eps_1,\eps_2>0$ such that
\beq \lb{1.6}
(\tht_0+\eps_1) \chi_{\{x\,|\,x\cdot e<R_1\}}(x) \le u_0(x)\le e^{-\eps_2(x\cdot e-R_2)}.
\eeq
In \eqref{1.5} we also  assume that $R_1$ is large enough (depending on $\eps_1$) to guarantee {\it spreading} (i.e., $\lim_{t\to\infty}u(x,t)= 1$ locally uniformly in $\bbR^d$), because otherwise one might have {\it quenching} (i.e., $\lim_{t\to\infty} \|u(t,\cdot)\|_\infty =0$) for ignition reactions.

\begin{theorem} \lb{T.1.2}
(i) Let $f_0,K,$ and $\tht>0$ be as in (H) and let $\eta>0$, $\zeta\in(0,c_0^2/4)$, and \hbox{$f\in F(f_0,K,\tht,\zeta,\eta)$.}  Let $u$ solve \eqref{1.1}, \eqref{1.2} with spark-like or front-like $u_0\in[0,1]$ as above.
If $d\le 3$, then (C) holds with $\ell_\eps,m_\eps$ depending only on $\eps,f_0,K,\zeta,\eta$, and  $\tau_{\eps,\delta}$ in Definition~\ref{D.1.1b} also depending on  $\delta,f_1$.
%

(ii) If  $d\ge 4$, then there is $f$ as in (H) with $\tht>0$ and $f(x,u)= 0$ for $(x,u)\in \bbR^d\times[0,\tht_0]$ (so that $f\in F(f_0,K,\tht,\zeta,\eta)$ for any $\zeta,\eta>0$) such that all claims in (C) are false for any  $u_0\in[0,1]$ supported 
in the left half-space  for which $\limsup_{t\to\infty}\|u(t,\cdot)\|_\infty>0$. 
\end{theorem}

{\it Remarks.}  1.  As noted before, the hypothesis $\zeta<c_0^2/4$ is crucial in (i).  It guarantees that the reaction at small $u$ (where $f(x,u)\le \zeta u$) is not strong enough to cause spreading at speeds $\ge c_0$.  This is because spreading speeds for homogeneous  reactions bounded above by $\zeta u$ are no more than $2\sqrt\zeta<c_0$.  Since $f\ge f_0$ has spreading speed no less than $c_0$, one should then expect spreading to be driven by ``intermediate'' values of $u$ (above $\alpha_f(x)$ and not too close to 1, where $f$ is small).  Thus $u$ would be  a ``pushed'' solution, and one can hope for  it to have a bounded width, provided one can also show that values of $u$ close to 1 do not ``trail'' far behind the intermediate ones.  We will prove the latter for $d\le 3$ but also show in (ii) that it fails in general for $d\ge 4$.
\smallskip


2. Note that the second claim in \eqref{1.7} and parabolic regularity shows that $\Omega_{u,\eps}(t)$ grows with {\it instantaneous speed} greater than some  positive constant at all times $t\ge t_0+T_\eps$ in (i).  An upper bound on the instantaneous speed of growth does not exist in general, however, because for $\eps\in(0,\tfrac 12)$, $\Omega_{u,\eps}(t)$ may acquire new connected components (which then soon merge with the ``main'' component) as time progresses.
\smallskip

3.  As the proof of (i) shows, $T_\eps$ in (C) depends on $\eps,f_0,K,\zeta,\eta,\tht,R_2-R_1,\eps_1,\eps_2$, and  $T_{\eps,\delta}$ in Definition \ref{D.1.1b} also depends on  $\delta, f_1$.  
\smallskip

4. The result extends to monostable reactions in a weaker form.  (ii) holds without change (the counter-example we construct is easily modified) but in (i) we need to assume that either there are $R_1,R_2,\eps_1>0$ ($R_1$ sufficiently large, depending on $\eps_1$) and $x_0\in\bbR^d$ such that
\beq \lb{1.5a}
(\tht_0+\eps_1) \chi_{B_{R_1}(x_0)}(x) \le u_0(x)\le \chi_{B_{R_2}(x_0)}(x),
\eeq
or there are $R_1,R_2\in\bbR$, $\eps_1,\eps_2>0$, and  $e\in\bbS^{n-1}$ such that
\beq \lb{1.6a}
(\tht_0+\eps_1) \chi_{\{x\,|\,x\cdot e<R_1\}}(x) \le u_0(x)\le (1-\eps_2) \chi_{\{x\,|\,x\cdot e<R_2\}}(x).
\eeq
Then for any $\eps\in(0,\tfrac 12)$ there are $\ell_\eps,T_\eps\in(0,\infty)$, depending  on $\eps,f_0,K,\zeta,\eta,\eps_1$, and either on $R_2$ (for \eqref{1.5a}) or on $R_2-R_1,\eps_2$ (for \eqref{1.6a}), such that  $L_{u,\eps}(t)\le \ell_\eps$ for $t> t_0+T_\eps$. 
\smallskip

%

%


The first step in the proof of Theorem \ref{T.1.2}(i) will be to consider general solutions with $u_t\ge 0$.  
That is, such that on $\bbR^d$,
\beq \lb{3.6a}
\Delta u_0(\cdot) + f(\cdot,u_0(\cdot))\ge 0,
\eeq
which then guarantees $u_t\ge 0$ because  $v:=u_t$ solves $v_t=\Delta v + f_u(x,u(x))v$ with $v(0,x)\ge 0$. 
For $d\le 3$ we will show  that if the width of the reaction zone of such $u$ is controlled at the initial time $t_0$ (see \eqref{1.6b} below), then the conclusions of Theorem \ref{T.1.2}(i) continue to hold.  This step is  related to our proof of existence of transition fronts in \cite{ZlaGenfronts}, but will be considerably more involved, particularly for $d=3$.  

This latter result applies to any such solution $u$ (as well as solutions trapped between time-shifts of such $u$), not just the spark-like or front-like ones, and is stated next.  We let
\beq \lb{1.6b}
L_{u,\eps,\eps'}(t) := \inf \left\{L>0 \,\big|\, \Omega_{u,\eps}(t)\subseteq B_L \left(\Omega_{u,\eps'}(t)\right) \right\}
\eeq
be the width of the transition zone from $\eps$ to $\eps'$.  We will assume that $L_{u,\eps,\eps'}(t_0)<\infty$ for each $\eps>0$ and some fixed $\eps'>\tht_0$.  Here $\eps'$ can be arbitrary when $d\le 2$, and equals $1-\eps_0$ when $d=3$ (with $\eps_0=\eps_0(f_0,K)>0$ from Lemma \ref{L.2.1} below).  This choice of $\eps'$ will guarantee spreading for any solution satisfying \eqref{3.6a} and $u(t_0,x)\ge \eps'$ for some $x\in\bbR^{d}$.

\begin{theorem} \lb{T.1.3}
Let $d\le 3$, let $f_0,K,$ and $\tht>0$ be as in (H), and let $\eta>0$, $\zeta\in(0,c_0^2/4)$, and $f\in F(f_0,K,\tht,\zeta,\eta)$.  Let $u$ solve \eqref{1.1}, \eqref{1.2} with $u_0\in[0,1]$ satisfying \eqref{3.6a}.

(i) If $\eps'$ is as above and  $L_{u,\eps,\eps'}(t_0)<\infty$  for each  $\eps>0$,  then (C) holds with $\ell_\eps,m_\eps$ depending only on $\eps,f_0,K,\zeta,\eta$, and $\tau_{\eps,\delta}$ in Definition \ref{D.1.1b} also depending on  $\delta,f_1$.

(ii) If $u$ is as in (i), and a solution $v$ of \eqref{1.1} satisfies 
\[
u(t_0,\cdot)\le v(t_0+\tau,\cdot)\le u(t_0+2\tau,\cdot)
\]
 for some $\tau>0$, then (C) holds for $v$ with $\ell_\eps,m_\eps,\tau_{\eps,\delta}$ as in (i)
 (so independent of $\tau$).
\end{theorem}

{\it Remarks.}  
1.  In (i), $T_\eps$  in (C) depends on $\eps,f_0,K,\zeta,\eta,\tht, u_0$, the dependence on $u_0$ being only via the number  $L_{u,h,\eps'}(t_0)$ with $h:=\min\left\{ \tht(c_0^2-4\zeta)(c_0^2+4\zeta)^{-1},\frac\eta{4K},1-\eps',\tfrac \eps 2 \right\}$ (see the proof);  $T_{\eps,\delta}$ in Definition \ref{D.1.1b} also depends on  $\delta, f_1$.   In (ii) they also depend on $\tau$.

2.  (i) extends to monostable $f$ if we also assume $\sup_{\eps\in(0,1)} \eps e^{\sqrt\zeta L_{u,\eps,\eps'}(t_0)}< \infty$, but with  global mean speed  in $[c_0,c_Y']$, where $c_Y'$ is from \eqref{3.0a} below.
\smallskip

3. (ii) also extends to monostable $f$ if we assume $\sup_{\eps\in(0,1)} \eps e^{\sqrt\zeta L_{u,\eps,\eps'}(t_0)}< \infty$, but with $\tau$-dependent $\ell_\eps,\tau_{\eps,\delta}$, without the second claim in \eqref{1.7}, and with global mean speed in $[c_0,c_Y']$.\smallskip

Notice that in (ii), the bounds in (C) are independent of the time shift $\tau$.  To prove this, we will first need to show such solution-independent bounds for entire  solutions with bounded widths when $d\le 3$.  In particular, as long as such a solution has a bounded width, the bound on $\sup_{t\in\bbR} L_{u,\eps}(t)$ (for $\eps\in(0,\tfrac 12)$) will in fact {\it only depend on $\eps, f_0,K,\zeta,\eta$.}  It will then suffice to show, using parabolic regularity, that the solutions from (ii) asymptotically look  like entire solutions with bounded widths, where the bounds involved will be allowed to depend on $\tau$.

A crucial ingredient in this will be the proof that entire solutions with bounded widths satisfy $u_t\ge 0$ (in all dimensions).  Such a result was previously  proved in \cite{BH3} for transition fronts in a closely related setting.  This and the uniform bounds for entire solutions with bounded widths are stated in Theorem \ref{T.1.5} below.

Theorem \ref{T.1.2}(i) is proved similarly to Theorem \ref{T.1.3}(ii), but the solution will be sandwiched between time-shifts of a time-increasing solution, perturbed by certain exponentially in space decreasing functions.  We will therefore also need to prove stability of spark-like and front-like time-increasing solutions with respect to such perturbations.  This could be extended to other situations where time-increasing solutions with some specific profiles are stable with respect to appropriate (exponentially decreasing) perturbations.  For instance, one could handle in this way {\it cone-like} solutions, with initial data exponentially decreasing inside a $d$-dimensional cone and converging to 1 outside it.  We will not pursue this direction here.

We also note that these results cannot be extended to arbitrary spreading solutions, even for homogeneous pure ignition reactions $f(x,u)=f_0(u)$ and $d=1$.  Indeed, the author showed \cite{ZlaSharp} that then there exists a unique $M>0$ such that the solution of \eqref{1.1}, \eqref{1.2} with $u_0:=\chi_{[-M,M]}$ converges locally uniformly to $\tht_0$ as $t\to\infty$.  If we now let $R\gg 1$ and  $u_0:=\chi_{[-R,R]}+\sum_{n=1}^\infty \chi_{[a_n-M,a_n+M]}$ with sufficiently rapidly  growing  $a_n$, the solution $u$ will have increasingly long plateaus as $t\to\infty$.  Specifically, there will be $t_n,b_n\to\infty$ such that
\[
\lim_{n\to\infty} \sup_{x\in[a_n-b_n,a_n+b_n]} |u(t_n,x)-\tht_0|=0.
\]
Such $u$ therefore does not even have a {\it semi-bounded} width!

Finally,  most of the argument for $d\le 3$ also applies  if $d\ge 4$,  the one exception being Lemma~\ref{L.3.2} below.  The reason it fails for $d\ge 4$ lies in Lemma \ref{L.2.2}, which only excludes existence of equilibrium solutions to \eqref{1.1} which are independent of one coordinate when $d\le 3$.  Such solutions will be the basis of the counter-example proving Theorem \ref{T.1.2}(ii).

\bigskip
\noindent
{\bf Extensions to More General Reactions, Equilibria, and Solutions}
\smallskip
\smallskip

Let us now discuss the more general case when typical solutions transition from some equilibrium $u^-$ to another equilibrium $u^+$ (instead from 0 to 1), with $u^-<u^+$ and 
\beq\lb{1.19}
0<\inf_{x\in\bbR^d} [u^+(x)-u^-(x)] \le \sup_{x\in\bbR^d} [u^+(x)-u^-(x)]<\infty
\eeq
(the case $u^->u^+$ is identical, as one can consider the equation for $-u$ instead).  Our goal is to extend the positive results in Theorems \ref{T.1.2}(i) and \ref{T.1.3} to such situations.

We will assume $u^-\equiv 0$ without loss, because the general case is immediately reduced to this by taking $v:=u-u^-$, which solves \eqref{1.1} with  $f$ replaced by 
\[
g(x,v):=f(x,v+u^-(x))-f(x,u^-(x)).
\]  
Obviously, we can also assume $u^+\le 1$, by \eqref{1.19} and after scaling in $u$.

Thus we will now assume the following generalization of (H).

\medskip
{\it Hypothesis (H'):  $f$ is Lipschitz with constant $K\ge 1$ and 
\[
f(x,0)=0 \qquad \text{for $x\in\bbR^d$.}  
\]
There are also $0<\tht_0< \tht_1\le 1$ and Lipshitz  $f_0:[0,\tht_1]\to\bbR$  with $f_0(0)=f_0(\tht_0)=f_0(\tht_1)=0$, $f_0(u)< 0$ for $u\in(0,\tht_0)$, and $f_0(u)>0$ for $u\in(\tht_0,\tht_1)$,  such that $\int_0^{\tht_1} f_0(u)du>0$ and
\[
f(x,u)\ge f_0(u) \qquad \text{for $(x,u)\in \bbR^d\times [0,\tht_1]$}.
\]
Furthermore, we assume that there is an equilibrium solution $u^+$  of \eqref{1.1} with 
\beq \lb{1.20}
\tht_0<\inf_{x\in\bbR^d} u^+(x)\le\sup_{x\in\bbR^d} u^+(x)\le 1,
\eeq
and we have
\beq\lb{1.24}
\text{$f(x,u)\ge 0$ when $u< 0$} \qquad \text{and} \qquad \text{$f(x,u)\le f(x,u^+(x))$ when $u> u^+(x)$}
\eeq
Finally, there is $\tht\in[0,\tfrac {\tht_0}3]$ such that $f$ is non-increasing in $u$ on $[0,\tht]$ and on $[u^+(x)-\tht,u^+(x)]$ for each $x\in\bbR^d$ ($\tht=0$ obviously always works but we will obtain stronger results when $\tht>0$).}  
\medskip

That is, $f_0$ is now a {\it pure bistable reaction} (while $f_1$ in \eqref{1.4e} will still be pure ignition or pure monostable), so $f$ could be any mix of different reaction types.  The hypothesis $\int_0^{\tht_1} f_0(u)du>0$ is necessary for solutions of \eqref{1.1}, \eqref{1.2}, with reaction $f_0$ and large enough $u_0\in[0,\tht_1]$, to spread (i.e.,  $\lim_{t\to\infty} u(t,x)= \tht_1$ locally uniformly).  In fact, it guarantees that the front/spreading speed $c_0$ for this $f_0$ (which corresponds to the traveling front for $f_0$ connecting $0$ and $\tht_1$, and is unique just as for ignition reactions) is positive.
Thus, typical non-negative solutions of \eqref{1.1} transition {\it away from} $u=0$.  Transition {\it to} $u^+$ is, however, not guaranteed by (H') only.  Finally, \eqref{1.24} will be needed in Theorem \ref{T.1.12} to extend our results to solutions which are not necessarily between 0 and $u^+$.

We next need to generalize Definitions \ref{D.1.1}--\ref{D.1.1a} to the case at hand.  We will first consider solutions $0\le u \le u^+$ (henceforth denoted $u\in[0,u^+]$), when \eqref{1.24} is of no consequece.
Definitions \ref{D.1.1} and \ref{D.1.1b} are unchanged for such $u$, but  use (for $\eps\in (0,\tfrac{1}2)$)
\begin{align}
 \Omega_{u,\eps}(t) & :=\{x\in\bbR^d\,|\, u(t,x) \ge \eps \}, 
 \lb{1.20a}
\\  \Omega_{u,1-\eps}(t) & :=\{x\in\bbR^d\,|\, u(t,x) \ge u^+(x)-\eps \}. 
\lb{1.20b}
\end{align}

Definition \ref{D.1.1a}, on the other hand, needs to be changed because $f(x,u^+(x))\not\ge 0$ in general.  The motivation for this new form comes from the proofs of Theorems \ref{T.1.2}(i) and \ref{T.1.3}, specifically from the use of Lemma \ref{L.2.2} below in the proof of the $d=3$ case of Lemma \ref{L.3.2}.

\begin{definition} \lb{D.1.10}
Let $f_0,K,\tht$ be as in (H') and $\zeta, \eta>0$.  If $f$ satisfies (H'), define $\alpha_f(x;\zeta)$ as in \eqref{1.4a}.
Finally, let $F'(f_0,K,\tht,\zeta,\eta)$ be the set of all $(f,u^+)$ satisfying (H') such that $\alpha_f(x;\zeta)\ge \eta$ for all $x\in\bbR^d$ and  any equilibrium solution $p$ of \eqref{1.1} with $0<p<u^+$ satisfies
\beq \lb{1.21}
\sup_{x_0\in\bbR^d} \sum_{n\ge 1} \frac1{1+d(x_0,\calC_n)}  \le \frac 1\eta.
\eeq
Here $d(\cdot,\cdot)$ is the distance in $\bbR^d$ and  $\calC_1,\calC_2,\dots$ are all (distinct) unit cubes in $\bbR^d$,  whose corners have integer coordinates, such that $p(x)> \alpha_f(x;\zeta)$ for some $x\in\calC_n$.
\end{definition}

{\it Remarks.}  1.  The advantage of \eqref{1.4b}, relative to \eqref{1.21},  is that the former is a local condition while the latter is not.  Thus \eqref{1.21} is more difficult to check.  An obvious sufficient condition is when $p(\cdot)\le\alpha_f(\cdot ;\zeta)$  for each equilibrium $0<p<u^+$ (with $\zeta<c_0^2/4$, so that our results apply), which may be proved under some {\it quantitative} local hypotheses on $f$.  A simple such example is when $d=1=\tht_1$ and $f$ is sufficiently close to a homogeneous reaction $f_0$ as in (H') with $\int_0^\beta f_0(u) du>0$, where $\beta\in(\tht_0,1)$ is smallest number  such that $f_0(\beta)=c_0^2\beta/4$.
\smallskip

2.  Lemma \ref{L.2.2} shows that in the setting of (H),  \eqref{1.4b} implies \eqref{1.21} when $d\le 3$ (but not when $d\ge 4$), although with a different $\eta>0$.
\smallskip

3. \eqref{1.21} will cause typical solutions between 0 and $u^+$ to transition to $u^+$ (instead of to some other equilibrium $p<u^+$), and also to have a bounded width.  The latter need not be true without a condition like \eqref{1.21}, as is demonstrated by the example in the proof of Theorem \ref{T.1.2}(ii), for which the sum in \eqref{1.21} diverges, albeit slowly (as $\log n$).
\smallskip

Note that unlike $F(f_0,K,\tht,\zeta,\eta)$, the set $F'(f_0,K,\tht,\zeta,\eta)$ may be neither spatially translation invariant (although it would be if the $\calC_n$ were integer translations of any fixed unit cube $\calC$, and the sup in \eqref{1.21} were also taken over all such $\calC$) nor closed with respect to locally uniform convergence 
(i.e., locally uniform convergence for $q(x,u):=(f(x,u),u^+(x))$ on $\bbR^d\times\bbR$).
Since these properties will be essential in our analysis, in the following generalization of Theorems \ref{T.1.2}(i) and \ref{T.1.3} we will work with subsets $\calF\subseteq F'(f_0,K,0,\zeta,\eta)$ which possess them both (an example is the closure of all translations of a given $(f,u^+)$ with respect to locally uniform convergence).  We will denote $\calF_\tht:=\calF\cap F'(f_0,K,\tht,\zeta,\eta)$ for $\tht\ge0$, which then also has the same properties.

\begin{theorem} \lb{T.1.11}
Let  $f_0,K,$ and $\tht>0$ be as in (H') and let $\eta>0$, $\zeta\in(0,c_0^2/4)$, and $\calF \subseteq F'(f_0,K,0,\zeta,\eta)$  be spatially translation invariant and closed with respect to locally uniform convergence.  Let $(f,u^+)\in\calF_\tht$ and  let $u$   solve \eqref{1.1}, \eqref{1.2} with $u_0\in[0,u^+]$. 

(i) If $d\ge 1$ and $u_0$ satisfies \eqref{1.5} or \eqref{1.6}, then  (C)
holds with $1-\eps$ replaced by $u^+(x)-\eps$ in \eqref{1.7}, with $\ell_\eps,m_\eps$ depending only on $\eps,\calF$, and $\tau_{\eps,\delta}$ in Definition \ref{D.1.1b} also depending on  $\delta,f_1$.

(ii) If $d\ge 1$, $u_0$ satisfies \eqref{3.6a}, and $L_{u,\eps,1-\eps_0}(t_0)<\infty$  for $\eps_0>0$ from Lemma \ref{L.11.1} and each  $\eps>0$,  then (C)  holds for $u$ and for $v$ as in Theorem \ref{T.1.3}(ii), with $1-\eps$ replaced by $u^+(x)-\eps$ in \eqref{1.7}, with $\ell_\eps,m_\eps$ depending only on $\eps,\calF$, and $\tau_{\eps,\delta}$ in Definition \ref{D.1.1b} also depending on  $\delta,f_1$.
\end{theorem}

{\it Remarks.}  1.  Here  $T_\eps$ in (C) and $T_{\eps,\del}$ in Definition \ref{D.1.1b} depend on the same parameters as in Theorems \ref{T.1.2} (in (i)) and \ref{T.1.3} (in (ii)), but with $f_0,K,\zeta,\eta$ replaced by $\calF$.  This is also the case in Theorem~\ref{T.1.12} below, but there $T_\eps$ and $T_{\eps,\del}$ depend also on $\|u_0\|_\infty$.
\smallskip

2. These results again extend to the case $\tht=0$ in the slightly weaker form from Remark 4 after Theorem \ref{T.1.2} and Remarks 2,3 after Theorem \ref{T.1.3}.
\smallskip

Next, we consider extensions of our results to solutions that are not necessarily between the equilibria which they connect. 
We first  need to extend Definitions \ref{D.1.1} and \ref{D.1.1b} in a physically relevant manner to such solutions (we will do so for general $u^\pm$).
Namely, we will consider $u$ to be $\eps$-close to $u^\pm$ at $(t,x)$ if $|u(t,y)-u^\pm(y)|<\eps$ for all $y$ in a  ball centered at $x$, whose size grows to $\infty$ as $\eps\to 0$.  It will therefore be useful to define for $A\subseteq\bbR^d$,
\[
\text{$r$-int}\, A:=\{x\in A \,|\, B_r(x)\subseteq A \}.
\]

\begin{definition} \lb{D.1.4} 
Let $u^\pm$ be equilibrium solutions of \eqref{1.1} with  bounded Lipschitz $f$, satisfying \eqref{1.19}.  For a solution $u$ of \eqref{1.1} on $(t_0,\infty)\times\bbR^d$,  define (for $\eps\in(0,\tfrac12)$)
\[
\Omega_{u,\eps}(t):=\left\{x\in\bbR^d\,\big|\, |u(t,x)-u^-(x)|\ge \eps \right\}, 
\]
\[
\Omega_{u,1-\eps}(t):= \left\{x\in\bbR^d\,\big|\, |u(t,x)-u^+(x)|\le \eps \right\}, 
\]
\beq \lb{1.3b}
L_{u,\eps}(t) := \inf \left\{L>0 \,\big|\, \Omega_{u,\eps}(t)\subseteq B_L \left(\text{$\tfrac 1\eps$-int}\, \Omega_{u,1-\eps}(t) \right) \right\},
\eeq
\[
L_{u,1-\eps}(t) := \inf \left\{L>0 \,\big|\, \bbR^d\setminus 
\Omega_{u,1-\eps}(t) \subseteq B_L \left( \text{$\tfrac 1{\eps}$-int}\,\left[ \bbR^d\setminus
\Omega_{u,\eps}(t) \right] \right) \right\}, 
\]
\[
J_{u,\eps}(t) := \inf \left\{L>0 \,\big|\, \bbR^d= B_L \left(\text{$\tfrac 1\eps$-int}\, \Omega_{u,1-\eps}(t)  \cup  \text{$\tfrac 1{\eps}$-int}\,\left[ \bbR^d\setminus \Omega_{u,\eps}(t) \right] \right) \right\}. 
\]
We say that $u$ has a {\it bounded width (with respect to $u^\pm$)}  
if \eqref{1.4} holds for any $\eps\in(0,\tfrac 12)$, a {\it doubly-bounded width} if  \eqref{1.4} holds for any $\eps\in(0,\tfrac 12)\cup (\tfrac 12,1)$,  and a {\it semi-bounded width}  
if \eqref{1.4f} holds for any $\eps\in(0,\tfrac 12)$.  Definition \ref{D.1.1b} remains the same but with these new $\Omega_{u,\eps}(t)$.
\end{definition}

Parabolic regularity and strong maximum principle show that if $u^-\le u\le u^+$, then this new definition of bounded/doubly-bounded/semi-bounded width is equivalent to the one using \eqref{1.3}--\eqref{1.3aa} and these new $\Omega_{u,\eps}(t)$ (which are those from \eqref{1.20a}, \eqref{1.20b} if also $u^-=0$).  In fact, while  the new $L_{u,\eps}(t)$ 
is larger than the original one for such $u$, it is finite for all $\eps\in(0,\tfrac 12)$ resp.~all $\eps\in(0,\tfrac 12)\cup(\tfrac 12,1)$ as long as the same is true for the original $L_{u,\eps}(t)$. 


Finally, let us extend the definition of {\it spark-like} and {\it front-like} initial data as follows.  We will assume that
either there are $x_0\in\bbR^d$,  $R_2\ge R_1>0$, and $\eps_1,\eps_2>0$ such that
\beq \lb{1.22}
(\tht_0+\eps_1) \chi_{B_{R_1}(x_0)}(x) - e^{-\eps_2(|x-x_0|-R_2)} \chi_{\bbR^d\setminus B_{R_1}(0)}(x)\le u_0(x)\le e^{-\eps_2(|x-x_0|-R_2)}
\eeq
(with $R_1$ sufficiently large, depending on $\eps_1,\eps_2,R_2-R_1$, to guarantee spreading),
or there are $e\in\bbS^{n-1}$, $R_2\ge R_1$, and $\eps_1,\eps_2>0$ such that
\beq \lb{1.23}
(\tht_0+\eps_1) \chi_{\{x\,|\,x\cdot e<R_1\}}(x) -e^{-\eps_2(x\cdot e-R_2)} \chi_{\{x\,|\,x\cdot e\ge R_1\}}(x) \le u_0(x)\le e^{-\eps_2(x\cdot e-R_2)}.
\eeq

\begin{theorem} \lb{T.1.12}
Consider the setting of Theorem \ref{T.1.11} but with $u_0$ only bounded.

(i) Theorem \ref{T.1.11}(i) holds with \eqref{1.5}/\eqref{1.6} replaced by \eqref{1.22}/\eqref{1.23}, provided that
in the case of \eqref{1.23}, ``$\le$'' is replaced by ``$<$'' in \eqref{1.24} for all $(f,u^+)\in\calF$.

(ii) Theorem \ref{T.1.11}(ii) holds, provided that 
``$\le$'' and ``$\ge$'' are replaced by ``$<$'' and ``$>$'' in \eqref{1.24} for all $(f,u^+)\in\calF$.
\end{theorem}

{\it Remarks.}  1.  The extra condition in (i) guarantees  $\limsup_{t\to \infty} \sup_{x\in\bbR^d} [u(t,x)-u^+(x)]\le 0$ for any bounded $u_0$, uniformly in $\calF$.  This as well as (i) also hold for  from \eqref{1.23} if instead we assume $\limsup_{x\cdot e\to -\infty} [u_0(x)-u^+(x)]\le 0$, but then $T_\eps,T_{\eps,\del}$ in (i) depend on $u_0$ also via the rate of this decay (cf. Remark 1 after Theorem \ref{T.1.11}).

\smallskip

2.  The extra condition in (ii) guarantees $\limsup_{t\to \infty} \sup_{x\in\bbR^d} [u(t,x)-u^+(x)]\le 0$ and $\liminf_{t\to \infty} \inf_{x\in\bbR^d} u(t,x)\ge 0$ for any bounded $u_0$, uniformly in $\calF$.  
\smallskip


3.  Theorems \ref{T.1.2}(i) and \ref{T.1.3} extend similarly to solutions $u$ not necessarily in $[0,1]$.
\smallskip

\bigskip
\noindent
{\bf Entire Solutions with Bounded Widths}
\smallskip
\smallskip

Finally, let us turn to the discussion of the above-mentioned  entire solutions of \eqref{1.1}.  

\begin{definition} \lb{D.1.4a}
Let $u^\pm$ be equilibrium solutions of \eqref{1.1} satisfying \eqref{1.19}.  A {\it transition solution (connecting $u^-$ to $u^+$)} for \eqref{1.1} is a bounded entire 
solution $u$ of \eqref{1.1} which satisfies
\beq \lb{1.8}
\lim_{t\to\pm\infty} u(t,x) = u^\pm(x)
\eeq
locally uniformly on $\bbR^d$.  
\end{definition}

As above, we will assume $u^-\equiv 0$ without loss in the following.

\begin{theorem} \lb{T.1.5}
Let $u^-\equiv 0$ and  $u^+$ satisfy \eqref{1.19} and be equilibrium solutions of \eqref{1.1} with some Lipschitz $f$, satisfying \eqref{1.24} (but not necessarily (H')).  Let $u\not\equiv 0,u^+$ be a bounded entire solution of \eqref{1.1} which has a bounded width with respect to $0,u^+$.

(i) We have $0<u<u^+$.

(ii)  If $u$ propagates with a positive global mean speed, then $u$ is a transition solution.
 If,  in addition, there is  $\tht>0$ such that $f$ is non-increasing in $u$ on $[0,\tht]$ and on $[u^+(x)-\tht,u^+(x)]$ for each $x\in\bbR^d$, then $u_t>0$. 

(iii) Assume  $f_0,K,$ and $\tht>0$ are as in (H') and $\eta>0$, $\zeta\in(0,c_0^2/4)$, $\calF \subseteq F'(f_0,K,0,\zeta,\eta)$  is spatially translation invariant and closed with respect to locally uniform convergence. If $(f,u^+)\in\calF_\tht$,  
then  (C) holds for $u$, with $t_0+T_\eps$ replaced by $-\infty$ and $1-\eps$ by $u^+(x)-\eps$ in \eqref{1.7},  with $\ell_\eps,m_\eps$ depending only on $\eps,\calF$, and  $\tau_{\eps,\delta}$ in Definition \ref{D.1.1b} also depending on  $\delta,f_1$.
\end{theorem}

{\it Remarks.} 
1.  (i,ii) were proved in  \cite[Theorem 1.11]{BH3}, in a more general setting 
and for a smaller class of entire solutions called {\it invasions}. The latter have doubly-bounded widths and their reaction zones satisfy an additional geometric requirement  (see the discussion below). Our proof proceeds along similar lines, using a version of the sliding method.
\smallskip

2.  (ii) will play a crucial role in the proofs of Theorems \ref{T.1.2}(i), \ref{T.1.3}, \ref{T.1.11}, and \ref{T.1.12}.
\smallskip

3.  Notice that as long as $u$ has a bounded width in (iii),  we actually have the $u$-independent bound $\sup_{t\in\bbR} L_{u,\eps}(t)\le \ell_\eps$.
\smallskip

The hypothesis \eqref{1.24} is necessary in Theorem \ref{T.1.5},  even for homogeneous $f$ and $d=1$.  It is well known that, for instance, if $0\le f(u)\le f'(0)u$ for $u\in[0,1]$ (i.e., $f$ is a {\it KPP reaction} with $f'(0)>0$) and $f(u)=f'(0)u$ for $u<0$, then for any $c\in(0,2\sqrt{f'(0)})$ there is a traveling front solution $u(t,x)=U_c(x-ct)$ of \eqref{1.1} on $\bbR\times\bbR$ with $\lim_{s\to-\infty}U_c(s)=1$, $\lim_{s\to\infty}U_c(s)=0$, and $\inf_{s\in\bbR} U_c(s)<0$.  This solution satisfies neither (i) nor (ii).  Counter-examples with ignition $f$ also exist.
\smallskip

Theorem \ref{T.1.5} suggests a couple of interesting questions.
\smallskip

{\it Open problems.}  
1.  Does $u_t>0$ hold in Theorem \ref{T.1.5}(ii) when $\tht=0$?
\smallskip

2.  Does Theorem \ref{T.1.5}(ii) and/or Theorem \ref{T.1.5}(iii) hold if we drop the hypotheses of bounded width and positive global mean speed and instead only assume that $u\in[0,u^+]$ is a transition solution?  Of course, bounded width and positive global mean speed would then follow from the claim of Theorem \ref{T.1.5}(iii).
\smallskip



A natural question is whether solutions considered in Theorem \ref{T.1.5} must always exist. 
The following result answers this in the affirmative under the hypotheses of Theorem \ref{T.1.11}, even when $\tht=0$ in (H').
It also shows that transition solutions with {\it doubly-bounded width} need not exist for $d\ge 2$ even for ignition reactions, as was discussed in the introduction.

\begin{theorem} \lb{T.1.6}
(i) If $(f,u^+)\in\calF$, with $\calF$ as in Theorem \ref{T.1.11} (so $\tht=0$), then there exists a transition solution $u\in(0,u^+)$ for \eqref{1.1} with $u_t>0$ and  a bounded width. 

(ii) If $d\ge 2$, then there exists $f$ as in Theorem \ref{T.1.2} such that any bounded entire solution $u\not\equiv 0,1$ for \eqref{1.1} with bounded width is a transition solution $u\in(0,1)$ satisfying $u_t>0$ and $\lim_{t\to\infty} \inf_{x\in\bbR^2} u(t,x)=1$.
In particular, there exists no transition solution with a doubly-bounded width for \eqref{1.1} (and hence also no transition front --- see the discussion below). 
\end{theorem}

{\it Remarks.}  1. The hypothesis $\zeta<c_0^2/4$ is at least qualitatively necessary in (i), as counterexamples  with $\zeta >c_0^2/2$ exist even for $d=1$  \cite{NRRZ}.  
\smallskip

2. Note that for $d=1$, transition fronts always exist under the hypotheses in (ii) \cite{ZlaGenfronts}.  The first example of non-existence of fronts was given in \cite{NRRZ} for {\it KPP reactions} (and $d=1$).  It is based on the construction of $f$ for which the equilibrium $u\equiv 0$ is {\it strongly unstable} in some region of space, so that arbitrarily small amounts of heat diffusing far ahead of the reaction zone quickly ignite on their own inside this region. (ii) is the first non-existence result for ignition reactions (so it does not rely on this strong instability property of KPP reactions).
\smallskip

Before proving the above results, let us note that while the concepts of bounded and semi-bounded width of solutions to \eqref{1.1}
are new for $d\ge 2$, 
the concept of doubly-bounded width
is closely related to 
the Berestycki-Hamel definition of transition fronts from \cite{BH2,BH3}, which motivated this work.
The latter definition is more geometric in nature and its scope is slightly different from ours.  It involves entire solutions  rather than solutions of the Cauchy problem,  and is also stated for wider classes of PDEs and spatial domains, and vector-valued solutions with possibly time-dependent coefficients and $u^\pm$.  This is beyond the scope of the present paper (although the corresponding generalizations are rather straightforward), so we will only discuss the case at hand: \eqref{1.1} on $\bbR\times\bbR^d$ with bounded Lipschitz $f$ and time-independent $u^\pm$ satisfying \eqref{1.19}.  

In this setting, the definition in \cite{BH3} says that a {\it transition front} connecting $u^-$ and $u^+$ 
is an entire solution $u$ such that for each $t\in\bbR$ there are open non-empty sets $\Omega_t^\pm\subseteq\bbR^d$ satisfying
\beq \lb{1.11}
\Omega_t^-\cap\Omega_t^+=\emptyset, \quad \partial\Omega_t^-=\partial\Omega_t^+=:\Gamma_t, \quad \Omega_t^-\cup\Gamma_t\cup\Omega_t^+=\bbR^d,
\eeq
\beq \lb{1.13}
\sup \left\{ d(y,\Gamma_t)\,\big|\, y\in  \Omega_t^\pm\cap \partial B_r(x) \right\} \to \infty \text{ as $r\to\infty$, uniformly in $t\in\bbR $ and $x\in\Gamma_t$,}
\eeq
\beq \lb{1.14}
u(t,x)-u^\pm(x) \to 0 \text{ as $d(x,\Gamma_t)\to\infty$ and $x\in\Omega_t^\pm$, uniformly in $t\in\bbR$,}
\eeq
and there is $n\ge 1$ such that for each $t\in\bbR$,
\beq \lb{1.15}
\text{$\Gamma_t$ is a subset of $n$ (rotated in $\bbR^d$) graphs of functions from $\bbR^{d-1}$ to $\bbR$.}
\eeq

While we have to forgo geometric conditions, such as \eqref{1.15}, in our definitions (as was explained earlier), it is not difficult to see that \eqref{1.11}--\eqref{1.14}
for an entire solution $u\not\equiv u^\pm$ are in fact {\it equivalent} to $u$ having a doubly-bounded width!   Indeed, if $u\not\equiv u^\pm$ has a doubly-bounded width (in the sense of Definition \ref{D.1.4} if $u\notin[u^-,u^+]$), one only needs to take 
\beq \lb{1.16}
\Omega_t^+:= {\rm int}\left\{ x\in\bbR^d \,\Bigg|\, u(t,x)\ge \frac {u^+(x)+u^-(x)}2 \right\},
\eeq
 $\Gamma_t:=\partial\Omega_t^+$, and $\Omega_t^-:=\bbR^d\setminus\bar\Omega_t^+$.  (Of course, the sets $\Omega_t^\pm,\Gamma_t$ from \eqref{1.11} are not unique!)
On the other hand, when \eqref{1.11}--\eqref{1.14} holds, it is easy to see that $\Gamma_t$ and the boundary of the set from \eqref{1.16} are within a (uniformly in $t$) bounded distance of each other.
So transition fronts are precisely those entire solutions with doubly-bounded widths which also satisfy \eqref{1.15}.  In particular, Theorems \ref{T.1.5} and \ref{T.1.6}(ii) apply to them, the latter also showing that transition fronts need not always exist in dimensions $d\ge 2$.

We note that the condition \eqref{1.8} for our transition solutions also has a counterpart  in \cite{BH3}.  There an {\it invasion} of $u^-$ by $u^+$ is defined to be a transition front connecting $u^\pm$ for which  
\beq \lb{1.17}
\text{$\Omega_s^+\subseteq \Omega_t^+$ when $s\le t$} \qquad\text{and}\qquad \lim_{r\to\infty} \inf_{|t-s|=r} d(\Gamma_t,\Gamma_s)= \infty.
\eeq
This condition, together with \eqref{1.11}--\eqref{1.14}, implies \eqref{1.8} but is stronger than our definition of transition solutions with doubly-bounded widths. Nevertheless, if we relax \eqref{1.17} to the existence of $T$ such that
\beq \lb{1.18}
\text{$\Omega_s^+\subseteq \Omega_t^+$ when $s+T\le t$} \qquad\text{and}\qquad \lim_{r\to\infty} \inf_{|t-s|=r} d(\Gamma_t,\Gamma_s)= \infty,
\eeq
then  \eqref{1.11}--\eqref{1.14}, \eqref{1.18} are in fact {\it equivalent} to our definition of transition solutions with doubly-bounded widths which also propagate with a positive global mean speed.  Indeed, notice that \eqref{1.18} implies that $\inf_{|t-s|=r} d(\Gamma_t,\Gamma_s)$ grows at least linearly as $r\to\infty$, so we can again use \eqref{1.16} to define $\Omega_t^+$.

\vbox{
\bigskip
\noindent
{\bf Organization of the Paper and Acknowledgements}
\smallskip
\smallskip

In Section \ref{S2} we prove some preliminary results.  Section \ref{S3} is the heart of the argument proving bounded widths of solutions for $d\le 3$ (the proof of Lemma~\ref{L.3.2} is considerably more complicated for $d=3$, so it is postponed until Section \ref{S5}).
Theorem \ref{T.1.3}(i) will then be obtained in the short Section \ref{S4}, and a more involved argument (along with Theorem \ref{T.1.5}(ii)) will be needed to prove Theorem~\ref{T.1.3}(ii) and Theorem \ref{T.1.2}(i) in Section \ref{S4a}.  All these arguments are extended in Section \ref{S11} to obtain proofs of Theorems \ref{T.1.11} and \ref{T.1.12}, and in Section \ref{S8} we prove Theorem \ref{T.1.5} (the proof of its parts (i,ii) only uses Lemma \ref{L.8.1} below and, in particular, not the results proved in Section \ref{S4a}).  In Section \ref{S6} we prove Theorem \ref{T.1.2}(ii) by means of a counter-example (the proof is also independent of the rest of the paper) and  Theorem~\ref{T.1.6} is proved in Section \ref{S7}.
}

The author thanks \' Arp\' ad Baricz, Henri Berestycki, Fran\c cois Hamel, and Hiroshi Matano for helpful discussions and comments.  He also acknowledges partial support  by NSF grants DMS-1056327, DMS-1113017,  and DMS-1159133.

\section{Preliminaries (case $u^+\equiv 1$)} \lb{S2}

In Chapters \ref{S2}--\ref{S5} we consider the setting of Theorems \ref{T.1.2} and \ref{T.1.3}, with $f_0,K,\tht$ as in (H), $u^+\equiv 1$,  and $u\in[0,1]$.  We will extend the results below to the setting of (H') in Chapter \ref{S11}.

Let us start with some useful preliminary lemmas.  

\begin{lemma} \lb{L.2.1}
There is $\eps_0=\eps_0(f_0,K)> 0$ such that for each $c<c_0$ and $\eps>0$ there is $\tau=\tau(f_0,K,c,\eps)\ge 0$ such that the following holds.  If $u\in[0,1]$ solves \eqref{1.1}, \eqref{1.2} with $f$ from (H), and $u(t_1,x)\ge 1-\eps_0$ for some $(t_1,x)\in [t_0+1,\infty)\times\bbR^d$, then for each $t\ge t_1+\tau$,
\beq \lb{2.1}
\inf_{|y-x|\le c(t-t_1)} u(t,y) \ge 1-\eps.
\eeq
The same result holds if the hypothesis $u(t_1,x)\ge 1-\eps_0$ is replaced by
\beq\lb{2.1a}
u(t_1,\cdot) \ge \frac {1+\tht_0}2 \chi_{B_R(x)}(\cdot)
\eeq
 for some $(t_1,x)\in [t_0,\infty)\times\bbR^d$ and a large enough $R=R(f_0)>0$.
\end{lemma}


\begin{proof}
The second claim is proved in \cite{AW} when $f(y,\cdot)=f_0(\cdot)$ for all $y\in\bbR^d$ and follows for general $f$ by the comparison principle.  

The first claim holds because \eqref{2.1a} follows from $u(t_1,x)\ge 1-\eps_0$, 
provided $\eps_0>0$ is sufficiently small (depending on $f_0,K$) .  Indeed, assume that for each $n\in\bbN$ there were $f_n$ satisfying (H) and $u_n$ solving \eqref{1.1} on $(-1,\infty)\times\bbR^d$ with $f=f_n$, such that $u_n(0,0)\ge 1-\tfrac 1n$ and $\inf_{y\in B_R(0)} u_n(0,y)<\tfrac 12(1+\tht_0)$ (note that we can shift $(t_1,x)$ to $(0,0)$ without loss, and then $t_0\le -1$).  By parabolic regularity, there is a subsequence $\{{n_j}\}_{j\ge 1}$ with $u_{n_j}$ and $f_{n_j}$  locally uniformly converging to $u\in[0,1]$ and $f$ such that $f$ satisfies (H) and $u$ solves \eqref{1.1} on $(-1,\infty)\times\bbR^d$, with $u(0,0)=1$ and $\inf_{y\in B_R(0)} u(0,y)<1$.  But this contradicts the strong maximum principle, and we are done.
\end{proof}

The first claim of this result immediately shows that solutions with bounded widths propagate with global mean speed in $[c_0,\infty]$.  It turns out that bounded width also makes the global mean speed not exceed $c_1$, at least in the ignition case.  This can be proved by a separate argument and we state both these results in the following lemma.

\begin{lemma} \lb{L.2.1a}
Let $f_0,K$ be as in (H) and $f_1$ be pure ignition. 
For each $\eps\in(0,\tfrac 12)$ and $\delta>0$ there is $\eps'>0$ and $\tau<\infty$  such that the following holds.
If $u\in[0,1]$ solves \eqref{1.1} on $(t_0,\infty)\times\bbR^d$ with ignition $f$ from (H) satisfying \eqref{1.4e}, and $\sup_{t\in[t_0+1,t_3]}L_{u,\eps'}(t)\le L$, then
\[
 B_{(c_0-\delta)(t_2-t_1)-L} \left(\Omega_{u,\eps}(t_1) \right) \subseteq \Omega_{u,1-\eps}(t_2) 
\qquad\text{and}\qquad
\Omega_{u,\eps}(t_2) \subseteq B_{(c_1+\delta)(t_2-t_1)+L} \left(\Omega_{u,1-\eps}(t_1) \right)
\]
whenever $t_1\ge t_0+1$ and $t_2\in[t_1+\tau,t_3]$. 
\end{lemma}


\begin{proof}
The first inclusion is immediate for any $\eps'\in(0,\min\{\eps,\eps_0\}]$, with $\tau$ from Lemma \ref{L.2.1} with $\eps$ and $c:=c_0-\delta$.  Indeed, if $x\in\Omega_{u,\eps}(t_1)$, then $\bar B_L(x)\cap \Omega_{u,1-\eps_0}(t_1)\neq \emptyset$, so Lemma \ref{L.2.1} yields the result (even for monostable $f$).

Let us now consider the second inclusion.  Extend $f_1$ by 0 to $\bbR\setminus[0,1]$.  It is well known that for any $\delta>0$ there is $\eps'\in(0,\tfrac \eps 2)$ and a traveling front for some  $f_2\ge f_1(\ge 0)$ with $f_2\equiv 0$ on $[0,2\eps']\cup\{1+\eps'\}$, which has speed $c_2\in[c_1, c_1+\tfrac \delta 3]$ and connects $\eps'$ and $1+\eps'$.  That is, there is a solution of $U''+c_2U'+f_2(U)=0$ on $\bbR$ with $U'<0$, $U(-\infty)=1+\eps'$ and $U(\infty)=\eps'$ (and we can also assume $U(0)=2\eps'$ after translation).  Indeed, one only needs to take $\eps'$ small enough and $f_2$ close enough to $f_1$.

Let $z_1:=\tfrac{6d}\delta$, $z_2:=\tfrac{6d+7}\delta$ and let $h:[0,\infty)\to[0,\infty)$ be any $C^2$ function with $h\equiv 0$ on $[0,z_1]$, $h'\equiv 1$ on $[z_2,\infty)$, and $h'\le 1$ and $h''\in[0,\tfrac \delta 6]$ on $[z_1,z_2]$.  We now claim that
\beq\lb{1.3hh}
v(t,x):=U \left( z_2-h(|x|)- \left(c_2+\frac\delta 3 \right)t \right)
\eeq
satisfies 
\beq \lb{1.3jj}
v_t\ge \Delta v+f_2(v) \qquad\text{on $(-\infty,0)\times\bbR^d$.}  
\eeq
Indeed, for $|x|\le z_1$ the argument of $U$ is positive (so $f_2(U)=0$) and we have
\[
v_t-\Delta v - f_2(v) = -\left(c_2+\frac\delta 3 \right)U' \ge 0.
\]
For $|x|\ge z_1$ we get
\[
-v_t+\Delta v + f_2(v) = \left[ \left(c_2+\frac\delta 3 \right) -h''(|x|)-\frac{d-1}{|x|}h'(|x|) \right] U'+ \left( h'(|x|) \right)^2 U'' +f_2(U) =(*).
\]
If $|x|\ge z_2$, then 
\[
(*)= \left[ \left(c_2+\frac\delta 3 \right) -\frac{d-1}{|x|} \right] U'+  U'' +f_2(U) = \left[ \frac\delta 3  -\frac{d-1}{|x|} \right] U' \le 0.
\]
If $|x|\in[z_1,z_2]$, then again the argument of $U$ is positive (so $f_2(U)=0$) and we have
\begin{align*}
(*) & = \left[ \left(c_2+\frac\delta 3 \right) -h''(|x|)-\frac{d-1}{|x|}h'(|x|) - c_2 \left( h'(|x|) \right)^2 \right] U'
\\ & = \left[ c_2 \left( 1-\left( h'(|x|) \right)^2 \right) + \left( \frac\delta 6-h''(|x|) \right) + \left( \frac\delta 6 -\frac{d-1}{|x|}h'(|x|) \right) \right] U'.
\end{align*}
Each of the three terms in the last square bracket is non-negative, so  again $(*)\le 0$ and \eqref{1.3jj} holds.

We now let $\tau:=\tfrac 3\delta(2z_2-U^{-1}(1))$ and consider arbitrary $y\notin B_{(c_1+\delta)(t_2-t_1)+L} \left(\Omega_{u,1-\eps}(t_1) \right)$.  By the hypothesis and $\eps>\eps'$, we have $u(t,x)<\eps'$ for all $x\in B_{(c_1+\delta)(t_2-t_1)}(y)$.  The function 
\[
w(t,x):=v(t-t_2,x-y)\quad(\ge \eps')
\]
is obviously  a super-solution of \eqref{1.1} on $(t_1,t_2)\times\bbR^d$, and for $x\notin B_{(c_1+\delta)(t_2-t_1)}(y)$ we have
\[
w(t_1,x) \ge U \left( 2z_2-|x-y|- \left(c_1+\frac{2\delta} 3 \right)(t_1-t_2) \right) \ge U \left( 2z_2-\frac{\delta} 3 (t_2-t_1) \right) \ge U \left( 2z_2-\frac{\delta} 3\tau \right)=1
\]
by $U'<0$, $h(z)\ge z-z_2$, $c_2\le c_1+\tfrac\delta 3$, and $t_2-t_1\ge \tau$.  Hence $w(t_1,\cdot)\ge u(t_1,\cdot)$ and so
\[
u(t_2,y)\le w(t_2,y)=v(0,0)=U(z_2)<2\eps'<\eps.
\]  
Thus $y\notin\Omega_{u,\eps}(t_2)$ and we are done.
\end{proof}

During the proofs of our main results, we will sometimes need to pass to limits along subsequences of $\{(f_n,u_n)\}$, where $f_n$ satisfy (H) and $u_n\in[0,1]$ have uniform-in-$n$ bounds on their widths.  The following will be useful.

For $\eps\in(0,\tfrac 12)$, $\ell>0$,  and $t_0\in[-\infty,\infty)$, let $S_{t_0,\eps,\ell}=S_{t_0,\eps,\ell}(f_0,K,\tht)$ be the set of all pairs $(f,u)$ such that  $f$ satisfies (H) with the given $f_0,K,\tht$
and $u\in[0,1]$ solves \eqref{1.1} on $(t_0,\infty)\times\bbR^d$ and satisfies $L_{u,\eps'}(t)\le \ell$ for all $\eps'\in(\eps,\tfrac 12)$ and all $t> t_0$.  For non-increasing and left-continuous $L:(0,\tfrac 12)\to(0,\infty)$, let 
\[
S_{t_0,L}=S_{t_0,L}(f_0,K,\tht):= \bigcap_{\eps\in(0,1/2)} S_{t_0,\eps,L(\eps)} (f_0,K,\tht).
\]
(so $(f,u)\in S_{t_0,L}$ implies $L_{u,\eps}(t)\le L(\eps)$ for $\eps\in(0,\tfrac 12)$ and $t>t_0$, by left-continuity of $L$) and
\[
S_L=S_L(f_0,K,\tht):= \{ (f,u)\,|\, (f,u)\in S_{-\infty,L}(f_0,K,\tht) \text{ and } u\not\equiv 0,1 \}.
\]
Thus any entire solution $u\in[0,1]$ of \eqref{1.1}  with bounded width, except $u\equiv 0,1$, appears in some $S_L$. 
Of course, strong maximum principle gives $u\in(0,1)$ if $(f,u)\in S_L$.

\begin{lemma} \lb{L.8.1}
Fix $f_0,K,\tht$ and $L$ as above and let $t_0\in[-\infty,\infty)$.

(i) If for $\eps\in(0,\tfrac 12)$ and $\ell>0$ we have $(f_n,u_n)\in S_{t_n,\eps,\ell}(f_0,K,\tht)$ and $\limsup_{n\to\infty}t_n\le t_0$, then there is $n_j\to\infty$ (as $j\to\infty$) and  $(f,u)\in S_{t_0,\eps,\ell}(f_0,K,\tht)$ such that $f_{n_j}\to f$  locally uniformly on $\bbR^{d}\times[0,1]$ and $u_{n_j}\to u$ locally uniformly on $(t_0,\infty)\times\bbR^d$.  

(ii) If for each $\eps\in(0,\tfrac 12)$ we have $(f_n,u_n)\in S_{t_n(\eps),\eps,L(\eps)}(f_0,K,\tht)$ and $\limsup_{n\to\infty}t_n(\eps)\le t_0$, then there is $n_j\to\infty$ (as $j\to\infty$) and  $(f,u)\in S_{t_0,L}(f_0,K,\tht)$ such that $f_{n_j}\to f$  locally uniformly on $\bbR^{d}\times[0,1]$ and $u_{n_j}\to u$ locally uniformly on $(t_0,\infty)\times\bbR^d$.

(iii)  If $\eps\in(0,2\eps_0]$ and $\ell>0$, then 
\beq \lb{1.7b}
\inf \left\{ u_t(t,x) \,\Big|\, (f,u)\in S_{0,\eps/2,\ell}(f_0,K,0),\, u_t\ge 0 \text{ on $(0,\infty)\times\bbR^d$, } 
t\ge 1, 
u(t,x)\in[\eps,1-\eps] \right\} >0
\eeq
\end{lemma}

\begin{proof}
(i) The properties of $f_n$, uniform boundedness of $u_n$, and standard parabolic regularity for $u_n$ prove existence of locally uniform limits $f,u$ along a subsequence $\{n_j\}_{j\ge 1}$, as well as that $f$ satisfies (H) (with the same $f_0,K,\tht$) and $u$ solves \eqref{1.1}.  Locally uniform convergence $u_{n_j}\to u$ then yields $L_{u,\eps'}(t)\le \ell$ for all $\eps'\in(\eps,\tfrac 12)$ and all $t> t_0$ (just pick any $\eps''\in(\eps,\eps')$ and then a large enough $j$).  Thus  $(f,u)\in S_{t_0,\eps,\ell}$.

(ii)  The proof is identical to (i).

(iii)  
Assume that the $\inf$ in \eqref{1.7b} is 0.   Then there are $(f_n,u_n)\in S_{0,\eps/2,\ell}$ with $(u_n)_t\ge 0$ and $(t_n,x_n)\in[1,\infty)\times\bbR^d$ such that $u_n(t_n,x_n)\in[\eps,1-\eps]$ and $(u_n)_t(t_n,x_n)\in[0,\tfrac 1n]$.  After shifting $(t_n,x_n)$ to $(1,0)$ and applying (i), we obtain $(f,u)\in S_{0,\eps/2,\ell}$ with  $u(1,0)\in[\eps,1-\eps]$ and $u_t\ge 0 =u_t(1,0)$.  The strong maximum principle for the linear PDE $v_t=\Delta v+f_u(x,u(t,x))v$, satisfied by $u_t$, then implies $u_t\equiv 0$.  This however contradicts Lemma \ref{L.2.1}, which yields $\lim_{t\to\infty} u(t,0)= 1\,(>u(1,0))$ because $\sup_{x\in B_\ell(0)} u(1,x)\ge 1-\tfrac\eps 2 \,(\ge 1-\eps_0)$.
\end{proof}

An important role in the proof of Theorems \ref{T.1.2} and \ref{T.1.3} will be played by equilibrium solutions of \eqref{1.1}.  

\begin{lemma} \lb{L.2.2}
Let $f\ge 0$ be Lipschitz and $v\in[0,1]$ satisfy
\beq \lb{2.2}
\Delta v + f(x,v)=0
\eeq
on $\bbR^d$.  If $d\le 2$, then $v$ is constant and $f(x,v(x))\equiv 0$.  If $d\ge 3$, then 
\beq \lb{2.3}
\int_{\bbR^d} |x|^{2-d} f(x,v(x)) dx\le (d-2)|\partial B_1(0)|.
\eeq
\end{lemma}

\begin{proof}
Integrating \eqref{2.2} over $B_r:=B_r(0)$ and using the divergence theorem yields
\[
\int_{B_r} f(x,v(x)) dx = -\int_{\partial B_r} \nabla v(x) \cdot n(x) \,d\sigma_r(x) = -r^{d-1} \int_{\partial B_1} \til v_\rho(r,y) \,d\sigma_1(y)
\]
where $n$ is the unit outer normal and $\sigma_r$ the surface measure for $\partial B_r$, and $\til v(\rho,y)=v(\rho y)$ for $(\rho,y)\in(0,\infty)\times \partial B_1$.  Multiplying by $r^{1-d}$ and integrating in $r\in[0,r_0]$  gives
\[
\int_{B_{r_0}} [l(r_0)-l(|x|)] f(x,v(x)) dx = \int_{0}^{r_0} r^{1-d} \int_{B_r} f(x,v(x)) dx dr = \int_{\partial B_1} [\til v(0,y) - \til v(r_0,y)] \,d\sigma_1(y),
\]
where $l(r)=\ln r$ if $d=2$ and $l(r)=r^{2-d}/(2-d)$ otherwise.
Taking $r_0\to \infty$ finally yields
\[
\int_{\bbR^d} [l(\infty)-l(|x|)] f(x,v(x)) dx = |\partial B_1(0)| v(0) - \lim_{r\to \infty} r^{1-d}\int_{\partial B_r} v(x) d\sigma_r(x).
\]
Since $v\in[0,1]$, either $f(x,v(x))\equiv 0$ (and then $v$ is constant) or $d\ge 3$ and \eqref{2.3} holds.
\end{proof}

\begin{lemma} \lb{L.2.3}
For $\zeta>0$, let $\Psi(x)=\psi(|x|)$ be the radially symmetric solution of 
\beq \lb{2.4}
\Delta \Psi = \zeta \Psi 
\eeq
on $\bbR^d$ with $\Psi(0)=1$. Then $\psi,\psi'>0$ on $(0,\infty)$ and
\beq \lb{2.5}
\lim_{r\to\infty}  \left( \sqrt\zeta r \right)^{(d-1)/2} e^{-\sqrt\zeta r} \psi^{(k)}(r)= \zeta^{k/2} l_{d}
\eeq
for some $l_{d}\in(0,\infty)$ and $k=0,1$.  In particular,
\beq \lb{2.6}
\lim_{r\to\infty}  \psi'(r)\psi(r)^{-1}=\sqrt\zeta
\eeq
\end{lemma}

{\it Remark.} We only need  $k=0,1$  here but \eqref{2.5} holds for any $k\ge 0$.

\begin{proof}
Here $\psi$ is the unique solution of $\psi''+\tfrac{d-1}r\psi'=\zeta\psi$ on $(0,\infty)$, with $\psi(0)=1$ and $\psi'(0)=0$, which is obviously positive along with $\psi'$.   If $d=1$, one easily checks that 
\beq \lb{2.7}
\psi(r)=\frac {e^{\sqrt\zeta r} + e^{-\sqrt\zeta r}} 2,
\eeq
so \eqref{2.5} holds with $l_1=\tfrac 12$.  
If $d\ge 2$, then $\phi(r):=r^{(d-2)/2} \psi(\zeta^{-1/2}r)$ satisfies 
\[
\phi''+\frac 1r\phi'- \left[1+ \frac{(d-2)^2}{4r^2} \right] \phi=0
\]
 on $(0,\infty)$, with $\lim_{r\to 0} r^{(2-d)/2} \phi(r)=1$ and $\lim_{r\to 0} \tfrac d{dr}[r^{(2-d)/2} \phi(r)]=0$.  Thus by \cite[p.375]{AS}, $\phi=c_dI_{(d-2)/2}$ for $I_\nu$ ($\nu\in\bbC$) the modified Bessel function of the first kind and some $c_d>0$  (in fact, $c_d=2^{(d-2)/2}\Gamma(\tfrac d2)$).  But now \eqref{2.5} follows from $\lim_{r\to\infty} \sqrt{r} e^{-r}I_\nu^{(k)}(r)=(2\pi)^{-1/2}$ for $k=0,1$ \cite[pp.~377 and 378]{AS}, with $l_d:=(2\pi)^{-1/2}c_d$.
\end{proof}

\section{Bounded Widths for Solutions $u\in[0,1]$ with $u_t\ge 0$ (case $u^+\equiv 1$)} \lb{S3}

Again we consider $f_0,K,\tht$ as in (H), $u^+\equiv 1$, $u\in[0,1]$, and also $\eta>0$ and  $\zeta\in(0,c_0^2/4)$. All constants in this section will depend on $f_0,K,\zeta,\eta$ (but not on $\tht$, unless explicitly noted!).

We define $\zeta':=\tfrac {c_0^2}8+\tfrac\zeta 2 \in(\zeta, c_0^2/4)$ and choose any 
\beq\lb{3.00}
h\in \left[0,\min\left\{ \tht\frac{c_0^2-4\zeta}{c_0^2+4\zeta},\frac\eta{4K} \right\} \right]
\eeq
(obviously $h=0$ when $\tht=0$).
This yields $\zeta'(\tht-h)\ge\zeta\tht$, 
which guarantees  $\zeta'(u-h)\ge \zeta u$ for all $u\ge\tht$.  Hence, any $f\in F(f_0,K,\tht,\zeta,\eta)$ satisfies
 \beq \lb{3.0b}
 f(x,u)\le \zeta'(u-h) \qquad\text{for $x\in\bbR^d$ and $u\in[h, \alpha_f(x)]$}.
 \eeq 
Here, and always,  $\alpha_f(x)=\alpha_f(x;\zeta)$ (not $\alpha_f(x;\zeta')$).
Let us also take  $\eps_0$ from Lemma \ref{L.2.1} and  $\psi$ from Lemma \ref{L.2.3} corresponding to $\zeta'$.  Below, $\psi^{-1}(\cdot)$ is the inverse function to $\psi(\cdot)$ on $[0,\infty)$ while $\psi(\cdot)^{-1}=1/\psi(\cdot)$.  

In the following we will assume that  $f\in F(f_0,K,\tht,\zeta,\eta)\, (\subseteq F(f_0,K,0,\zeta,\eta))$ and $u\in[0,1]$  solves \eqref{1.1}, \eqref{1.2}.  For any $(t,y)\in[t_0,\infty)\times\bbR^d$ we define
\begin{align}
Z_y(t):= & \inf_{u(t,x)\ge 1-\eps_0} |x-y| \qquad(\in [0,\infty]), \lb{3.1} 
\\ Y^{h}_y(t):= & \sup \left\{ \rho \,\big|\, u(t,\cdot)\le h+\psi(\rho)^{-1} \psi(|\cdot-y|) \right\} \qquad(\in [0,\infty]), \lb{3.2}
\end{align}
and $\gamma^{h}_y(t):=\psi(Y^{h}_y(t))^{-1}$.
That is, $Z_y(t)$ is the distance from $y$ to the nearest point with value of $u$ sufficiently close to 1, while $Y^{h}_y(t)$ is the distance from $y$ to the points where the best 
upper bound of the form $h+\gamma\psi(|\cdot-y|)$ on $u$ takes the value $1+h$ (both at time $t$), and $\gamma^{h}_y(t)$ is the $\gamma$ from that bound. The latter is clearly non-increasing in $h$, hence  $Y^{h}_y(t)$ is non-decreasing in $h$.  Note that \eqref{2.6} immediately shows
 \beq \lb{3.2b}
 Y^{h}_y(t)\le Z_y(t) + M
 \eeq 
 for some ($\tht,h$-independent) $M\ge 0$.
 
 Let us also fix any $c_Y,c_Z$ such that  
\beq \lb{3.0}
2\sqrt{\zeta'} < c_Y < c_Z < c_0,
\eeq
for instance, $c_Y:=\tfrac 14 c_0 + \tfrac 32\sqrt{\zeta'}$ and $c_Z:=\tfrac 34 c_0 + \tfrac 12 \sqrt{\zeta'}$.
Let $\tau_Z\ge 0$ correspond to $c=c_Z$ and $\eps=\eps_0$ in Lemma \ref{L.2.1} and  let $r_Y\ge 0$ be such that 
\[
\frac{\psi'(r)}{\psi(r)}\ge \frac {4\zeta'} {c_Y +2\sqrt{\zeta'}} \qquad \left( > \frac {2\zeta'} {c_Y} \right)
\]
 for $r\ge r_Y$ (which exists by \eqref{2.6}, with $\zeta'$ in place of $\zeta$, and $c_Y>2\sqrt{\zeta'}$).  Finally,  let 
\beq \lb{3.0a}
c_Y':= \frac{(K +\zeta')c_Y}{2\zeta' } \qquad \left(> \frac{(K +\zeta')}{\sqrt{\zeta'}}\ge 2\sqrt{K } \ge c_1 \right).
\eeq

The choice of $Y^{h}_y$ is motivated by the following result.

\begin{lemma} \lb{L.3.1}
Let $(t_1,y)\in [t_0,\infty)\times\bbR^d$.  

(i) If $t\ge t_1$ is such that $Y^{h}_y(t_1)-c_Y'(t-t_1)\ge r_Y$, then
\beq \lb{3.3}
Y^{h}_y(t)\ge Y^{h}_y(t_1) - c_Y'(t-t_1). 
\eeq

(ii) If $t_2\ge t_1$ is such that $Y^{h}_y(t_1)-c_Y(t_2-t_1)\ge r_Y$ and 
$u(t,x)\le \alpha_f(x)$ on the set $A:=\{(t,x) \,|\, t\in [t_1,t_2] \text{ and }|x-y|\le Y^{h}_y(t_1)-c_Y(t-t_1)\}$, then
\beq \lb{3.4}
Y^{h}_y(t)\ge Y^{h}_y(t_1) - c_Y(t-t_1) 
\eeq
for any $t\in [t_1,t_2]$.  

(iii)  If $t_1\ge t_0+1$ and $t\ge t_1+\tau_Z$, 
then
\beq \lb{3.5}
Z_y(t)\le \left[ Z_y(t_1) - c_Z(t-t_1) \right]_+ .
\eeq
\end{lemma}

 {\it Remark.}  The point here is that (ii) and (iii), together with $c_Y<c_Z$, will  keep $Z_y(t)-Y^{h}_y(t)$ uniformly bounded above.
This is done in Lemma \ref{L.3.2} below.  It turns out, however, that the hypothesis of (ii) is too strong  to make this idea directly applicable for $d\ge 3$.  Lemma \ref{L.3.2} nevertheless still holds for $d=3$, albeit with a considerably more involved proof (see Section \ref{S5} below).  For $d\ge 4$ the lemma is false in general.
 
\begin{proof}
(i) Since $w(t,x):=h+e^{(K +\zeta')(t-t_1) } \gamma^{h}_y(t_1)\psi(|x-y|)$ is a super-solution of \eqref{1.1} due to $f(x,u)\le K(u-\tht)\le K(u-h)$, the comparison principle gives
\beq \lb{3.6}
\gamma^{h}_y(t)\le e^{(K +\zeta')(t-t_1) } \gamma^{h}_y(t_1)
\eeq
for any $t\ge t_1$.  From this and \eqref{3.2} we obtain
\[
\ln \psi(Y^{h}_y(t)) \ge \ln \psi(Y^{h}_y(t_1)) -(K +\zeta')(t-t_1).
\]
Since $\tfrac d{dr}[\ln\psi(r)]\ge 2\zeta'/c_Y$ for $r\ge r_Y$, it follows that
\[
Y^{h}_y(t)\ge Y^{h}_y(t_1) - \frac{(K +\zeta')c_Y}{2\zeta'}(t-t_1) = Y^{h}_y(t_1)-c_Y'(t-t_1)
\]
for all $t\in[t_1,t_2]$, where $t_2\ge t_1$ is the first time such that $Y^{h}_y(t_2)= r_Y$.  Thus $r_Y\ge Y^{h}_y(t_1)-c_Y'(t_2-t_1)$, so $t\le t_2$ due to $Y^{h}_y(t_1)-c_Y'(t-t_1)\ge r_Y$,  and we are done.

(ii)  Let $\beta(t)$ be such that $w(t,x):=h+e^{\beta(t) } \gamma^{h}_y(t_1)\psi(|x-y|)$ equals $1+h$ when $t\in[t_1,t_2]$ and $|x-y|= Y^{h}_y(t_1)-c_Y(t-t_1)$.  Then $\beta(t_1)=0$ and from $\tfrac d{dr}[\ln\psi(r)]\ge 2\zeta'/c_Y$ for $r\ge r_Y$ we obtain $\beta'(t)\ge 2\zeta'$ on $[t_1,t_2]$.  
Thus we have
\[
w_t\ge \Delta w +  \zeta' (w-h).
\]
From $w\ge h$, \eqref{3.0b}, and the hypothesis it follows that $w$ is a super-solution of \eqref{1.1} on $A$.
Since $u(t,x)\le 1\le w(t,x)$ when $t\in[t_1,t_2]$ and $|x-y|\ge Y^{h}_y(t_1)-c_Y(t-t_1)$, we obtain $w\ge u$ for $t\in[t_1,t_2]$ because $u(t_1,\cdot)\le w(t_1,\cdot)$.  Therefore $\gamma^{h}_y(t)\le e^{\beta(t) } \gamma^{h}_y(t_1)$ for  $t\in[t_1,t_2]$ and the result follows.

(iii) This is immediate from Lemma \ref{L.2.1}.
\end{proof}

 The following crucial lemma, which requires $u_t\ge 0$, will enable us to prove the claim in the remark after Lemma \ref{L.3.1}.  It essentially shows that $Y^{h}_y$ cannot decrease faster than at speed $c_Y$ ($<c_Z$) whenever $Z_y$ is much larger than $Y^{h}_y$.

\begin{lemma} \lb{L.3.2}
Let $d\le 3$.  There are ($\tht, h,f,u_0$-independent) $T_Y > 0$ and $\tau_Y\ge T_Y+1$ such that we have the following whenever 
\eqref{3.6a} holds on $\bbR^d$. 
If
\beq \lb{3.7}
Z_y(t_1+\tau_Y)\ge Y^{h}_y(t_1)
\eeq
for some $(t_1,y)\in[t_0,\infty)\times \bbR^d$ and $Y^{h}_y(t_1)-c_YT_Y\ge r_Y$, then
\beq \lb{3.8}
Y^{h}_y(t_1+T_Y)\ge Y^{h}_y(t_1)-c_YT_Y.
\eeq
%
%
\end{lemma}

{\it Remarks.}  
1. For $d\le 2$ we can take {\it any} $T_Y>0$.  For $d=3$ any large enough $T_Y$ works.  
\smallskip

2.  When $d\ge 4$, this result fails in general!  The same is true for $d\ge 1$ if $f$ satisfies (H) but we do not require  $f\in F(f_0,K,\tht,\zeta,\eta)$.
\smallskip

\begin{proof}
%
%
%
We split the proof in two cases, $d\le 2$ and $d=3$, due to Lemma \ref{L.2.2}.

{\bf Case $d\le 2$:}  We first claim that there is $\tau\ge 1$ such that if a solution $u\in[0,1]$ of \eqref{1.1} on $(0,\infty)\times\bbR^d$ satisfies $u_t\ge 0$ and $u(0,0)>\alpha_f(0)$, then $u(\tau,0)>1-\eps_0$.  Assume that for each $\tau=1,2,\dots$ there is some couple $f_\tau\in F(f_0,K,0,\zeta,\eta)$ and $u_\tau$ contradicting this statement with $(f,u)=(f_\tau,u_\tau)$. Then parabolic regularity shows that there is a sequence $\tau_j\to\infty$ such that $f_{\tau_j}$ and  $u_{\tau_j}$  converge locally uniformly on $\bbR^d\times[0,1]$ and on $(0,\infty)\times\bbR^d$  to some $f\in F(f_0,K,0,\zeta,\eta)$ and some solution $u\in[0,1]$ of \eqref{1.1} such that $u_t\ge 0$ and $\lim_{t\to\infty}u(t,0)\le 1-\eps_0$.  Moreover, $u_{\tau_j}(0,0)\ge \alpha_{f_{\tau_j}}(0)$  and $f_{\tau_j}\in F(f_0,K,0,\zeta,\eta)$ guarantee that $f(0,\cdot)\ge f_0(\cdot)+\eta\chi_{[u(0,0),\tht_0]}(\cdot)$.  But then $v(x):=\lim_{t\to\infty} u(t,x)$ satisfies \eqref{2.2} on $\bbR^d$ (so it is constant by Lemma \ref{L.2.2}) with  $f(0,v(0))>0$, a contradiction.

We now pick any $T_Y>0$ and apply this claim with the point $(0,0)$ shifted to $(t_1+T_Y,x)$, for any $x\in B_{Y^{h}_y(t_1)}(y)$.  If we  let $\tau_Y:=T_Y+\tau$, it follows from \eqref{3.7} that $u(t_1+T_Y,x)\le \alpha_f(x)$, and thus $u(t,x)\le \alpha_f(x)$ for all $(t,x)\in[t_1,t_1+T_Y]\times B_{Y^{h}_y(t_1)}(y)$.  Lemma \ref{L.3.1}(ii) now yields \eqref{3.8}.

{\bf Case $d=3$:}  This case is considerably more involved, due to the limitation in Lemma \ref{L.2.2}.  We postpone its proof 
until Section \ref{S5} in order to not interrupt the flow of the presentation.
\end{proof}

Note that in the case $d\le 2$, this result  holds even if \eqref{1.4b} is replaced by 
\[
\inf_{\substack{x\in\bbR^d \\  u\in[\alpha_f(x),\tht_0]}} \sup_{y\in B_R(x)} f(y,u)  \ge \eta
\]
 for some $R<\infty$, because we still obtain $f(x,v(x))>0$ for the constant function $v$ and some $|x|\le R$.  Theorems \ref{T.1.2}(i) and \ref{T.1.3} also extend accordingly.
 
The following result is at the heart of the proofs of our main results.

\begin{theorem} \lb{T.3.3}
Let $d\le 3$, let $f_0,K$ be as in (H), and let  $\eta>0$ and $\zeta\in(0,c_0^2/4)$. 

(i) There is $M>0$ such that if $\tht\ge 0$, $h$ satisfies \eqref{3.00}, $f\in F(f_0,K,\tht,\zeta,\eta)$,  $u_0\in[0,1]$ satisfies \eqref{3.6a}, and $u$ solves \eqref{1.1}, \eqref{1.2} on $(t_0,\infty)\times\bbR^d$, then for any $(t,y)\in (t_0,\infty)\times \bbR^d$ we have
\beq \lb{3.9}
Z_y(t)-Y^{h}_y(t) \le M + \left [Z_y(t_0)-Y^{h}_y(t_0) - \left(\frac{c_0}2 - \sqrt{\zeta'} \right) (t-t_0) \right]_+  .
\eeq
Moreover, for any $\eps\in(h,\tfrac 12)$ there is ($\tht,h,f,u_0$-independent) $\tau_\eps>0$, continuous and  non-increasing in $\eps>0$,  such that
\beq \lb{3.9a}
L_{u,\eps}(t) \le M_{\eps-h} + \left [ \sup_{y\in\bbR^d} \left( Z_y(t_0)-Y^{h}_y(t_0) \right) - \left(\frac{c_0}2 - \sqrt{\zeta'} \right) (t-t_0) \right]_+  
\eeq
for any $t\ge t_0+\tau_\eps$, with $M_{\delta}:=M+ c_Y' \tau_\delta+ \psi^{-1}(\delta^{-1})$.

(ii) If $\tht,h,M,M_{\delta},\tau_\eps,f,u$ are from (i) and $v\in[0,1]$ satisfies
\beq\lb{3.9b}
u(t-T,\cdot)-\frac \eps 2 \le v(t,\cdot) \le u(t+T,\cdot) + \frac \eps 2
\eeq
for some  $\eps\in(2h,\tfrac 12)$, $T\ge 0$, and $t\ge t_0+T+\tau_{\eps/2}$, then for such $t$,
\beq \lb{3.9c}
L_{v,\eps}(t) \le M_{\eps/2-h} + 3c_Y'T+ \left [ \sup_{y\in\bbR^d} \left( Z_y(t_0)-Y^{h}_y(t_0) \right) - \left(\frac{c_0}2 - \sqrt{\zeta'} \right) (t-t_0) \right]_+  
\eeq
\end{theorem}

{\it Remarks.} 1.  Recall also the bound from below in \eqref{3.2b}.
\smallskip

2.  $\tfrac 12 c_0-\sqrt{\zeta'}$ can be replaced by any $c<c_0-2\sqrt{\zeta'}$, and then $M,M_{\delta}$ also depend on $c$.
\smallskip

3. Obviously $M_{\delta}$ is continuous and decreasing  in $\delta>0$.
\smallskip

\begin{proof}
(i) Let us start with \eqref{3.9}.  Assume, without loss, that $y=0$ and $t_0=0$, and denote $Y^{h}_0=Y$ and $Z_0=Z$.  Recall that $c_Z=\tfrac 34 c_0+\tfrac 12 \sqrt{\zeta'}$ and $c_Y=\tfrac 14 c_0+\tfrac 32 \sqrt{\zeta'}$, so that $c_Z-c_Y= \tfrac 12 {c_0} - \sqrt{\zeta'}$, and then pick $c_Y',\tau_Z,r_Y,T_Y,\tau_Y$ as above (all these constants are independent of $\tht,h$).  

We can assume $Z(t)>0$ because otherwise the claim is obvious.  It is also sufficient to prove the claim for $t$ such that $Y(t)\ge c_Y'(\tau_Y+\tau_Z) + r_Y$ because then the result follows for all $t> 0$ after increasing M by $ c_Y'(\tau_Y+\tau_Z) +r_Y$.  This is because $Z$ (and also $Y$) is non-increasing due to \eqref{3.6a}.  We also note that $Y$ is then continuous by Lemma \ref{L.3.1}(i), while $Z$ is right-continuous and lower-semi-continuous by continuity of $u$ on $(0,\infty)\times\bbR^d$.  Finally, we can assume that $t>\tau_Y+\tau_Z$, because for $t\in(0,\tau_Y+\tau_Z]$ the estimate follows for any $M\ge  (c_Y'+\tfrac 12 c_0-\sqrt{\zeta'})(\tau_Y+\tau_Z) $ due to Lemma \ref{L.3.1}(i), $Z$ and $Y$ being non-increasing, and the assumption $Y(t)\ge c_Y'(\tau_Y+\tau_Z) + r_Y$.

We will now prove \eqref{3.9} assuming $t>\tau_Y+\tau_Z$, $Z(t)>0$ and $Y(t)\ge  c_Y'(\tau_Y+\tau_Z) + r_Y$, with 
\[
M:=c_Z\tau_Y+ c_Y'(\tau_Y+\tau_Z).
\]
Let $t_2$ be the smallest number in $[0,t-\tau_Y]$ such that $Z(t_1+\tau_Y) \ge Y(t_1)$ for all $t_1\in(t_2,t-\tau_Y)$.    Lower-semi-continuity of $Z(\cdot+\tau_Y)-Y(\cdot)$ now shows the following.  If $t_2=0$, then $Z(\tau_Y) \ge Y(0)$; if $t_2\in(0,t-\tau_Y)$, then $Z(t_2+\tau_Y) = Y(t_2)$; and  if $t_2=t-\tau_Y$, then $Z(t) \le Y(t-\tau_Y)$.  

If $t_2=0$, let $N:=\lfloor t/T_Y \rfloor$.  Applying Lemma \ref{L.3.1}(i) once and then Lemma \ref{L.3.2} $N$ times, we obtain using  $Y(t)- c_Y'T_Y \ge r_Y$ (recall that $\tau_Y>T_Y$),
\[
Y(t)\ge Y(NT_Y) - c_Y'T_Y \ge Y(0) - Nc_YT_Y - c_Y'T_Y \ge Y(0) - c_Y t - c_Y'T_Y.
\]
On the other hand, Lemma \ref{L.3.1}(iii) and $t\ge \tau_Y+\tau_Z$ yield  
\[
Z(t) \le Z(\tau_Y) - c_Z(t-\tau_Y) \le Z(0) -c_Z t + c_Z \tau_Y
\]
(notice that $Z(\tau_Y)-c_Z(t-\tau_Y)>0$ because otherwise $Z(t)=0$ by Lemma~\ref{L.3.1}(iii)).
Thus 
\[
Z(t)-Y(t) \le c_Z\tau_Y+ c_Y'T_Y + Z(0)-Y(0) - (c_Z-c_Y) t \le M + Z(0)-Y(0) - (c_Z-c_Y) t.
\]

If $t_2\in(0,t-\tau_Y-\tau_Z)$, then let $N:=\lfloor (t-t_2)/T_Y \rfloor$.  An identical argument now yields
\[
Y(t)\ge Y(t_2) - c_Y(t-t_2)-c_Y'T_Y
\]
and
\[
Z(t) \le Z(t_2+\tau_Y) - c_Z(t-t_2-\tau_Y).
\]
Thus $Z(t_2+\tau_Y) = Y(t_2)$ yields
\[
Z(t)-Y(t) \le c_Z\tau_Y+ c_Y'T_Y + Z(t_2+\tau_Y)-Y(t_2) - (c_Z-c_Y) (t-t_2) \le c_Z\tau_Y+ c_Y'T_Y \le M .
\]

If $t_2\in[t-\tau_Y-\tau_Z,t-\tau_Y]$, then $Z(t_2+\tau_Y) \le Y(t_2)$, so that
\[
Z(t)-Y(t)\le Z(t_2+\tau_Y) -Y(t) \le Y(t_2) - Y(t) \le c_Y'(\tau_Y+\tau_Z) \le M.
\]
by Lemma \ref{L.3.1}(i).  The proof of \eqref{3.9} is finished.

Let us now turn to \eqref{3.9a} and again assume $t_0=0$. Let $c:=\tfrac 12 c_0$, and for any $\eps\in(h,\tfrac 12)$ let $\tau_\eps:=\tau+1$, with $\tau$ from Lemma~\ref{L.2.1} (this can obviously be chosen continuous and non-increasing in $\eps>0$).  For $t\ge \tau_\eps$, let $x\in\Omega_{u,\eps}(t)$.  Then $Y^{h}_x(t)\le \psi^{-1}((\eps-h)^{-1})$, so
\[
Y^{h}_x(t-\tau)\le \psi^{-1}((\eps-h)^{-1})+c_Y'\tau \le \psi^{-1}((\eps-h)^{-1})+c_Y'\tau_\eps
\]
 by Lemma \ref{L.3.1}(i), and  \eqref{3.9} gives
\[
Z_x(t-\tau) \le \psi^{-1}((\eps-h)^{-1})+c_Y'\tau_\eps + M + \left [\sup_{y\in\bbR^d} (Z_y(0)-Y^{h}_y(0)) - \left(\frac{c_0}2 - \sqrt{\zeta'} \right) (t-\tau) \right]_+ .
\]
Lemma \ref{L.2.1} with $t_1:=t-\tau$ now shows that there is $y$ with $u(t,y)\ge 1-\eps$ and 
\[
|y-x| \le \psi^{-1}((\eps-h)^{-1})+c_Y'\tau_\eps + M + \left [\sup_{y\in\bbR^d} (Z_y(0)-Y^{h}_y(0)) - \left(\frac{c_0}2 - \sqrt{\zeta'} \right) (t-\tau) \right]_+ - \frac {c_0}2\tau.
\]
This yields \eqref{3.9a} for $t\ge\tau_\eps$ (recall that $t_0=0$ and $\tau\ge 0$).

(ii) Again assume $t_0=0$.  An identical argument to the last one shows that if $t\ge T+\tau_{\eps/2}$ and $x\in\Omega_{u,\eps/2}(t+T)$ (the latter holds when $v(t,x)\ge \eps$), then 
\[
Y^{h}_x(t-T-\tau)\le \psi^{-1} \left( \left( \frac\eps 2-h \right)^{-1} \right)+c_Y'(2T+\tau_{\eps/2}),
\]
and ultimately that there is $y$ with $u(t-T,y)\ge 1-\tfrac \eps 2$ (which yields $v(t,y)\ge 1-\eps$) and 
\[
|y-x| \le \psi^{-1} \left( \left( \frac\eps 2-h \right)^{-1} \right)+c_Y'(2T+\tau_{\eps/2}) + M + \left [\sup_{y\in\bbR^d} (Z_y(0)-Y^{h}_y(0)) - \left(\frac{c_0}2 - \sqrt{\zeta'} \right) (t-T-\tau) \right]_+ - \frac {c_0}2\tau.
\]
This proves \eqref{3.9c} because $\tfrac 12 c_0-\sqrt{\zeta'}\le c_Y'$.
\end{proof}

Before moving onto the proofs of our main results, we state an important corollary of Theorem \ref{T.3.3}.  For a solution $u$  of \eqref{1.1} on $(t_0,\infty)\times\bbR^d$ and for $t\ge t_0$, we let
\beq\lb{3.9d}
\Lambda^h_u(t) := \sup_{y\in\bbR^d} \left(Z_y(t)-Y^{h}_y(t) \right) \qquad(\le\infty)
\eeq
for $h$ from \eqref{3.00}.  Then $\Lambda^h_u$ is non-increasing and right-continuous in $h$, by  definition of $Y^h_y$.

Notice that if $\Lambda^h_u$ is finite, then it controls $L_{u,\eps}$ for any $\eps\in(h,\tfrac 12)$.
Indeed, the argument proving \eqref{3.9a} from \eqref{3.9} applies to any $u$ (even if $u_t\not\ge 0$) and, with the notation from Theorem \ref{T.3.3}, yields for $t\ge t_0+\tau_\eps$,
\beq \lb{3.9e}
L_{u,\eps}(t) \le M_{\eps-h} - M + \Lambda^h_{u}(t-\tau_\eps+1).
\eeq

For ignition  $f$ we also have the opposite direction.  For  any $(t,y)\in(t_0,\infty)\times\bbR^d$ and $h\in\left(0,\min\left\{ \tht(c_0^2-4\zeta)(c_0^2+4\zeta)^{-1},\frac\eta{4K},\eps_0 \right\} \right]$, we have
\[
\sup_{|x-y|<Y^{h}_y(t)} u(t,x)> h
\]
because otherwise we would have $u(t,\cdot)\le  h+\gamma\psi(|\cdot|)$ for some $\gamma<\psi(Y^{h}_y(t))^{-1}$, contradicting the definition of $Y^{h}_y(t)$.  But then $Z_y(t)\le Y^{h}_y(t)+L_{u,h}(t)$ by $h\le\eps_0$, so 
\beq \lb{3.9f}
\Lambda^h_u(t)\le L_{u,h}(t).
\eeq

We now have the following result for entire solutions with $u_t\ge 0$.

\begin{corollary} \lb{C.3.3a}
Let $d\le 3$, let $f_0,K,$ and $\tht\ge 0$ be as in (H), and let  $\eta>0$, $\zeta\in(0,c_0^2/4)$, and  $f\in F(f_0,K,\tht,\zeta,\eta)$. Assume that $u\in[0,1]$ solves \eqref{1.1} and satisfies $u_t\ge 0$ on $\bbR\times\bbR^d$. 

(i) If $h$ is as in \eqref{3.00} and $\limsup_{t\to-\infty} \Lambda^h_u(t)<\infty$, then in fact $\sup_{t\in\bbR} \Lambda^h_u (t) \le M$ and  $\sup_{t\in\bbR} L_{u,\eps}(t) \le M_{\eps-h}$ for any $\eps\in(h,\tfrac 12)$  (here $M,M_{\delta}$ are from Theorem \ref{T.3.3}).  We  also have
\beq \lb{3.9g}
\inf_{ u(t,x)\in[\eps,1-\eps]} u_t(t,x) \ge \mu_{\eps,M_{\eps/2-h}}
\eeq
 for any $\eps\in(2h,2\eps_0]$, where $\mu_{\eps,\ell}>0$ is the $\inf$ in \eqref{1.7b}.
 
 (ii) If $\tht>0$ and $u$ has bounded width, then $\sup_{t\in\bbR} \Lambda^0_u (t) \le M$ (and so (i) holds with $h=0$ and $\eps\in(0,\tfrac 12)$).  Moreover, if a pure ignition $f_1$ 
 satisfies \eqref{1.4e}, then 
$u$ propagates with global mean speed in $[c_0,c_1]$, with $\tau_{\eps,\delta}$ in Definition \ref{D.1.1b} depending only on  $\delta,f_1,\eps,f_0,K,\zeta,\eta$.
\end{corollary}

{\it Remark.}  Recall that $M,M_{\eps},\mu_\eps$ depend on $f_0,K,\zeta,\eta$ (and $\eps$)  {\it but not on $\tht,h,f,u$}.    This and Theorem \ref{T.1.5}(ii) will be the key to the independence of the bounds in Theorem \ref{T.1.2}(i) on $u_0$.

\begin{proof}
(i) The first claim is immediate from \eqref{3.9} after letting $u_0(x):=u(t_0,x)$ and then sending $t_0\to-\infty$.  The second then follows from \eqref{3.9e}, and the third  from Lemma \ref{L.8.1}(iii) with $\tht=0$, applied to $u$ shifted in time by $1-t$.

(ii) For any $h\in\left(0,\min\left\{ \tht(c_0^2-4\zeta)(c_0^2+4\zeta)^{-1},\frac\eta{4K},\eps_0 \right\} \right]$, \eqref{3.9f}
 and (i) show $\sup_{t\in\bbR} \Lambda^h_u (t) \le M$, and right-continuity of $\Lambda^h_u$ in $h$ then yields $\sup_{t\in\bbR} \Lambda^0_u (t) \le M$.  The second claim now follows from Lemma \ref{L.2.1a} with $\delta/2$ instead of $\delta$, using the bound $L_{u,\eps'}(t) \le M_{\eps'}$ for  $\eps'$ from that lemma (which holds by \eqref{3.9e} with $h=0$). Indeed, we only need to take $\tau_{\eps,\delta}\ge \max\{2\delta^{-1}M_{\eps'},\tau\}$ in Definition \ref{D.1.1b}, with $\tau$ from Lemma \ref{L.2.1a}.  
\end{proof}

{\it Open problem.} It is an interesting question whether there is a transition solution $u$ satisfying all the hypotheses of Corollary \ref{C.3.3a}(ii), except of the hypothesis of bounded width, such that $\Lambda^0_u$ is unbounded (cf.~Open problem 2 after Theorem \ref{T.1.5}). It is obvious from \eqref{3.9f} and \eqref{3.9a} that in that case one would have $\liminf_{t\to-\infty} |t|^{-1} \Lambda^h_u(t)>0$.

\section{Proof of  Theorem \ref{T.1.3}(i)} \lb{S4}

We can  assume $u_0\not\equiv 0,1$ because then the result holds trivially.  
As in Section \ref{S3}, all constants  will depend on $f_0,K,\zeta,\eta$ (but not on $\tht$ from (H), unless explicitly noted).

The second claim in \eqref{1.7} follows immediately from the first.  Indeed: it is sufficient to prove it for $\eps\in(0,2\eps_0]$;  if $\mu_{\eps,\ell}>0$ is the $\inf$ in \eqref{1.7b} for such $\eps$ and $\ell_\eps,T_\eps$ are from the first claim in \eqref{1.7}, then the second claim follows  with $m_\eps:= \mu_{\eps,\ell_{\eps/2}}$ and $T_\eps$ replaced by $T_{\eps/2}+1$, after applying Lemma \ref{L.8.1}(iii) to $u$ shifted in time by $-(t_0+T_{\eps/2})$.  

Similarly, the claim about global mean speed also follows from the first claim in \eqref{1.7}.  Indeed, if $\eps',\tau$ are from Lemma \ref{L.2.1a} with $\delta/2$ instead of $\delta$, then that lemma shows that we only need $T_{\eps,\delta}\ge T_{\eps'}+1$ and   $\tau_{\eps,\delta}\ge\max\{2\delta^{-1}\ell_{\eps'},\tau\}$ in  Definition \ref{D.1.1b}.

We are left with proving the first claim in \eqref{1.7} (which also proves that $u$ has a bounded width).  We will do so with $\ell_\eps:=M_{\eps/2}$ from Theorem \ref{T.3.3}.  We define $Z_y,Y^{h}_y$ as in Section \ref{S3} and split the proof into two cases.

{\bf Case $d=3$:}  Let $\eps':=1-\eps_0$ (which depends on $f_0,K$) and given any $\eps\in(0,\tfrac 12)$, 
let $h:=\min\left\{ \tht(c_0^2-4\zeta)(c_0^2+4\zeta)^{-1},\frac\eta{4K},\eps_0,\tfrac \eps 2 \right\}$.  
The argument which proves \eqref{3.9f} now shows
$Z_y(t_0)\le Y^{h}_y(t_0)+L_{u,h,\eps'}(t_0)$ for each $y\in\bbR^3$. Hence the  right-hand side of \eqref{3.9a} equals $M_{\eps-h}\,(\le M_{\eps/2})$ for all large enough $t$ and we are done.  

{\bf Case  $d\le 2$:}  First, there is $\tau\ge 1$ such that if a solution $u\in[0,1]$ of \eqref{1.1} on $(t_0,\infty)\times\bbR^d$ with $f$ as in the theorem satisfies $u_t\ge 0$ and $u(t_0,x)\ge\eps'$, then $u(t_0+\tau,x)>1-\eps_0$. This is proved just as a similar claim in the proof of Lemma \ref{L.3.2}.  

Define now $Z_y'$ as $Z_y$ but with $\eps'$ in place of $1-\eps_0$, and given any $\eps\in(0,\tfrac 12)$, 
let $h:=\min\left\{ \tht(c_0^2-4\zeta)(c_0^2+4\zeta)^{-1},\frac\eta{4K},1-\eps',\tfrac \eps 2 \right\}$. 
The argument which proves \eqref{3.9f} now shows $Z_y'(t_0)\le Y^{h}_y(t_0)+L_{u,h,\eps'}(t_0)$, and then the previous paragraph and Lemma \ref{L.3.1}(i) yield
\[
Z_y(t_0+\tau)\le Y^{h}_y(t_0)+L_{u,h,\eps'}(t_0)\le Y^{h}_y(t_0+\tau)+c_Y'\tau+L_{u,h,\eps'}(t_0).
\]
This holds for all $y\in\bbR^d$, so we conclude as in the first case.

The proof of Theorem \ref{T.1.3}(i) is finished.
\smallskip


\begin{proof}[Proof of Remark 2 after Theorem \ref{T.1.3}]  The second claim in \eqref{1.7} follows from the first as above.  To prove the first claim in \eqref{1.7}, again consider two cases.

{\bf Case $d=3$:}  Let $\eps':=1-\eps_0$. The hypothesis and \eqref{2.5} for $k=0$ and $\zeta'$ in place of $\zeta$ imply
\[
C:=\sup_{\eps\in(0,1), r\ge 0} \eps \psi(r)^{-1}\psi(r+L_{u,\eps,\eps'}(t_0))<\infty.
\]
Assume that $u_0\not\equiv 0$ because otherwise the result holds trivially.
By the definition of $Y^{h}_y$, for any $y\in\bbR^3$, there is $x\in\bbR^3$ such that 
\[
u(t_0,x)=\psi(Y^{0}_y(t_0))^{-1}\psi(|x-y|) \qquad (=:\eps>0).
\]
Then there is $x'\in B_{L_{u,\eps,\eps'}(t_0)}(x)$ with $u(t_0,x')\ge \eps'$ ($=1-\eps_0$), and we have
\[
\psi(|x'-y|)\le \psi(|x-y|+L_{u,\eps,\eps'}(t_0)) \le C\eps^{-1}\psi(|x-y|) = C\psi(Y^{0}_y(t_0)).
\]
Since $C$ is independent of $y$, this means $\sup_{y\in\bbR^3} (Z_y(t_0)-Y^{0}_y(t_0))<\infty$.  We conclude as in the ignition case, using \eqref{3.9a}.

{\bf Case  $d\le 2$:}  The argument from the first case shows $\sup_{y\in\bbR^d} (Z_y'(t_0)-Y^{0}_y(t_0))<\infty$, with  $Z_y'$ from the ignition case $d\le 2$.  As in the ignition case $d\le 2$, and with the same $\tau$, we obtain $\sup_{y\in\bbR^d} (Z_y(t_0+\tau)-Y^{0}_y(t_0+\tau))<\infty$ and the result follows as before.  

Finally, the first inclusion of the claim about global mean speed follows from the first claim in \eqref{1.7} as in the ignition case because the first inclusion in Lemma \ref{L.2.1a} holds also for monostable $f$.  The second inclusion in  Definition \ref{D.1.1b}, with $c':=c_Y'$, follows from  $\Lambda^0_u(t)\le M$ (which holds for all large enough $t$ by \eqref{3.9}) and \eqref{3.3}.
\end{proof}


\section{Proofs of  Theorems \ref{T.1.2}(i) and \ref{T.1.3}(ii)} \lb{S4a}

As in Section \ref{S3}, all constants  will depend on $f_0,K,\zeta,\eta$ (but not on $\tht$ from (H), unless explicitly noted).

Note that the claim about global mean speed follows in both cases from the first claim in \eqref{1.7} as in the proof of Theorem \ref{T.1.3}(i).  Since bounded width also follows from the first claim in \eqref{1.7}, we are therefore left with proving \eqref{1.7} in both cases.

Let us start with the (easier to prove) analogous results for monostable reactions from Remark 4 after Theorem \ref{T.1.2} and Remark 3 after Theorem \ref{T.1.3}.

\begin{proof}[Proof of Remark 4 after Theorem \ref{T.1.2}]
We can assume without loss that $t_0=0$.
The idea is to construct $w_0$ such that 
\beq \lb{4.2}
\Delta w_0(\cdot) + f(\cdot,w_0(\cdot))\ge 0
\eeq
and the solution $w$ to \eqref{1.1} with $w(0,x)=w_0(x)$  satisfies $w(\tau,\cdot)\ge u_0(\cdot)$ and $u(\tau ,\cdot)\ge w_0(\cdot)$ for some $\tau >0$.  Then  $u$ will satisfy
\beq \lb{4.2a}
w(t-\tau ,\cdot) \le u(t,\cdot) \le w(t+\tau,\cdot)
\eeq
for $t\ge \tau$.  Since $w_t\ge 0$, Theorem \ref{T.3.3}(ii) for $w,u$ in place of $u,v$ will now do the trick.  

Let us first consider \eqref{1.6a}, and  sssume without loss $e=(1,0,\dots,0)$.
Let  $s_0>0$ be such that there is a smooth, even, $2s_0$-periodic $C^2$ function $U:\bbR\to[0,\tfrac 12(1+\tht_0)]$ satisfying $U'' + f_0(U)>0$ on $\bbR$, $U(0)=\tfrac 12(1+\tht_0)$, $U(s_0)=0$, and $U'<0$ on $(0,s_0)$ (then obviously $U'(0)=U'(s_0)=0$).  Such $U$ is obtained by perturbing the solution of $\til U'' + f_0(\til U)=0$ with $\til U(0)=\tfrac 12(1+\tht_0)$ and $\til U'(0)=0$.  The latter satisfies $\til U'<0$ on some interval $(0,\til s_0]$ with $\til U(\til s_0)=0$ because multiplying the ODE by $\til U'$ and integrating yields 
\[
\til U'(s)^2 = \til U'(0)^2 + 2\int_{\til U(s)}^{\til U(0)} f_0(u)du = 2\int_{\til U(s)}^{(1+\tht_0)/2} f_0(u)du >0
\] 
as long as $\til U(s)>0$ (notice that $\til U''(0)<0$).  Thus we can perturb $\til U$ to obtain the desired $U$, with $s_0$ near $\til s_0$. Then $U,s_0$ depend only on $f_0$, and we define
\beq\lb{4.1}
W(s):= 
\begin{cases}
\tfrac 12(1+\tht_0) & s\le R_2,
\\ U(s-R_2) & s\in(R_2,R_2+s_0),
\\ 0 & s\ge R_2+s_0.
\end{cases}
\eeq
Note that  
\[
\inf_{s<R_2+s_0} [W''(s)+f_0(W(s))] = \inf_{s\in\bbR} [U''(s)+f_0(U(s))] >0.
\]  

Then $w_0(x):=W(x_1)$ satisfies \eqref{4.2}, and $w(\tau,\cdot)\ge u_0(\cdot)$ follows for some $(f_0,\eps_2)$-dependent $\tau$, from $\eps_2>0$ and  the second claim in Lemma \ref{L.2.1}.  Similarly, $u(\tau,\cdot)\ge w_0(\cdot)$ follows for some $(f_0,R_2-R_1,\eps_1)$-dependent $\tau$ from  the second claim in Lemma \ref{L.2.1} (which holds with $\tfrac 12(1+\tht_0)$ replaced by $\tht_0+\eps_1$ when $\eps_1>0$, and with $R=R(f_0,\eps_1)$ \cite{AW}).

Thus  $w$  satisfies \eqref{4.2a} for all $t\ge \tau $.  Let us increase $\tau$ so that 
\[
w(\tau,\cdot)\ge (1-\eps_0) \chi_{\{x\,|\,x\cdot e<R_2\}}(\cdot).
\]
This makes $\tau$ also depend on $K$, and then Lemma \ref{L.3.1}(i) applied to $w$ yields 
\beq\lb{4.1cc}
\Lambda_w^0(\tau)\le s_0+c_Y'\tau
\eeq
because $Y_y^0(0)\ge y_1-(R_2+s_0)$ if $Y_y^0$ is defined with respect to $w$.
With $M_\eps,\tau_\eps$ from Theorem~\ref{T.3.3},  let $\ell_\eps:=M_{\eps/2}+3c_Y'\tau $ and 
\beq\lb{4.1b}
T_\eps:=  2\tau  +  \tau_{\eps/2} + \left(\frac{c_0}2 - \sqrt\zeta \right)^{-1} (s_0+c_Y'\tau) .
\eeq
Then Theorem \ref{T.3.3}(ii) with $w,u$ in place of $u,v$ gives $L_{u,\eps}(t)\le\ell_\eps$ for $t\ge T_\eps$. The proof in the case \eqref{1.6a} is finished.

Let us now assume \eqref{1.5a} as well as $x_0=0$ without loss, and first also assume that $\sup u_0<1$.  
We can also assume without loss that (with $U$ as above) 
\beq\lb{4.1a}
R_2\ge \frac{(d-1)\|U'\|_\infty}{\inf_{s\in\bbR} [U''(s)+f_0(U(s))]}.
\eeq
The result now holds for any $R_1\ge R(f_0,\eps_1)$, where the latter is from the argument above so that the conclusion of Lemma \ref{L.2.1} still holds.
Indeed, this time we let $w_0(x):=W(|x|)$, which also satisfies \eqref{4.2} due to \eqref{4.1a}.  As above, we obtain \eqref{4.2a} for some $\tau>0$, and  then again $L_{u,\eps}(t)\le\ell_\eps$ for $t\ge T_\eps$.

Finally,  assume \eqref{1.5a} with $x_0=0$, and $\sup u_0= 1$.  We let again \eqref{4.1a} and $w_0(x):=W(|x|)$, but now $w(\tau,\cdot)\not\ge u_0(\cdot)$ for all $\tau\ge 0$.  Solve \eqref{1.1}, \eqref{1.2} and replace $u_0$ by $u(1,\cdot)$.  It is obviously sufficient to prove this claim for the new $u_0$.  This now satisfies 
\[
u_0(x) \le \min \left\{ 1-\eps_2,\frac{ |B_1(0)| R_2^d}{(4\pi)^{d/2}} e^{K } e^{ -\max\{|x|-R_2,0\}^2/4} \right\},
\]
by the comparison principle, for some $(K,R_2)$-dependent $\eps_2>0$.  Since $w(\tau,\cdot)$ converges locally uniformly to $1$ as $\tau\to\infty$, $w_t\ge 0$, and $w(\tau,x)\sim e^{-|x|^2/4\tau}$ as $|x|\to\infty$ by the heat equation asymptotics, we again obtain $w(\tau,\cdot)\ge u_0(\cdot)$ for some $\tau$.  The rest of the argument is as before.
\end{proof}

\begin{proof}[Proof of Remark 3 after Theorem \ref{T.1.3}]
The first claim in \eqref{1.7} is proved as above, now with $u$ playing the role of $w$ because $u_t\ge 0$.  Indeed, assume $t_0=0$ and let $\tau<\infty$ be such that $s:=\sup_{y\in\bbR^d}(Z_y(\tau)-Y^0_y(\tau))<\infty$ (which exists by the proof of Remark 2 after Theorem \ref{T.1.3}).
With $M_\eps,\tau_\eps,\ell_\eps$ from the previous proof, let $T_\eps$ be from \eqref{4.1b} but with $s_0+c_Y'\tau$ replaced by $s$.  
Then again $L_{v,\eps}(t)\le\ell_\eps$ for $t\ge T_\eps$ by Theorem \ref{T.3.3}(ii), as above.

The  global mean speed claim is proved as in the proof of Remark 2 after Theorem \ref{T.1.3}, this time using $\Lambda^0_v(t)\le \Lambda^0_u(t-\tau)+2c_Y'\tau \le M+ 2c_Y'\tau$ (which again holds for all large $t$).
\end{proof}

The proof of \eqref{1.7} in Theorem \ref{T.1.2}(i) resp.~Theorem \ref{T.1.3}(ii) is similar but a little more involved.  To show that $\ell_\eps$ is independent of $R_1,R_2,\eps_1,\eps_2$ resp.~$\tau$, as well as to obtain the second claim in \eqref{1.7}, we will need to use Theorem \ref{T.1.5}.  In addition, the exponential tails of the initial data in Theorem \ref{T.1.2}(i) will be handled by constructing appropriate super-solutions and obtaining inequalities as in  \eqref{3.9b} instead of \eqref{4.2a}.

We will start with proving the result for general solutions $u$ which (essentially) lie between two time-translates of a solution $w$ with initial datum satisfying \eqref{4.2}.  The bounds in this result will, in fact, be independent of $u,w$ for large $t$ as long as the number $\Lambda^h_{w}(0)$, defined in \eqref{3.9d}, is finite for each small enough $h>0$.  

\begin{theorem} \lb{T.5.1}
Let $d\le 3$, let $f_0,K$ be as in (H), and let $\eta>0$ and $\zeta\in(0,c_0^2/4)$. For any $\eps'\in(0,\tfrac 12)$, there are $\ell_{\eps'},m_{\eps'}\in(0,\infty)$ such that if  $\tht>0$, $\lambda:(0,\tfrac 12)\to(0,\infty)$ is left-continuous and non-increasing, $\tau <\infty$, and $\nu:(0,\infty)\to[0,\infty)$ satisifies $\lim_{t\to\infty}\nu(t)=0$, then there is $T_{\eps',\tht,\lambda,\tau ,\nu}<\infty$
such that  the following holds.  If $f\in F(f_0,K,\tht,\zeta,\eta)$ and $u,w\in[0,1]$ are solutions of \eqref{1.1} on $(0,\infty)\times\bbR^d$ with $w_0(\cdot):=w(0,\cdot)$ satisfying \eqref{4.2}, with 
$\Lambda^h_{w}(0)\le \lambda(h)$ for all $h\in \left( 0,\min\left\{ \tht(c_0^2-4\zeta)(c_0^2+4\zeta)^{-1},\frac\eta{4K},\eps_0 \right\} \right]$, and with 
\beq\lb{5.2}
w(t-\tau ,\cdot)-\nu(t) \le u(t,\cdot) \le w(t+\tau ,\cdot)+\nu(t)
\eeq
for each $t>\tau $, then 
\beq\lb{5.3}
\sup_{t\ge T_{\eps',\tht,\lambda,\tau ,\nu}} L_{u,\eps'}(t)\le \ell_{\eps'} \qquad\text{and}\qquad
\inf_{\substack{t\ge T_{\eps',\tht,\lambda,\tau ,\nu} \\ u(t,x)\in[\eps',1-\eps']}} u_t(t,x) \ge m_{\eps'}.
\eeq
\end{theorem}

{\it Remark.}  We stress that $\ell_{\eps'},m_{\eps'}$ are independent of $f,u$ as well as of $\tht,\lambda,\tau,\nu$.


\begin{proof}
Let  $h_0:= \min\left\{ \tht(c_0^2-4\zeta)(c_0^2+4\zeta)^{-1},\frac\eta{4K},\eps_0 \right\}>0$.  For $\eps\in(0,4h_0]$, and with $M_\delta,\tau_\eps$ from Theorem \ref{T.3.3}, let 
$T(\eps)\ge \tau +\tau_{\eps/2}$ be such that $\sup_{t\ge T(\eps)} \nu(t)\le \tfrac \eps 2$, 
and define
 $L(\eps):=  M_{\eps/4} + 3c_Y'\tau  +\lambda(\tfrac \eps 4)$.  For $\eps\in(4h_0,\tfrac 12)$ let $T(\eps):=T(4h_0)$ and $L(\eps):=L(4h_0)$.

Then for any $\eps\in(0,4h_0]$, by Theorem \ref{T.3.3}(ii) with $h:=\tfrac\eps 4$, 
\beq\lb{5.4a}
L_{u,\eps}(t)\le L(\eps) \qquad\text{for $t\ge T(\eps)$}
\eeq
(and this also holds for $\eps\in(4h_0,\tfrac 12)$ because $L_{u,\eps}(t)$ is non-increasing in $\eps$).

Let us now prove the first claim in \eqref{5.3}, with $\ell_{\eps'} :=M_{\eps'/2}+1$.  If there is no such $T_{\eps',\tht,\lambda,\tau ,\nu}$,  then there is a sequence $(f_n,u_n, w_n,t_n,x_n)$ with $f_n,u_n,w_n$ satisfying the hypotheses of the theorem, $\lim_{n\to\infty} t_n=\infty$, $u_n(t_n,x_n)\in [\eps',1-\eps']$ and 
\beq\lb{5.5}
 \inf_{u_n(t_n,y)\ge 1-\eps'} |y-x_n| > \ell_{\eps'}.
\eeq
After shifting $f_n$ by $(-x_n,0)$ and  $u_n$ by $(-t_n,-x_n)$, and then applying Lemma \ref{L.8.1}(ii) (with $-t_n+T(\eps)$ in place of $t_n(\eps)$, using \eqref{5.4a}),
we obtain new $(f,u)\in S_{-\infty,L}(f_0,K,\tht)$ such that $u(0,0)\in [\eps',1-\eps']$ and $L_{u,\eps'/2}(0)\ge \ell_{\eps'}>M_{\eps'/2}$.  We also have $f\in F(f_0,K,\tht,\zeta,\eta)$ because that set is closed under locally uniform limits. 

Thus $(f,u)\in S_L(f_0,K,\tht)$ since $u\not\equiv 0,1$.  Then
Theorem \ref{T.1.5}(ii)  shows  $u_t\ge 0$, because bounded width of $u$ and Lemma \ref{L.2.1} immediately show that $u$ propagates with a positive global mean speed.  But then $L_{u,\eps'/2}(0)>M_{\eps'/2}$ yields a contradiction with Corollary \ref{C.3.3a}(ii,i) (with $h=0$).  The first claim in \eqref{5.3} is proved.


The second claim is proved similarly with $m_{\eps'}:=\tfrac 12\mu_{\eps',M_{\eps'/2}}$ for $\eps\in(0,2\eps_0]$, where $\mu_{\eps,\ell}>0$ is the $\inf$ in \eqref{1.7b} (then it also holds with $m_{\eps'}:=m_{2\eps_0}$ for $\eps'\in(2\eps_0,\tfrac 12)$).
Non-existence of $T_{\eps',\tht,\lambda,\tau ,\nu}$ again yields a sequence $(f_n,u_n,w_n,t_n,x_n)$ with $f_n,u_n,w_n$ satisfying the hypotheses of the theorem, $\lim_{n\to\infty} t_n=\infty$, $u_n(t_n,x_n)\in [\eps',1-\eps']$ and $(u_n)_t(t_n,x_n)<m_{\eps'}$.   We again obtain new $(f,u)\in S_L(f_0,K,\tht)$ such that $f\in F(f_0,K,\tht,\zeta,\eta)$, 
$u_t\ge 0$, as well as $u(0,0)\in [\eps',1-\eps']$, and  $u_t(0,0)\le m_{\eps'}<\mu_{\eps',M_{\eps'/2}}$.  This contradicts Corollary \ref{C.3.3a}(ii,i) (with $h=0$), and the second claim in \eqref{5.3} is also proved.
\end{proof}

Recall that in the proof of  Theorem \ref{T.1.3}(i) we obtained $\Lambda^h_u(t_0+T)\le c_Y'T+L_{u,h,\eps'}(t_0)$ for all $h\in \left( 0,\min\left\{ \tht(c_0^2-4\zeta)(c_0^2+4\zeta)^{-1},\frac\eta{4K},\eps_0 \right\} \right]$, with $T=0$ if $d=3$ and some $T>0$ if $d\le 2$.  If we thus let $\lambda(h):=c_Y'T+\inf_{h'\in(0,h)}L_{u,h',\eps'}(t_0)$ (which is left-continuous) and $\nu\equiv 0$, then \eqref{1.7} in Theorem \ref{T.1.3}(ii)  immediately follows from Theorem \ref{T.5.1} with $u,v$ in place of $w,u$ and time shifted by $-(t_0+T)$.

Hence we are left with proving \eqref{1.7} in Theorem \ref{T.1.2}(i).  As in the proof of Remark 4 after Theorem \ref{T.1.2}, we will start with assuming \eqref{1.6}, and also without loss that $t_0=0$, $e=(1,0,\dots,0)$, as well as $\eps_2\le c_0/4$
(recall that $u\in[0, 1]$).  
We again let $w$ solve \eqref{1.1} with $w(0,x)=W(x_1)$, where  $W$ is from \eqref{4.1}.
As before, $w_t\ge 0$ and we have $u(\tau,\cdot) \ge w(0 ,\cdot)$  provided $\tau $ is large enough (depending on $f_0,R_2-R_1,\eps_1$).  This yields the first inequality in \eqref{5.2}, with $\nu\equiv 0$.  


To obtain the second inequality in \eqref{5.2}, we define $\beta(t):= \tau -e^{-\eps_2^2 t}$ and 
\beq \lb{5.9b}
v(t,x):=w(t+\beta(t),x)+e^{\eps_2^2 t - \eps_2(x_1-R_2)}
\eeq
for some large $\tau $ to be determined later.  We then have for $t>0$,
\beq \lb{5.6}
v_t-\Delta v-f(x,v) = f(x,w(t+\beta(t),x))-f(x,v) + \eps_2^2 e^{-\eps_2^2 t} w_t(t+\beta(t),x),
\eeq
where we extend $f$ so that $f(x,u)\le 0$ for $u\ge 1$ (cf.~\eqref{1.24}).  

We want to show that $v$ is a super-solution of \eqref{1.1}, that is, the right hand side of \eqref{5.6} is $\ge 0$ for $t>0$ and $x\in\bbR^d$.  When $w(t+\beta(t),x)\ge 1-\tht$, then $f(x,w(t+\beta(t),x))\ge f(x,v(t,x))$ by the hypotheses on $f$ and $w\le 1$, so this is indeed the case.

Let $\ell_{\tht/4},m_{\tht/2}$ be  from Theorem \ref{T.5.1} (i.e., with $\eps':=\tfrac\tht 4$ and $\eps':=\tfrac\tht 2$).  We now let $\tau $ be large enough so that $w(t+\tau -1,x)\ge 1-\tht$ whenever $t\ge 0$ and 
\beq \lb{5.7}
x_1\le \frac {c_0}2 t + \frac 1{\eps_2} \log \max \left\{ \frac K {\eps_2^2 m_{\tht/2}}, \frac 2\tht \right\} +R_2,
\eeq
and also that
\beq\lb{5.7a}
\sup_{t\ge 0} L_{w,\tht/4}(t+\tau-1)\le\ell_{\tht/4} \qquad\text{and}\qquad \inf_{\substack{t\ge 0 \\ w(t+\tau -1,x)\in[\tht/2,1-\tht/2]}} w_t(t+\tau -1,x) \ge m_{\tht/2}.
\eeq
The former holds for all large $\tau $  due to the second claim in Lemma \ref{L.2.1}.  The latter holds for all large $\tau $  due to Theorem \ref{T.5.1} applied to $u=w$, $\nu\equiv 0$, and $\tau =0$, but starting from some positive time for which  $\Lambda^0_w$ ($\ge \Lambda^h_w$ for all $h> 0$) is finite (see \eqref{4.1cc}), instead from time 0.
This $\tau$ then only depends on $f_0,K,\zeta,\eta,\eps_2,\tht$.

When $w(t+\beta(t),x)< 1-\tht$, then $w(t+\tau-1,x)< 1-\tht$ by $w_t\ge 0$, so 
\[
e^{\eps_2^2 t - \eps_2(x_1-R_2)} \le \min \left\{ \frac {\eps_2^2 m_{\tht/2}}K, \frac \tht 2 \right\} e^{(\eps_2^2 -c_0\eps_2/2) t} \le \min \left\{ \frac {\eps_2^2 m_{\tht/2}}K e^{(\eps_2^2 -c_0\eps_2/2) t}, \frac\tht 2 \right\}
\]
by the opposite inequality to \eqref{5.7}.  So either $w(t+\beta(t),x)\le \tfrac\tht 2$, in which case $v(t,x)\le\tht$ and we have $f(x,w(t+\beta(t),x))= f(x,v(t,x))=0$; or $w(t+\beta(t),x)\in (\tfrac\tht 2,1-\tht)$, in which case the right hand side of \eqref{5.6} can be bounded below by 
\[
- Ke^{\eps_2^2 t - \eps_2(x_1-R_2)}  + \eps_2^2 e^{-\eps_2^2 t} m_{\tht/2} \ge  -  \eps_2^2 m_{\tht/2} e^{(\eps_2^2 -c_0\eps_2/2) t} +  \eps_2^2 m_{\tht/2} e^{-\eps_2^2 t}  \ge 0
\]
(using $\eps_2\le c_0/4$ in the last inequality).

It follows that $v$ is a super-solution of \eqref{1.1}, with $v(0,\cdot)\ge u(0,\cdot)$ due to \eqref{1.6}.  Hence $v\ge u$, and the second inequality in \eqref{5.2} holds with 
\[
\nu(t):= \max \left\{ \sup_{x_1\le R_2+c_0t/2} [1-w(t+\tau -1, x)],\, e^{(\eps_2^2 -c_0\eps_2/2) t} \right\}
\]
because $u\le 1$ and $w_t\ge 0$ (notice that $\nu$ depends only on $\tau,f_0,\eps_2$).    Since $\lim_{t\to\infty}\nu(t)=0$ due to Lemma \ref{L.2.1} and $0<\eps_2<c_0/2$, Theorem \ref{T.1.2}(i) for \eqref{1.6} follows from Theorem \ref{T.5.1}.

The proof of \eqref{1.7} in the \eqref{1.5} case  of  Theorem \ref{T.1.2}(i) is similar, with $x_1$ replaced by $|x-x_0|$ in the whole argument,
and $\eps_2(d-1)|x-x_0|^{-1}e^{\eps_2^2 t - \eps_2(|x-x_0|-R_2)}$ added to the right hand side of \eqref{5.6}.
\smallskip

{\it Remark.}  For later reference, in the proof of Theorem \ref{T.1.12} below, we also construct a sub-solution of \eqref{1.1} with the same flavor.  Let $w$ be as in the above proof, solving \eqref{1.1} with $w(0,x):=W(x_1)$.  
We have  $\inf_{w(0,x)\ge \tht/4} w_t(0,x)>0$ by the construction of $W$ (because $U''+f_0(U)>0$), uniformly in all $f\ge f_0$.
It follows from this and parabolic regularity that on some short time interval $[0,\til t]$, $w_t$ is bounded away from zero uniformly in all $(t,x)$ with $w(t,x)\ge\tfrac\tht 2$ and in all $f\ge f_0$  with $f(x,0)\equiv 0$ and  Lipschitz constant $K$.  
This, $w_t\ge 0$, and the first claim in \eqref{5.7a}  now yield
\beq\lb{5.8a}
m:=\inf \left\{ w_t(t,x) \,\bigg|\, f\in F(f_0,K,\tht,\zeta,\eta), \, t\ge 0, \text{ and } w(t,x)\in \left[ \frac\tht 2, 1- \frac\tht 2 \right]  \right\}>0,
\eeq
provided we also assume (without loss) that $\tht\le 4\eps_0$.  This is because an argument as in Lemma \ref{L.8.1}(iii) shows that otherwise there would be some $f\in F(f_0,K,\tht,\zeta,\eta)$ and a solution $u$ of \eqref{1.1} on $(-\til t,\infty)\times\bbR^d$ with $\sup_{t\ge \tau} L_{u,\tht/4}(t)\le\ell_{\tht/4}$ for $\tau$ from \eqref{5.7a}, $u(0,0)\in[\tfrac\tht 2, 1-\tfrac\tht 2]$, and $u_t\equiv 0$.  Then Lemma \ref{L.2.1} and $\tfrac\tht 4\le\eps_0$ show $\lim_{t\to\infty}u(t,0)= 1$, contradicting $u_t\equiv 0$.

Next, pick any $r\in\bbR$ and define
\beq \lb{5.9a}
v(t,x):=w(t-1+e^{-\eps_2^2 t},x)-e^{\eps_2^2 t - \eps_2(x_1-r)},
\eeq
so that for $t>0$,
\beq \lb{5.9}
v_t-\Delta v-f(x,v) = f(x,w(t-1+e^{-\eps_2^2 t},x))-f(x,v) - \eps_2^2 e^{-\eps_2^2 t} w_t(t+\beta(t),x).
\eeq
Here we extend $f$ so that $f(x,u)\ge 0$ for $u\le 0$ (cf.~\eqref{1.24}). We would like to show that the right hand side of \eqref{5.9} is $\le 0$.  

This is obviously true when $v(t,x)\ge 1-\tht$ because then the hypotheses on $f$ and $w\le 1$ show $f(x,w(t-1+e^{-\eps_2^2 t},x))\le f(x,v(t,x))$.  

Now consider $(t,x)\in(0,\infty)\times\bbR^d$ with
\beq \lb{5.8}
x_1\ge \frac {c_0}2 t + \frac 1{\eps_2} \log \max \left\{ \frac K {\eps_2^2 m}, \frac 2\tht \right\}+r.
\eeq
Then
\[
e^{\eps_2^2 t - \eps_2(x_1-r)} \le \min \left\{ \frac {\eps_2^2 m}K, \frac \tht 2 \right\} e^{(\eps_2^2 -c_0\eps_2/2) t} \le \min \left\{ \frac {\eps_2^2 m}K e^{(\eps_2^2 -c_0\eps_2/2) t}, \frac\tht 2 \right\}.
\]
This, $w\in[0,1]$, and the hypotheses on $f$ show that if $w(t-1+e^{-\eps_2^2 t},x)\notin (\tht,1-\tfrac\tht 2)$, then we have $f(x,w(t-1+e^{-\eps_2^2 t},x))\le f(x,v(t,x))$, so the right hand side of \eqref{5.9} is indeed $\le 0$.  If instead $w(t-1+e^{-\eps_2^2 t},x)\in (\tht,1-\tfrac\tht 2)$, then we conclude the same because the right hand side can be bounded above by
\[
Ke^{\eps_2^2 t - \eps_2(x_1-r)} -\eps_2^2 e^{-\eps_2^2 t} m \le \eps_2^2 m e^{(\eps_2^2 -c_0\eps_2/2) t} -\eps_2^2 m e^{-\eps_2^2 t} \le 0.
\]

We cannot, however, conclude this when the opposite of \eqref{5.8} holds and $v(t,x)< 1-\tht$.  Thus we have obtained that $v$ is a sub-solution of \eqref{1.1} on the  set of $(t,x)\in(0,\infty)\times\bbR^d$ such that either \eqref{5.8} holds or $v(t,x)\ge 1-\tht$.  This will turn out to be sufficient for our purposes because typical solutions $u$ spread with speed $>c_0/2$.  Hence for appropriate $u$ we will have $u(t,x)\ge 1-\tht$ when the opposite of \eqref{5.8} holds, and we will still be able to conclude $u\ge v$.

\section{Proof of Lemma \ref{L.3.2} in the case $d=3$ (case $u^+\equiv 1$)} \lb{S5}


Recall the setup from the beginning of Section \ref{S3}.  In particular, all constants depend on $f_0,K,\zeta,\eta$ (but not on $\tht,h$ unless explicitly noted). Let us also assume, without loss, that $t_1=0$ and $y=0$, and denote $Y^{h}_0=Y$, $Z_0=Z$.   Thus \eqref{3.7} becomes $Z(\tau)\le Y(0)$.  Finally, recall that $\alpha_f(x)=\alpha_f(x;\zeta)$ and $\psi$ corresponds to $\zeta'$ in Lemma \ref{L.2.3}.

Let $\kappa\in(0,\tfrac12)$ be such that if $u(t,\til x)\in [\alpha_f(\til x),1-\eps_0]$ for some $(t,\til x)\in [\tfrac 12,\infty)\times\bbR^3$, then 
\beq \lb{3.10}
u(t,x)\ge \frac\eta{2K} \qquad \text{and}\qquad f(x,u(t,x)) \ge \kappa \qquad \text{for any $x\in B_{\sqrt 3 \kappa}(\til x)$.}
\eeq
Note that $\kappa$ exists and is independent of $f,u$ due to $\inf_{x\in\bbR^3}\alpha_f(x)\ge \eta/K$, parabolic regularity, and $f\in F(f_0,K,0,\zeta,\eta)$.  Let also $Q\ge K $ be such that if $\calC:=[0,\kappa)^3$ and $\til w\ge 0$ solves 
\[
\til w_t = \Delta \til w + \left[\zeta' + Q \chi_{[1/2,1]}(t) \chi_{\calC}(x) \right] \til w
\]
on $(\tfrac 12,\infty)\times\bbR^3$ with $\til w(\tfrac 12, \cdot)\ge \tfrac\eta{4K}\chi_{\calC}(\cdot)$, then $\til w(t,\cdot)\ge \chi_{\calC}(\cdot)$ for any $t\ge 1$ (which exists because $\zeta'>0$).  

Assume now that $v\in[0,1]$ solves \eqref{2.2}, and that $\calC_1',\calC_2',\dots$ are all (finitely or infinitely many) disjoint cubes  such that $\calC_n'$ is a $\kappa\bbZ^d$-translation of $\calC$ (i.e., by an integer multiple of $\kappa$ in each coordinate) and $v(x_n')\in( \alpha_f(x_n'), 1-\eps_0]$ for some $x_n'\in \bar \calC_n'$.  Since \eqref{3.10} applies to $v$ in place of $u(t,\cdot)$, its  second claim and  \eqref{2.3} show for each $x_0\in\bbR^d$,
\beq \lb{3.13b}
 \sum_{n\ge 1} (1+|x_n'-x_0|)^{-1} \le \kappa^{-4}.
\eeq

Let $T=T_Y>0$, $R\ge T$, and $\tau=\tau_Y\ge T+1$, all to be chosen later (but independent of $\tht,h,f,u$).  Also let $\calC_1,\dots,\calC_N$ be as above but such that $u(T,x)> \alpha_f(x;\zeta')$ for some $x\in \calC_n\cap B_{Y(0)}(0)$.  Let $t_n\in[0,T)$ be the last time such that $u(t,x)\le \alpha_f(x;\zeta')$ for all $(t,x)\in [0,t_n]\times [\calC_n\cap B_{Y(0)}(0)]$, let $I_n:=[t_n,t_n+1]$, and let $x_n\in \calC_n\cap B_{Y(0)}(0)$ be any point such that  $u(T,x_n)\ge \alpha_f(x_n;\zeta')$ ($(t_n,x_n)$ will be fixed from now on).
Then $u_t\ge 0$ and $Z(\tau)\le Y(0)$ show that 
\beq \lb{3.11a}
u(t,x_n)\in [\alpha_f(x_n;\zeta'), 1-\eps_0] \qquad\text{for $n=1,\dots,N$ and  $t\in[T,\tau]$.}
\eeq

We now claim that if $\tau$ is large enough (depending only on $T,R$ in addition to $f_0,K,\zeta,\eta$), then we must have
\beq \lb{3.13}
 \sum_{|x_n-x_0|\le 2R+2} (1+|x_n-x_0|)^{-1} < 2 \kappa^{-4}
\eeq
for each $x_0\in\bbR^3$.
 This holds due to the same argument as in the case $d\le 2$ of this lemma.  Indeed, if such $\tau$ did not exist, we take a sequence of counter-examples $(f_\tau,u_\tau,x_0^\tau)$ to \eqref{3.13} for $\tau=T+1,T+2,\dots$ and shift each in space by (the negative of) the vector whose each coordinate is the largest multiple of $\kappa$ smaller than the same coordinate of $x_0^\tau$. Parabolic regularity then shows that there is a subsequence  along which these shifted solutions converge locally uniformly to a solution of \eqref{1.1} (with some $f\in F(f_0,K,0,\zeta,\eta)$), whose $t\to\infty$ limit $v\in[0,1]$ satisfies \eqref{2.2}.  Moreover, by taking a further subsequence (along which those shifted $\calC_1^\tau,\dots,\calC_{N_\tau}^\tau$ for which $|x_n^\tau-x_0^\tau|\le 2R+2$ are all the same,  and the corresponding shifted $x_n^\tau$ and as well as the shifted $x_0^\tau$ converge), one obtains existence of $x_0'\in\bar\calC$ and of $\kappa\bbZ^d$-translations $\calC_1',\dots,\calC_{N'}'\subseteq B_{2R+4}(0)$ of $\calC$ and $x_n'\in \bar \calC_n'\cap B_{2R+3}(0)$, such that $v(x_n')\in (\alpha_f(x_n'), 1-\eps_0]$ (because \eqref{3.11a} holds for each $(u_\tau,f_\tau,x_n^\tau)$, and $f_\tau(x_n^\tau,\alpha_{f_\tau}(x_n^\tau; \zeta'))=\zeta'\alpha_{f_\tau}(x_n^\tau; \zeta') \ge \zeta \alpha_{f_\tau}(x_n^\tau; \zeta')+(\zeta'-\zeta)\tfrac \eta K$) and 
 \[
  \sum_{n=1}^{N'} (1+|x_n'-x_0'|)^{-1} \ge 2\kappa^{-4}.
 \]
This obviously contradicts \eqref{3.13b}, so there must be $\tau\ge T+1$ such that \eqref{3.13} holds.

We now reorder the $(\calC_n,t_n,x_n)$ so that $t_1\le \dots\le t_N$.  Define
\[
A(t,x) := Q \sum_{n=1}^N \chi_{I_n}(t) \chi_{\calC_n}(x)
\]
and let $w$ solve
\beq\lb{3.13a}
w_t=\Delta w + [\zeta' + A(t,x)](w-h)
\eeq
on $(0,\infty)\times\bbR^3$, with $w(0,x)= h+\psi(Y(0))^{-1}\psi(|x|)$ (so that $w(0,\cdot)\ge u(0,\cdot)$ by the definition of $Y(0)$).  
 We will now show that $w\ge u$ on $[0,T]\times\bbR^3$.

Since the time-independent function  $h+\psi(Y(0))^{-1}\psi(|x|)$ is a sub-solution of \eqref{3.13a}, we have $w(t,x)\ge 1 \ge u(t,x)$ for $(t,x)\in [0,T]\times (\bbR^3\setminus B_{Y(0)}(0))$. 
Also, $w\ge h$ and  \eqref{3.0b} show 
\[
[\zeta'+A(t,x)](w(t,x)-h)  \ge f(x,w(t,x))
\]
for $(t,x)\in [0,T]\times (B_{Y(0)}(0)\setminus \bigcup_{n=1}^N \calC_n)$, as well as for any $n$ and any $(t,x)\in [0,t_n]\times \calC_n$.  The same is true for $(t,x)\in I_n\times \calC_n$  because $f\in F(f_0,K,\tht,\zeta,\eta)$, $h\le\tht$, and  $Q\ge K$.  Since $w(0,\cdot)\ge u(0,\cdot)$ and $u\le 1$  solves \eqref{1.1},  the comparison principle yields $w\ge u$ on $[0,t_1+1]\times\bbR^3$.  

From $Z(\tau)\le Y(0)$, the definition of $t_1\in [0,\tau-1]$, and  the first claim in \eqref{3.10} we have $u(t_1+\tfrac 12,\cdot)\ge \frac\eta{2K}\chi_{\calC_1}(\cdot)$.  Since $w(t_1+\tfrac 12,\cdot)\ge u(t_1+\tfrac 12,\cdot)$ and $h\le\tfrac\eta{4K}$, the function $\til w(t,x):=w(t-t_1,x)-h$ satisfies $\til w(\tfrac 12,\cdot)\ge \frac\eta{4K}\chi_{\calC_1}(\cdot)$. Our choice of $Q$ then shows $w(t,x)\ge 1$ ($\ge u(t,x)$) for $(t,x)\in [t_1+1,T]\times \calC_1$, so  the comparison principle now yields $w\ge u$ on $[0,t_2+1]\times\bbR^3$. 

Using the argument from the previous paragraph $n-1$ more times (with $t_2,\dots,t_n$ in place of $t_1$) ultimately indeed gives $w\ge u$ on $[0,T]\times\bbR^3$.  It therefore suffices to show
\beq \lb{3.14}
w(T,\cdot) -h\le \psi(Y(0)-c_YT)^{-1}\psi(|\cdot|)
\eeq 
to conclude the proof.  This will be achieved by using \eqref{3.13}, for appropriately chosen $T,R$.

Let 
\[
a(t,x):= e^{-2\zeta' t}\psi(Y(0))\psi(|x|)^{-1} (w(t,x)-h),
\]
 so that we have $a(0,x)\equiv 1$ and 
\[
a_t = \Delta a + \frac {2 x\psi'(|x|)}{|x|\psi(|x|)} \cdot \nabla a + A(t,x) a.
\]
Thus \eqref{3.14} will follow if we prove
\beq \lb{3.15}
\|a(T,\cdot)\|_{L^\infty(\bbR^3)} \le \frac{e^{-2\zeta' T}\psi(Y(0))} {\psi(Y(0)-c_YT)}
\eeq
Let 
\[
\delta:= 2\zeta'\frac {c_Y-2\sqrt{\zeta'}}{c_Y+2\sqrt{\zeta'}}>0.
\]
Since $\tfrac d{dr}[\ln\psi(r)]\ge 4\zeta'(c_Y+2\sqrt{\zeta'})^{-1}=(2\zeta'+\delta)c_Y^{-1}$ for $r\ge r_Y$ and we assume $Y(0)-c_YT\ge r_Y$, it follows that $\psi(Y(0))\psi(Y(0)-c_YT)^{-1} \ge e^{(2\zeta'+\delta)T}$.  Hence it suffices to prove
\beq \lb{3.16}
\|a(T,\cdot)\|_{L^\infty(\bbR^3)} \le e^{\delta T}.
\eeq

We now choose $R\ge T$ to be such that if $B_t$ is the standard Brownian motion in $\bbR^3$ (defined on some probability space $(\Omega,\calF,\bbP)$), then
\beq \lb{3.17}
\bbP \left( \sup_{t\in[0,T]} |B_t| \ge R-2\|\psi'\psi^{-1}\|_\infty T \right) \le \frac 13 e^{-QT}
\eeq
(so $R$ depends on $T$ in addition to $f_0,K,\zeta,\eta$).  For any $|x|\le Y(0)$, we let $X^x_t$ be the stochastic process given by $X^x_0=x$ and
\[
dX^x_t = b(X^x_t)dt + dB_t :=  \frac {2 X^x_t\psi'(|X^x_t|)}{|X^x_t|\psi(|X^x_t|)} dt +  dB_t.
\]
Then the well-known Feynman-Kac formula gives  
\beq \lb{3.18}
a(T,x)= \bbE \left(  e^{\int_0^T A(T-t,X^x_t) dt} \right)
\eeq

We will now show that this is $\le e^{\delta T}$ for $x=0$ (the general case is identical).  Denote $X^0_t=X_t$ and note that $|b|\le 2\|\psi'\psi^{-1}\|_\infty$ yields
\beq \lb{3.19}
\bbP \left( \sup_{t\in[0,T]} |X_t| \ge R \right) \le   \bbP \left( \sup_{t\in[0,T]} |B_t| \ge R-2\|\psi'\psi^{-1}\|_\infty T \right) \le \frac 13 e^{-QT}.
\eeq
Since $A\le Q$, this means that the contribution to \eqref{3.18} from those paths which leave $B_{R}(x=0)$ before time $T$ is at most $\tfrac 13 e^{-QT} e^{QT}=\tfrac 13$.

Next reorder again the $(\calC_n,t_n,x_n)$, now so that $\calC_n\cap B_R(0)\neq \emptyset$ precisely when $n\le N'$ (for some $N'$).   Since \eqref{3.13} holds for any $x_0\in \bbR^3$, we have 
\beq \lb{3.20}
 \sum_{n=1}^{N'} (1+|x_n-x_0|)^{-1} < 2 \kappa^{-4}
\eeq
for any $x_0\in B_{R+1}(0)$.  Then $x_n\in B_{R+1}(0)$ implies that \eqref{3.20} holds for any $x_0\in\bbR^3$.

Consider now the paths which stay in $B_{R}(0)$ until time $T$.  These have non-zero $A(T-t,X_t)$ only at those times $t\in[0,T]$ for which $X_t\in \calC_n$ for some $n\le N'$, and also $T-t\in I_n$.  Since  $|I_n|=1$,  the contribution to \eqref{3.18} from those of these paths which hit fewer than $\delta (2Q)^{-1}T$ of the cubes $\calC_1,\dots,\calC_{N'}$ before time $T$ is at most $\exp(Q\delta (2Q)^{-1} T)\le \tfrac 13 e^{\delta T}$, provided we choose $T\ge \delta^{-1}\ln 9$.

Finally, the contribution to \eqref{3.18} from those paths which stay in $B_{R}(0)$ until time $T$ and hit at least $\delta (2Q)^{-1}T$ of the cubes $\calC_1,\dots,\calC_{N'}$ before time $T$ is at most $e^{-2QT}e^{QT}\le \tfrac 13$ by $A\le Q$ and Lemma \ref{L.3.4} below, provided we let $p:=\delta (2Q)^{-1}$, $P:=2 \max \{\kappa^{-4},Q,  \|\psi'/\psi\|_\infty\}$, and choose $T$ large enough (which then depends on $f_0,K,\zeta,\eta$).

Thus $a(T,x=0)\le e^{\delta T}$, and the general case $|x|\le Y(0)$ is identical.  Hence  \eqref{3.16} follows for any such $T$ and the proof will be finished once we prove the following lemma.

\begin{lemma} \lb{L.3.4}
If $p,P>0$, then for any large enough $T>0$ (depending only on $d,p,P$) the following holds.  If $N\le\infty$, the points $x_n\in\bbR^d$  satisfy 
\beq \lb{3.21}
 \sum_{n=1}^{N} (1+|x_n-x|)^{-1} \le P
\eeq
for any $x\in\bbR^d$,
and the process $X_t$ satisfies $X_0=0$ and $dX_t=b(X_t)dt+dB_t$ with $\|b\|_\infty\le P$ and $B_t$ the standard Brownian motion in $\bbR^d$, then 
\beq \lb{3.22}
\bbP \left( X_t \text{ hits at least $pT$ of the balls $B_1(x_n)$ before time $T$} \right) \le  e^{-PT}.
\eeq
\end{lemma}

{\it Remark.}  The point here is that if $X_t$ hits at least $pT$ of these balls,  the sum of the $\lceil pT \rceil$ displacements it undergoes in-between hits will be bounded below by a quantity super-linear in $T$ because of \eqref{3.21}.  The same will then hold for $B_t$ because $b$ is bounded, but the probability of this decreases super-exponentially in $T$ due to the nature of the Gaussian.

\begin{proof}
Define the stopping times $t_0:=0$,
\[
t_j' := \inf \{ t\ge 0 \,|\, X_s \text{ hits at least $j$ of the balls $B_1(x_n)$ before time $t$}  \},
\]
and $t_j:=\min\{t_j',T\}$ for $j= 1,\dots,\lceil pT \rceil$.  Let $h_j:= \sum_{k=1}^j | X_{t_k}-X_{t_{k-1}} |$ and let 
\[
j_t:=\max \left\{ j\le \lceil pT \rceil \,\big|\, t_j<t  \right\}
\]
 be the smaller of $\lceil pT \rceil$ and the number of the balls hit by $X_s$ before time $t\in (0, T]$ (if $t>T$, then $j_t=\lceil pT \rceil$).    Of course, these are all measurable functions of $\omega\in\Omega$.   Finally, let $\Omega':=\{\omega\in\Omega \,|\, j_T(\omega)=\lceil pT \rceil \}$ be the set of those $\omega$ for which at least $\lceil pT \rceil$ balls are hit by $X_s(\omega)$ before time $T$.  Thus we need to show $\bbP(\Omega') \le  e^{-PT}$.

We now claim that there is $\gamma(T)\to\infty$ as $T\to\infty$ (also depending on $p,P$ but nothing else) such that  (cf.~the Remark above)
\beq \lb{3.24}
 h_{\lceil pT \rceil}  \ge \gamma(T)T \qquad \text{for any $\omega\in\Omega'$.}
\eeq
Indeed, let $\omega\in\Omega'$ and $H=H(\omega):=h_{\lceil pT \rceil}(\omega)$.
For any  $x\in\bbR^d$ we have by \eqref{3.21},
\[
\sum_{j=1}^{ \lceil pT \rceil} \left( 2+\left| X_{t_j}-x \right| \right)^{-1} \le P.
\]
If we take $x:=rX_{t_k}+(1-r)X_{t_{k-1}}$ for some $k=1,\dots,\lceil pT \rceil$ and $r\in[0,1)$, then this gives
\beq \lb{3.23}
\sum_{j=1}^{\lceil pT \rceil} \left( 2+ |rh_k+(1-r)h_{k-1}-h_j| \right)^{-1} \le P.
\eeq
For each $q\in[0,H)$ we let $(r_q,k_q)\in[0,1)\times\{1,\dots,\lceil pT \rceil\}$ be the unique couple such that $q=r_qh_{k_q}+(1-r_q)h_{k_q-1}$.  Integrating \eqref{3.23} over $q$ with $(r,k)=(r_q,k_q)$ yields 
\[
\sum_{j=1}^{\lceil pT \rceil} \int_0^{H} (2+|q-h_j|)^{-1} dq  \le P H.
\]
Since $h_j\in[0,H]$, we obtain
\[
\lceil pT \rceil \ln \frac{2+H}2 \le PH,
\]
yielding \eqref{3.24} with $\gamma(T):=\tfrac pP\ln T$ for $T\ge e^{2P/p}$.  Then  $|b|\le P$ implies also 
\beq \lb{3.25}
\sum_{k=1}^{\lceil pT \rceil} | B_{t_k}-B_{t_{k-1}} |\ge (\gamma(T)-P)T  \qquad \text{for any $\omega\in\Omega'$.}
\eeq

Let now $\{e^{(1)},\dots,e^{(d)}\}$ be the standard basis in $\bbR^d$ and let 
\[
E:=\{ e\in\bbR^d \,|\, e\cdot e^{(l)}\in\{1,-1\} \text{ for each } l=1,\dots,d \}
\]
be the set of the $2^d$ reflections of $(1,\dots,1)$ across subspaces generated by all $2^d$ subsets of the standard basis.  Notice that $E$ is a group if endowed with coordinate-wise multiplication.  For any $\hat e=(e_0,\dots,e_{\lceil pT \rceil})\in E^{\lceil pT \rceil +1}$, define
\[
B_t^{\hat e} := \sum_{j=1}^{j_t} \left( B_{t_j} - B_{t_{j-1}} \right) \prod_{k=0}^{j-1} e_k + \left( B_{t} - B_{t_{j_t}} \right) \prod_{k=0}^{j_t} e_k,
\]
with all multiplications coordinate-wise.
That is, $B_t^{\hat e}$ is obtained from $B_t$ after $\lceil pT \rceil +1$ reflections corresponding to $e_0,\dots,e_{\lceil pT \rceil}$ at stopping times $t_0=0,t_1,\dots,t_{\lceil pT \rceil}$.  (Note that what gets reflected according to $e_j$ is  the displacement $B_t-B_{t_j}$ for any $t>t_j$.  So in particular, $B_t-B_{t_{j_t}}$ gets reflected $j_t+1$ times --- according to $e_0,e_1,\dots,e_{j_t}$.)  

Since $t_j$ are stopping times, each $B_t^{\hat e}$ is also a standard Brownian motion.  For any $\omega\in\Omega$, there  is $\hat e\in E^{\lceil pT \rceil +1}$ such that for $j=1,\dots,\lceil pT \rceil$ (and with $\cdot$ the usual dot product in $\bbR^d$),
\[
\left[ ( B_{t_j} - B_{t_{j-1}})\prod_{k=0}^{j-1} e_k \right] \cdot (1,\dots,1)\ge |B_{t_j} - B_{t_j-1}|.
\]
Indeed, one only needs to choose $e_j$ successively so that $( B_{t_j} - B_{t_j-1})\prod_{k=0}^{j-1} e_k$ has all $d$ coordinates non-negative.  So by \eqref{3.25}, for each $\omega\in\Omega'$, there is $\hat e\in E^{\lceil pT \rceil +1}$ such that
\[
B_{t_{\lceil pT \rceil}}^{\hat e} \cdot (1,\dots,1) \ge  (\gamma(T)-P)T
\]
Since $t_{\lceil pT \rceil}\le T$, we obtain 
\[
\bbP(\Omega') \le 2^{d(\lceil pT \rceil+1)} \bbP \left( B_t\cdot (1,\dots,1) \ge  (\gamma(T)-P)T \text{ for some } t\le T \right).
\]
Given any $C>0$, the last probability is less than  $e^{-CT}$ for all large enough $T$ because $\lim_{T\to\infty} \gamma(T) =\infty$.  Taking $C:=2dp\ln 2+P$ yields $\bbP(\Omega') \le  e^{-PT}$ for all large enough $T$, finishing the proof of Lemma \ref{L.3.4}.
\end{proof}

\section{Proofs of Theorems \ref{T.1.11} and \ref{T.1.12}} \lb{S11}

These proofs follow the same lines as those of Theorems \ref{T.1.2}(i) and \ref{T.1.3}.  The only differences in the proof of Theorem \ref{T.1.11} will be in the proofs of Lemma \ref{L.2.1} and Lemma \ref{L.3.2}, where \eqref{1.21} will play a central role.  In what follows, let us consider the setting of Theorem \ref{T.1.11} (in particular, $\calF$ is fixed) but for now also only $(f,u^+)\in\calF$ and $\tht\ge 0$.  All constants will now depend on  $\calF$ instead of $f_0,K,\zeta,\eta$ (but not on $\tht$ from (H) unless explicitly noted).

Before we start, let us note that since $\inf_{x\in\bbR^d} u^+(x) >\tht_0$ and $f(x,u)\ge f_0(u)>0$ for $u\in(\tht_0,\tht_1)$,  it follows from elliptic regularity and the maximum principle that in fact $\inf_{x\in\bbR^d} u^+(x) \ge \tht_1$.

\begin{lemma} \lb{L.11.1}
There is $\eps_0=\eps_0(\calF)> 0$ such that for each $c<c_0$ and $\eps>0$ there is $\tau=\tau(\calF,c,\eps)$ such that the following holds.  If $u\in[0,u^+]$ solves \eqref{1.1}, \eqref{1.2} with $(f,u^+)\in\calF$, and $u(t_1,x)\ge u^+(x)-\eps_0$ for some $(t_1,x)\in [t_0+1,\infty)\times\bbR^d$, then for each $t\ge t_1+\tau$,
\beq \lb{11.1}
\sup_{|y-x|\le c(t-t_1)} \left[ u^+(x)-u(t,y) \right] \le \eps.
\eeq
The same result holds if the hypothesis $u(t_1,x)\ge u^+(x)-\eps_0$ is replaced by
\beq\lb{11.1a}
u(t_1,\cdot) \ge \frac {\tht_1+\tht_0}2 \chi_{B_R(x)}(\cdot)
\eeq
 for some $(t_1,x)\in [t_0,\infty)\times\bbR^d$ and a large enough $R=R(f_0)>0$.
\end{lemma}


\begin{proof}
Without loss we can assume $t_0=0$ and $x=0$.
As in the argument in the proof of  Lemma \ref{L.2.1}, one shows that $u(t_1,x)\ge u^+(x)-\eps_0$ (with $t_1\ge 1$ and $u^+\ge\tht_1$) yields \eqref{11.1a}, provided $\eps_0>0$ is sufficiently small. So we only need to prove the second claim.  

Let us therefore assume \eqref{11.1a}. The result from \cite{AW} used in Lemma \ref{L.2.1} also applies to bistable $f_0$, and together with the comparison principle gives for any $c'<c_0$ and $\eps'>0$ existence of $\tau'$ such that,
\beq \lb{11.2}
\inf_{|y|\le c't} u(t,y) \ge \tht_1-\eps'
\eeq
when $t\ge \tau'$. To upgrade this to \eqref{11.1}, we will use \eqref{1.21} along with $\sup_{x\in\bbR^d}\alpha_f(x)<\tht_1$.  The latter holds because otherwise $f_0(u)\le\zeta u$ for all $u\in[0,\tht_1]$, which contradicts $c_0> 2\sqrt\zeta$.

If \eqref{11.1} does not hold for some $c<c_0$ and $\eps>0$, we let $u_n$ be a counterexample with $t_n$ ($\to\infty$ as $n\to\infty$) and $|y_n|\le ct_n$, corresponding to some $(f_n,u^+_n)\in\calF$.  We can assume $(u_n)_t\ge 0$, because \eqref{11.2} with $c'\in(c,c_0)$ and a small enough $\eps'\in(0,\eps)$ lets us find a time-increasing solution $w_n$ of \eqref{1.1} between 0 and $u_n$, defined for $t\ge t'$ with some $t'\ge \tau'$, which still spreads with speed $\ge c'$ in the sense of \eqref{11.2}.  Indeed, similarly to \eqref{4.1}, we let $w_n(t',x):=W(|x|)$, where 
\beq\lb{11.3}
W(s):= 
\begin{cases}
\tht_1-\eps' & s\le R',
\\ U(s-R') & s\in(R',R'+s_0],
\\ 0 & s>R'+s_0
\end{cases}
\eeq
and $U,s_0$ are obtained as for \eqref{4.1} but with the current $f_0$ and $U(0)=\tht_1-\eps'$.  Here we need $\eps'>0$ to be small enough (such that $\int_0^{\tht_1-\eps'} f_0(u) du>0$) and $R'$ larger than the right-hand side of \eqref{4.1a} (so that  $W(|x|)$ satisfies \eqref{4.2}).  So each $w_n$ is time-increasing, and also satisfies \eqref{11.2} for large $t$ if $\eps'\le\tfrac 12(\tht_1-\tht_0)$ and $R'\ge R$, with $R$ from \eqref{11.1a}.  Then we only need $t'\ge \max\{\tau',(R'+s_0)/c'\}$ to get $w_n\le u_n$ for all $n$ (note that $\eps',c',\tau',U,s_0,R',t'$ are independent of $n$).

Let therefore $(u_n)_t\ge 0$ be such counterexamples, with $t_n\to\infty$ and $|y_n|\le ct_n$ such that $u_n(t_n,y_n)\le u^+_n(y_n)-\eps$.  After shifting $u_n$ by $(-\tfrac{c'+c}{2c'}t_n,-y_n)$ (and $f_n,u^+_n$ by $y_n$) and passing to a subsequence, we recover an entire time-increasing solution $u$ of \eqref{1.1} with new $(f,u^+)\in\calF$, such that $u\in[\tht_1,u^+]$ (due to \eqref{11.2} for each $u_n$ and all $\eps'>0$) and $\lim_{t\to\infty} u(t,0)\le u^+(0)-\eps$.   As before, the function $p(x):=\lim_{t\to\infty} u(t,x)$ is an equilibrium solution of \eqref{1.1} with $p(0)\le u^+(0)-\eps$.  Since $p\in[\tht_1,u^+]$, the strong maximum principle yields $p <u^+$, and we also have $p\ge\tht_1 >\alpha_f$ on $\bbR^d$.  So the sum in \eqref{1.21} equals $\infty$, contradicting $(f,u^+)\in\calF$.
\end{proof}

\begin{proof}[Proof of Theorem \ref{T.1.11}]
This is essentially identical to the proofs of Theorem \ref{T.1.2}(i) and Theorem~\ref{T.1.3} for $d=3$, but with $u\in[0,u^+]$ instead of $u\in[0,1]$ and using Lemma \ref{L.11.1} instead of Lemma \ref{L.2.1}.  We will again only assume $(f,u^+)\in\calF$ and $\tht\ge 0$ in most of the proof. 

Lemmas \ref{L.2.1a} and \ref{L.8.1} are unchanged, with the sets $S_{t_0,\eps,\ell}, S_{t_0,L}, S_L$ containing triples $(f,u^+,u)$ and restricted to $(f,u^+)\in\calF$ and $u\in[0,u^+]$, and with $1-\eps$ replaced by $u^+(x)-\eps$ in \eqref{1.7b}.  In Section \ref{S3} we take
\[
Z_y(t):= \inf_{u(t,x)\ge u^+(x)-\eps_0} |x-y| 
\]
and keep $Y^{h}_y(t)$ as before because $u^+\le 1$.  Lemma \ref{L.3.1} is unchanged and Lemma \ref{L.3.2} is proved as in the case $d=3$. We cannot use Lemma \ref{L.2.2} here but \eqref{1.21} will do the job.  Indeed, we let $\kappa\in(0,d^{-1/2})$ be such that if $u(t,\til x)\ge \eta$ for some $(t,\til x)\in [\tfrac 12,\infty)\times\bbR^d$, then 
\beq \lb{11.4}
u(t,x)\ge \frac\eta{2K}  \qquad \text{for any $x\in B_{\sqrt d \kappa}(\til x)$,}
\eeq
which replaces \eqref{3.10}.  We then still conclude \eqref{3.13b} using \eqref{1.21}, although with the right hand side being $\lceil \kappa^{-1} \rceil^d\eta^{-1}$ instead of $\kappa^{-4}$.  The rest of the proof is unchanged, as is Theorem \ref{T.3.3} and Corollary \ref{C.3.3a}, except for $1-\eps$ being replaced by $u^+(x)-\eps$ in \eqref{3.9g}.  Section \ref{S4} is also unchanged, using only the arguments in Case $d=3$.  This proves Theorem \ref{T.1.11}(ii) for $u$.
 
The proofs of the remarks at the beginning of Section \ref{S4a} remain the same, with $W$ from \eqref{11.3}  instead of \eqref{4.1} and $R':=R_2$.  Theorem \ref{T.5.1} is also unchanged (note that here we need $\tht>0$ because we employ Theorem \ref{T.1.5}(ii)) and so is the proof of Theorem \ref{T.1.3}(ii).  This proves Theorem \ref{T.1.11}(ii) for $v$.

Finally, since we have \eqref{1.24}, the argument after Theorem \ref{T.5.1} which proves Theorem \ref{T.1.2}(i) also remains the same, with  each ``$1-$'' is replaced by ``$u^+(x)-$''.
\end{proof}

\begin{proof}[Proof of Theorem \ref{T.1.12}]
Let us define $f_0(u)=0$ for $u<0$, and for $\gamma\le 0$ let $c_\gamma$ be the front/spreading speed for $f_0$ but corresponding to fronts connecting $\gamma$ and $\tht_1$ resp. to sufficiently large $u_0\in[\gamma,\tht_1]$ converging to $\gamma$ as $|x|\to\infty$.  It is well known (using phase-plane analysis) that $c_\gamma\in(0,c_0)$ for any $\gamma<0$ as well as $\lim_{\gamma\to 0^-} c_\gamma=c_0$.

(i)  To prove this we will need to construct an appropriate (non-positive) sub-solution, in addition to the super-solution constructed previously.  We will use here the remark at the end of Section \ref{S4a}.  Let us first assume \eqref{1.23}, and let $\gamma:=\inf_{x\in\bbR^d} u_0(x)\le 0$ (then $\gamma=\inf_{(t,x)\in[t_0,\infty)\bbR^d} u(t,x)$ by \eqref{1.24}).  Without loss, we also assume $t_0=0$, $e=e_1$, and $\eps_2\le c_\gamma/4$.  The latter can be done because \eqref{1.23} continues to hold if we replace $\eps_2$  by $\min\{\eps_2,c_\gamma/4\}$ and $R_2$ by $R_2+4c_\gamma^{-1}\ln_+\|u_0\|_\infty$.

First, we claim that 
\beq \lb{11.6}
\limsup_{t\to \infty} \sup_{x\in\bbR^d} [u(t,x)-u^+(x)]\le 0\qquad\text{and}\qquad \liminf_{t\to \infty} \inf_{x\in\bbR^d} u(t,x)\ge 0,
\eeq
where the rate of these decays depends on the same parameters as $T_\eps$ in (C) does, except of $\eps$ (by ``rate'' we mean a function $\til T:(0,\infty)\to (0,\infty)$ such that $\sup_{x\in\bbR^d} [u(t,x)-u^+(x)]\le\del$ and $\inf_{x\in\bbR^d} u(t,x)\ge-\del$ for all $t\ge \til T(\delta)$).  

For the first claim in \eqref{11.6}, let $v(t,x):= u(t,x)-u^+(x)$. Then $v_t\le \Delta v+g(v)$, with
\[
g(s):=\sup_{(f,u^+,x)\in\calF\times\bbR^d} [f(x,u^+(x)+s)-f(x,u^+(x))].
\]
We have $g(s)<0$ for all $s>0$ because otherwise translation invariance of $\calF$ and its closure under locally uniform convergence would yield $(f,u^+)\in\calF$ with $f(x,u^+(x)+s)=f(x,u^+(x))$ for some $s>0$, contradicting the extra hypothesis in the case of \eqref{1.23}.  Obviously $\sup_{x\in\bbR^d} v(t,x)\le \kappa(t)$, where  $\kappa(0):=\sup_{x\in\bbR^d} v(0,x)<\infty$ and $\kappa'=g(\kappa)$.  Thus $\lim_{t\to\infty} \kappa(t)=0$, and the first claim in \eqref{11.6} follows. (If we only have $g\le 0$ but assume $\limsup_{x_1\to-\infty} v(0,x)\le 0$, the claim is immediate from this and \eqref{1.23}.) 

We now turn to the second claim in \eqref{11.6}.  The result from \cite{AW} (see the proof of Lemma~\ref{L.11.1}) for \eqref{1.1} with  $u_0\ge (\tht_0+\eps_1+|\gamma|)\chi_{\{x\,|\,x_1<R_1\}}-|\gamma|$ shows
\beq \lb{11.7}
\inf_{x_1\le R_1+c't} u(t,x)\ge \tht_1-\eps'
\eeq
for any $c'<c_\gamma$, $\eps'>0$, and $t\ge \tau'$ (with $\tau'$ depending only on $f_0,\gamma,\eps_1,c',\eps'$).  
 The comparison principle and \eqref{1.24} yield $u(t,x)\ge  -e^{-\eps_2(x_1-\eps_2t-R_2)}$ because the latter solves the heat equation.  But this, \eqref{11.7} with $c':=c_\gamma/2$ and $\eps':= \tht_1$, and $\eps_2\le c_\gamma/4$ show $\inf_{x\in\bbR^d} u(t,x)\ge -e^{-\eps_2(c_\gamma t/4+R_1-R_2)}$ for $t\ge \tau'$.  The second claim in \eqref{11.6} follows.

\eqref{11.6} shows that for any $\eps>0$ and large enough $t$, the sets $\Omega_{u,\eps}(t)$ and $\Omega_{u,1-\eps}(t)$ from \eqref{1.20a} and \eqref{1.20b} are the same as those in Definition \ref{D.1.4}.  Hence we will use \eqref{1.20a} and \eqref{1.20b}.

We next claim that because of \eqref{11.6} and the parabolic Harnack inequality,  it suffices to prove the result with $L_{u,\eps}(t)$ from \eqref{1.3} instead of \eqref{1.3b}.   First, there is $(K,\|u\|_\infty)$-dependent $M\ge 1$ such that $\sup_{t\ge t_0+1} \|u_t\|_\infty\le M$ for any solution of \eqref{1.1} on $(t_0,\infty)\times\bbR^d$ with  $f$ Lipschitz with constant $K$ and satisfying $f(\cdot,0)\equiv 0$. Given any $\eps\in(0,\tfrac 12)$, consider some small $\eps'>0$ and let $v:=\eps'+u^+-u$ and $T'<\infty$ be such that $\inf_{x\in\bbR^d} v(t,x)\ge 0$ for all $t\ge T'$ (which exists by \eqref{11.6}).  Then we have $v_t- \Delta v - \lambda(t,x) v\ge 0$ for $t\ge T'$, with 
\[
\lambda(t,x):= v(t,x)^{-1}\min\{f(x,u^+(x))-f(x,u(t,x)),0\}
\]
which satisfies $\lambda(t,x)\in [-K,0]$ due to \eqref{1.24}.  So by the Harnack inequality, there is $C\ge 1$ (depending on $\eps, K$) such that if $v(t,x)\le 2\eps'$ for $(t,x)\in[T'+1,\infty)\times\bbR^d$, then $\sup_{|y-x|\le 1/\eps}v(t-\tfrac\eps{2M},y)\le C\eps'$.  If we now let $\eps':=\tfrac \eps{2C}$, 
then this means $u(t,y)\ge u^+(y)- \eps$ for all $y\in B_{1/\eps}(x)$ whenever  $u(t,x)\ge u^+(x)-\eps'$ and $t\ge T'+1$.  Therefore $\ell_{\eps'}$ for $L_{u,\eps'}(t)$ from \eqref{1.3} works as $\ell_\eps$ for $L_{u,\eps}(t)$ from \eqref{1.3b}, provided we can obtain (C) for the former.  

So we can consider $L_{u,\eps}(t)$ from \eqref{1.3}, with $\Omega_{u,\eps}(t)$ and $\Omega_{u,1-\eps}(t)$ from \eqref{1.20a} and \eqref{1.20b}, which is what was done in the proofs of Theorems \ref{T.1.2}(i) and \ref{T.1.11}(i).

Next, notice that \eqref{11.7} and $u\ge  \gamma$ can be upgraded to
\beq \lb{11.8}
\sup_{x_1\le R_1+ct} [u^+(x)-u(t,x)]\le \eps
\eeq
for any $c<c_\gamma$, $\eps>0$, and $t\ge \til\tau$ (with $\til\tau$ depending only on $\calF,c,\eps,\gamma,\eps_1$).  Indeed, this is done using \eqref{11.7} in the same way \eqref{11.2} is used to prove \eqref{11.1}, but with  $U(s_0)=-\gamma$ in the definition of $W$ (we still have $\int_{-\gamma}^{\tht_1-\eps'} f_0(u) du>0$).  In fact, \eqref{11.7} and then \eqref{11.8} hold for any $c<c_0$ because of the second claim in \eqref{11.6} and $\lim_{\gamma\to 0^-} c_\gamma=c_0$.

Let us assume, without loss, that $\tht>0$ is small enough so that $\tht\le \tfrac12(\tht_1-\tht_0)$ and  $\int_{0}^{\tht_1-\tht} f_0(u) du>0$.  We now use \eqref{11.8} with $c:=c_\gamma/2$, $\eps:=\tht$, and the corresponding $\til\tau$,  together with $u(t,x)\ge  -e^{-\eps_2(x_1-\eps_2t-R_2)}$ (shown above) and $\eps_2\le c_\gamma/4$, to obtain for $t\ge\til\tau$,
\[
u(t,x)\ge   (u^+(x)-\tht)\chi_{\{x\,|\, x_1\le R_1+c_\gamma t/2\}}(x) -e^{-\eps_2(x_1-R_2-c_\gamma t/4)} \chi_{\{x\,|\, x_1> R_1+c_\gamma t/2\}}(x) .
\]
Of course, \eqref{1.23} and $f(u)\le Ku$ for $u\ge 0$ also give
\[
u(t,x)\le e^{(\eps_2^2+K) t-\eps_2(x_1-R_2)}.
\]

Now consider $W$ from \eqref{11.3}, with  $\eps':=\tht$ and $R':=R_2$.  Consider also $w$ solving \eqref{1.1} with $w(0,x):=W(x_1)$. As in the remark at the end of Section \ref{S4a}, we obtain
\beq\lb{5.8b}
m:=\inf \left\{ w_t(t,x) \,\bigg|\, (f,u^+)\in \calF_\tht, \, t\ge 0, \text{ and } w(t,x)\in \left[ \frac\tht 2, u^+(x)- \frac\tht 2 \right]  \right\}>0.
\eeq
With $s_0$ from \eqref{11.3}, we let
\[
r:=R_2+s_0- \frac 1{\eps_2} \log \max \left\{ \frac K {\eps_2^2 m} , \frac 2\tht \right\},
\]
and shift $u$ by $(-T,-R)$, and $f,u^+$ by $-R$ in space, where
\[
T:=\max \left\{ \til\tau, 4{c_\gamma}^{-1} \left( 2R_2-R_1+s_0-r  \right) \right\},
\]
\[
R:= \frac {c_\gamma T}4 + R_2-r.
\]
Since $\tfrac 12 c_\gamma T\ge R_2-R_1+s_0+R $, the above estimates on the original $u(t,x)$ for $t\ge T$ ($\ge \til\tau$) now give  for the shifted $u,u^+$ and $t\ge 0$,
\beq\lb{11.10}
u(t,x)\ge (u^+(x)-\tht)\chi_{\{x\,|\, x_1\le R_2+s_0+c_\gamma t/2\}}(x) -e^{-\eps_2(x_1-r-c_\gamma t/4)} \chi_{\{x\,|\, x_1> R_2+s_0+c_\gamma t/2\}}(x) ,
\eeq
\beq\lb{11.11}
u(0,x)\le e^{-\eps_2(x_1-R_2+R-\eps_2T-KT/\eps_2)}.
\eeq

The crucial point here is that $u(t,x)\ge u^+(x)-\tht$ when
\beq\lb{11.9}
x_1\le \frac{c_\gamma t}2 + \frac 1{\eps_2} \log \max \left\{ \frac K {\eps_2^2 m} , \frac 2\tht \right\} +r .
\eeq
Hence $v$ from \eqref{5.9a}, which by the argument from the remark at the end of Section \ref{S4a} is a subsolution of \eqref{1.1} on the  set of $(t,x)\in(0,\infty)\times\bbR^d$ such that either the opposite of \eqref{11.9} holds or $v(t,x)\ge u^+(x)-\tht$, will stay below $u$ as long as $v(0,\cdot)\le u(0,\cdot)$.  But this holds due to \eqref{11.10} for $t=0$ because $w(0,\cdot)\le \tht_1-\tht\le u^+(\cdot)-\tht$ and $w(0,\cdot)$ vanishes for $x_1\ge R_2+s_0$.

Thus \eqref{5.9a} shows that the first inequality in \eqref{5.2} holds with $\tau=1$ and
\[
\nu(t):= \max \left\{ \sup_{x_1\le r+c_\gamma t/2} [u^+(x)-u(t, x)],\, e^{(\eps_2^2 -c_\gamma \eps_2/2) t} \right\}
\]
because $w\le u^+$ and $w_t\ge 0$.    Moreover,  $\lim_{t\to\infty}\nu(t)=0$ due to $0<\eps_2<c_\gamma/2$ and \eqref{11.8} for some $c>c_\gamma/2$ and any $\eps>0$.

On the other hand, as in Section \ref{S4a}, we also have a super-solution of \eqref{1.1} on $(0,\infty)\times\bbR^d$ from \eqref{5.9b}, with some large $\tau$ and $R_2$ replaced by $R_2-R+\eps_2T+KT/\eps_2$.  This then stays above $u$ due to \eqref{11.11}.  As in Section \ref{S4a}, we obtain the second inequality in \eqref{5.2} with some $\nu(t)\to 0$ as $t\to\infty$.  The (H') version of Theorem \ref{T.5.1} now finishes the proof.

The proof in the case \eqref{1.22} is similar, with $x_1$ replaced by $|x-x_0|$ and sufficiently large $R_1$ to guarantee \eqref{11.7} with $x_1$ replaced by $|x-x_0|$.  Notice that the first claim in \eqref{11.6} follows from \eqref{1.22}, even though now we only have  $g\le 0$.

(ii)  Similarly to (i), the extra hypotheses in (ii)  imply both claims in
 \eqref{11.6}.  So again we can consider $L_{u,\eps}(t)$ from \eqref{1.3}, with $\Omega_{u,\eps}(t)$ and $\Omega_{u,1-\eps}(t)$ from \eqref{1.20a} and \eqref{1.20b}.  Moreover \eqref{3.6a} shows $u_t\ge 0$ so we must have $u\le u^+$.

Assume again $t_0=0$ without loss and notice that the hypotheses continue to hold if we replace $u_0$ by $u(t,\cdot)$ for any $t\ge 0$.  Indeed,  for all small enough $\eps>0$ and all $y\in\bbR^d$ we have $Z_y(0)\le Y^{\eps/2}_y(0)+L_{u,\eps/2,1-\eps_0}(0)$ as in the case $d=3$ of the proof of Theorem \ref{T.1.3}(i).  Since $u_t\ge 0$, the (H') version of Lemma \ref{L.3.1}(i) now gives
\[
Z_y(t)\le Z_y(0)\le Y^{\eps/2}_y(0)+L_{u,\eps/2,1-\eps_0}(0)\le Y^{\eps/2}_y(t)+c_Y't+r_Y+L_{u,\eps/2,1-\eps_0}(0).
\]
If $u(t,y)\ge\eps$, then $Y^{\eps/2}_y(t)\le \psi^{-1}(\tfrac\eps 2)$, so this yields 
\[
L_{u,\eps,1-\eps_0}(t)\le  \psi^{-1}(\tfrac\eps 2)+c_Y't+r_Y+L_{u,\eps/2,1-\eps_0}(0)<\infty.
\]

This and \eqref{11.6} mean that we can assume without loss that $\gamma:= \min\{\inf_{x\in\bbR^d} u_0(x),0\}=\min\{\inf_{(t,x)\in[0,\infty)\bbR^d} u(t,x),0\}$ is such that $c_\gamma>c_Z$ from \eqref{3.0}.
But then Lemma \ref{L.11.1} shows that the (H') version of Lemma \ref{L.3.1}(iii) continues to hold because now $u\in[\gamma,u^+]$.  The rest of the proof of Theorem~\ref{T.1.3} (or rather Theorem \ref{T.1.11}(ii)) then carries over to the case $u\in[\gamma,u^+]$, with $c_0$ replaced by $c_\gamma$ and the obvious (minimal) changes
(notice also that the second claim in \eqref{11.6}, which holds for any $(f,u^+)\in\calF$ and bounded $u_0$ satisfying \eqref{3.6a}, precludes existence of equilibrium solutions $p$ of \eqref{1.1} with $\gamma<p<u^+$ and $\inf_{x\in\bbR^d} p(x)<0$).  Since we can take $\gamma$ arbitrarily close to 0, by replacing $u_0$ with $u(t,\cdot)$ for a large enough $t$, we finally also obtain global mean speed in $[c_0,c_1]$.
The proof is thus finished.
\end{proof}

\section{Proof of Theorem \ref{T.1.5}} \lb{S8}


Let $K\ge 1$ be a Lipschitz constant for $f$ and pick a non-increasing and continuous function $L:(0,\tfrac 12)\to(0,\infty)$ such that $(\sup_{t\in\bbR} L_{u,\eps}(t)=)\,L^{u,\eps}\le L(\eps)$ for all $\eps\in(0,\tfrac 12)$.  We can do so because $L^{u,\eps}$ is finite and non-increasing in $\eps$.

(i) Let $m_0:=\inf_{(t,x)\in \bbR^{d+1}} u(t,x)$ and $m_1:=\sup_{(t,x)\in \bbR^{d+1}} [u(t,x)-u^+(x)]$.  
It is easy to see that Lemma \ref{L.8.1}(i,ii) extends to the case when $S_{t_0,\eps,\ell}= S_{t_0,\eps,\ell}(K,m_0,m_1)$ is defined 
to be the set of all triples $(f,u^+,u)$ such that  $f$ is Lipschitz with constant $K$, the functions $u^-\equiv 0$ and $u^+$ satisfy \eqref{1.19} and are equilibrium solutions of \eqref{1.1}, and $u$ with $m_0\le u\le u^+ +m_1$ solves \eqref{1.1} on $(t_0,\infty)\times\bbR^d$ and satisfies $L_{u,\eps'}(t)\le \ell$ for all $\eps'\in(\eps,\tfrac 12)$ and all $t> t_0$ (with $L_{u,\eps}(t)$ from Definition \ref{D.1.4}).

Assume first that $m_1>0$.  Take $(t_n,x_n)$ such that $u(t_n,x_n)-u^+(x_n)\to m_1$ and define $f_n(x,u):=f(x-x_n,u)$, $u^+_n(x):=u^+(x-x_n)$, and $u_n(t,x):=u(t-t_n,x-x_n)$.   Since $(f_n,u^+_n,u_n)$ belongs to the corresponding  $S_{-\infty,L}=S_{-\infty,L}(K,m_0,m_1)$, we obtain as in the proof of Lemma \ref{L.8.1} some new $(f,u^+,u)\in S_{-\infty,L}$ (and thus also $u\not\equiv u^+ +m_1$) such that $u\le u^+ + m_1$ and $u(0,0)=u^+(0)+m_1$.  The function $u^+ +m_1$ is a super-solution of \eqref{1.1} due to \eqref{1.24}, so the strong maximum principle yields a contradiction with $u\not\equiv u^++m_1$.  Thus $m_1\le 0$ and the strong maximum principle also shows $u<u^+$.

The case $m_0<0$ is identical, this time using that the constant $m_0$ is a sub-solution of \eqref{1.1} due to  \eqref{1.24}.  We obtain $m_0\ge 0$ and then also $u>0$.  

(ii) By the discussion following Definition \ref{D.1.4} (see also the proof of Theorem \ref{T.1.12}(i) above),  (i) shows that it is equivalent to use  \eqref{1.3} instead of \eqref{1.3b} in what follows.  We will do so in (ii) and (iii), including in the definition of
$S_{t_0,\eps,\ell}=S_{t_0,\eps,\ell}(K,0,0)$ from (i) (and thus also in $S_{t_0,L},S_L$, with the condition $u\not\equiv 0,u^+$ for $S_L$).   In addition, (i) and $u^-\equiv0$ show that we have \eqref{1.20a} and \eqref{1.20b}.

Since $u$ propagates with a positive global mean speed,  \eqref{1.3g} shows that $u$ is a transition solution connecting $u^-\equiv 0$ and $u^+$.  Indeed, $u\not\equiv u^+$ gives $\Omega_{u,1-\eps}(0)\neq\bbR^d$ for all small $\eps>0$.  Thus the first inclusion in \eqref{1.3g}, with $t\to-\infty$ and $\tau:=-t$, shows the $t\to-\infty$ limit in \eqref{1.8}.  Also, $u\not\equiv 0$  gives $\Omega_{u,\eps}(0)\neq\emptyset$ for all small $\eps>0$.  Thus the  first inclusion in \eqref{1.3g}, with $t=0$ and $\tau\to\infty$,  shows the $t\to\infty$ limit in \eqref{1.8}.


Assume now that $\tht>0$ is as in (ii), take $\eps_0:= \tht/ 2>0$, and let
\[
u^s(t,x):=u(t+s,x)
\]
be a time shift of $u$.  It is then sufficient to show that $u^s\ge u$ for any $s\ge 0$.  Indeed, we then obtain $u_t\ge 0$, and the strict inequality follows from the strong maximum principle for $u_t$ (which satisfies a linear equation and is not identically 0 because $u$ is a transition solution).

\begin{lemma} \lb{L.8.2}
There is $s_0$ such that $u^s\ge u$ whenever $s\ge s_0$.  
\end{lemma}

\begin{proof}
Since $u$ propagates with a positive global mean speed (with some $c>0$ and $\tau_{\eps,\delta}<\infty$ in Definition \ref{D.1.1b}),  we have $\Omega_{u,\eps_0}(t)\subseteq \Omega_{u,1-\eps_0}(t+s)$ for all $t\in\bbR$ and $s\ge s_0:=\tau_{\eps_0,c/2}$.
Thus  for $s\ge s_0$ we have $u^s+\eps_0\ge u$ 
as well as
\beq \lb{8.2}
u^s(t,x)\ge u(t,x) \text{ whenever } u(t,x)\in [\eps_0,u^+(x)-\eps_0].
\eeq

Next take any $s\ge s_0$ and let $\eps\ge 0$ be the smallest number such that $u^s+\eps\ge u$.  Obviously, $\eps$ exists and  $\eps\le \eps_0$.  We need to show that $\eps=0$, so assume that $\eps>0$ and let $(t_n,x_n)$ satisfy
\[
\lim_{n\to\infty} [u^s(t_n,x_n)+\eps-u(t_n,x_n)] = 0.
\]
Then \eqref{8.2} shows that $u(t_n,x_n)\notin[\eps_0,u^+(x_n)-\eps_0]$ for large enough $n$, so  either $u(t_n,x_n)\in [\eps,\eps_0]$ or $u(t_n,x_n)\in [u^+(x_n)-\eps_0,u^+(x_n)]$.  
Apply the $u^+\not\equiv 1$ version of Lemma \ref{L.8.1}(ii) with $t_n(\eps)=-\infty$, $f_n(x,u):=f(x-x_n,u)$, $u^+_n(x):=u^+(x-x_n)$, and $u_n(t,x):=u(t-t_n,x-x_n)$.  We obtain $(\til f,\til u^+,\til u)\in S_{-\infty,L}(K,0,0)$ such that  $\til u\in[0,\til u^+]$, $\til u^s+\eps\ge \til u$, 
\beq\lb{8.2a}
\til u^s(t,x)\ge \til u(t,x) \text{ whenever } \til u(t,x)\in [\eps_0,\til u^+(x)-\eps_0]
\eeq
by \eqref{8.2}, and
\[
\til u^s(0,0)+\eps=\til u(0,0) \qquad( \in[\eps,\eps_0]\cup[\til u^+(0)-\eps_0,\til u^+(0)] ).
\]
Moreover, $\til f$ is non-increasing in $u$ on $[0,\tht]$ and on $[\til u^+(x)-\tht,\til u^+(x)]$ because $f,u^+$ have the same property.

Let now $v:= \til u^s+\eps-\til u\ge 0$, so that $v_t = \Delta v + \til f(x,\til u^s)-\til f(x,\til u)$.  We then have 
\[
v_t \ge \Delta v-Kv.
\]
Indeed,  this obviously holds when $\til u^s(t,x)\ge \til u(t,x)$.  When $\til u^s(t,x)< \til u(t,x)$, then \eqref{8.2a} and $\eps\le\eps_0$ show $\til u^s(t,x),\til u(t,x)\in [0,\eps_0]\cup[\til u^+(0)-2\eps_0,\til u^+(0)]$, so $\til f(x,\til u^s(t,x))\ge \til f(x,\til u(t,x))$ by $2\eps_0=\tht$.  Now we obtain $v\equiv 0$ on $(-\infty,0]\times\bbR$ by $v\ge 0$, $v(0,0)=0$, and the strong maximum principle.  But then $\til u^{-sn}\equiv \til u+n\eps$ on $(-\infty,0]\times\bbR$ for $n\in\bbN$, a contradiction with boundedness of $\til u$.   Thus $\eps=0$ for any $s\ge s_0$ and the proof is finished.
\end{proof}

\begin{lemma} \lb{L.8.3}
We have $u^s\ge u$ for any $s\ge 0$.
\end{lemma}

\begin{proof}
Let $s_1\ge 0$ be the smallest number such that $u^s\ge u$ for any $s\ge s_1$ (which obviously exists), and assume that $s_1>0$.  We first claim that
\beq\lb{8.1}
m:= \min \left\{ \eps_0, \inf_{u(t,x)\in[\eps_0,u^+(x)-\eps_0]} [u^{s_1}(t,x)-u(t,x)]  \right\} >0.
\eeq
Indeed, if $m=0$, then let $(t_n,x_n)$ be such that $u(t_n,x_n)\in[\eps_0,u^+(x_n)-\eps_0]$ and
\[
\lim_{n\to\infty} [u^{s_1}(t_n,x_n)-u(t_n,x_n)] = 0.
\]
The $u^+\not\equiv 1$ version of Lemma \ref{L.8.1}(ii) with $t_n(\eps):=-\infty$, $f_n(x,u):=f(x-x_n,u)$, $u^+_n(x):=u^+(x-x_n)$, and $u_n(t,x):=u(t-t_n,x-x_n)$, again yields $(\til f,\til u^+,\til u)\in S_{-\infty,L}(K,0,0)$ such that  
\[
\til u^{s_1} \ge \til u \qquad\text{and}\qquad \til u^{s_1}(0,0)=\til u(0,0) \in[\eps_0,\til u^+(0)-\eps_0].
\]
This contradicts the strong maximum principle for $v:=\til u^{s_1}-\til u\ge 0$, which satisfies a linear equation $v_t=\Delta v + \lambda(t,x) v$
with $\|\lambda\|_\infty\le K$,  because  $v(0,0)=0$ and $v\not \equiv 0$.  The latter holds because otherwise $\til u$ would be time-periodic, contradicting \eqref{1.8} for $\til u$ (which propagates with a positive global mean speed because the same is true for $u_n$, with $n$-independent constants in Definition \ref{D.1.1b}, so \eqref{1.8} holds for $\til u$ by the first claim in (ii)).

So $m>0$ and since $u_t$ is uniformly bounded by parabolic regularity, there is $s_2\in(0,s_1)$ such that $u^s \ge u^{s_1} -m$ for any $s\in[s_2,s_1]$.  Thus \eqref{8.2} as well  as $u^s+\eps_0\ge u$ hold for any 
 $s\in[s_2,s_1]$.  Fix any such $s$ and let $\eps\in [0,\eps_0]$ be the smallest number such that $u^s+\eps\ge u$.   The  argument  in the proof of Lemma \ref{L.8.2} now shows that $\eps=0$. But since $s\in [s_2,s_1]$ was arbitrary, this means that we  can decrease $s_1$ to $s_2$,  a contradiction.  It follows that $s_1=0$.
\end{proof}

This finishes the proof of (ii).

(iii)  From bounded width of $u$ and Lemma \ref{L.11.1} it follows that $u$ propagates with a positive global mean speed.   Thus (ii) yields $u_t>0$, and then the (H') version of Corollary \ref{C.3.3a}(ii) from the proof of Theorem \ref{T.1.11} gives the result. 
Note that we did not use Theorem \ref{T.1.5} in the proof of Corollary \ref{C.3.3a}. 

%

\section{Proof of Theorem \ref{T.1.2}(ii) 
} \lb{S6}

We will prove this by constructing an example of an ignition $f$ which prevents  most solutions from having a bounded width (an almost identical construction can be made with a monostable $f$).  The idea is to find $f$ such that there is an equilibrium solution $p\in(0,1)$ of \eqref{1.1}, independent of $x_1$, with the transition $0\to p$ propagating faster in the direction $e_1=(1,0,\dots,0)$ than the transition $p\to 1$.  Then as $t\to\infty$, the solution $u$ will be close to $p$ on a slab $I_t\times\bbR^{d-1}$ (to the left of which $u\sim 1$ and to the right of which $u\sim 0$), with $I_t$ an interval of linearly growing length.

Let $\til p:\bbR^{d-1}\to(0,1)$ be $C^\infty$, radially  symmetric, radially decreasing, with 
\[
\til p (\til x)=3^{d-4}|\til x|^{3-d} \text{ for $|\til x|\ge 3$,} \qquad \Delta \til p <0 \text{ on $B_3(0)$,} \qquad \text{and} \qquad \til p(B_1(0))\subseteq(\tfrac 23,\tfrac 34).
\]
 Let $f_0:[0,1]\to[0,\infty)$ be a $C^\infty$ ignition reaction with $f_0(\til p(\til x))=-\Delta\til p(\til x)$ for $\til x\in\bbR^{d-1}$ (so ignition temperature is $\tht_0:=\tfrac 13$) and decreasing on $[\tfrac 34,1]$.  Then $p(x):=p(x_2,\dots,x_d)\in(0,\tfrac 34)$ satisfies on $\bbR^d$,
\[
\Delta p + f_0(p) = 0.
\]

Consider $f$ that satisfies (H) with this $f_0$, some $K\ge 1$ and $\tht:=\tfrac 14$, as well as $f(x,u)=f_0(u)$ for all $x\in\bbR^d$ and $u\in[0,\tfrac 12]\cup [p(x),1]$, and $f(x,u)\ge f_0(u)$ for all $x\in\bbR^d$ and $u\in(\tfrac 12,p(x))$ (provided this interval is non-empty).  Then obviously $f(x,u)= 0$ for $(x,u)\in \bbR^d\times[0,\tht_0]$.

\begin{lemma} \lb{L.9.1}
Let $f$ be as above and $c:=\max\{2\sqrt{\|f_0'\|_\infty},1\}>0$.  If $u$ solves \eqref{1.1}, \eqref{1.2} with $t_0=0$ and $u_0(x)\le p(x)+e^{-c(x_1-z)/2}$ for all $x\in\bbR^d$ and some $z\in\bbR$, then
\beq \lb{9.1}
u(t,x)\le p(x)+e^{-c(x_1-z-ct)/2}.
\eeq
\end{lemma}

\begin{proof}
Let $v$ be the right-hand side of \eqref{9.1}.  Then
\[
v_t-\Delta v -f(x,v) =f_0(p(x))-f_0(v)+\frac{c^2}4 e^{-c(x_1-z-ct)/2} \ge 0
\]
by $c^2/4\ge \|f_0'\|_\infty$.  So $v$ is a super-solution for \eqref{1.1} with $v(0,\cdot)\ge u_0$, and we are done.
\end{proof}

That is, transition $p\to 1$ is propagating in the direction $e_1$ with speed at most $c$, which is independent of $K,f$.  We now make $f$ sufficiently large for $u\in(\tfrac 12,p(x))$ so that the transition $0\to p$ will be propagating faster than speed $c$.

Let $f^M$ be as $f$ above, with $f^M(x,u):=f_0(u)+M(u-\tfrac 12)(p(x)-u)$ for $u\in(\tfrac 12,p(x))$.  Let $t_0:=0$ and fix a radially symmetric and radially non-increasing $v_0$ such that
\beq \lb{9.2}
\frac 23 \chi_{B_{1/2}(0)} \le v_0 \le \frac 23 \chi_{B_{1}(0)} \qquad(\le p)
\eeq
and $\Delta v_0(x) + f^{M_0}(x,v_0(x))\ge 0$ for some $M_0\gg 1$.  Then the solution $v^M$ of \eqref{1.1}, with $f^M$ instead of $f$ and  $v^M(0,\cdot)=v_0(\cdot)$, is non-decreasing in both $t>0$ and $M\ge M_0$.  Therefore 
\[
w(t,x):=\lim_{M\to\infty} v^M(t,x) \qquad(\le p(x))
\]
 satisfies $w_t\ge 0$.  

We claim that $w(t,x)\notin (\tfrac 12,p(x))$ for any $(t,x)\in(0,\infty)\times\bbR^d$.  Otherwise there is $M_1\ge M_0$ such that $v^{M_1}(t,x)\in (\tfrac 12,p(x))$, and then there is $\eps>0$ such that $v^{M_1}(t-\eps,y)\in (\tfrac 12,p(y))$ for all $y\in B_\eps(x)$.  Since $v^M(t-\eps,\cdot)\ge v^{M_1}(t-\eps,\cdot)$ for all $M\ge M_1$, it follows from the definition of $f^M$ that $w(s,y)\ge p(y)$ for all $s>t-\eps$ and $y\in B_\eps(x)$.  This is a contradiction with the hypothesis $w(t,x)\in(\tfrac 12,p(x))$.  This,  \eqref{9.2}, and $v^M_t\ge 0$ thus show
\beq \lb{9.3}
w(t,x)=p(x) \qquad \text{for $(t,x)\in(0,\infty)\times B_{1/2}(0)$.}
\eeq

Pick some $0<\tau\ll 1$.  It is easy to show using the properties of the Gaussian that if $\tau$ is sufficiently small, then any super-solution of the heat equation on $D:=\bbR^d\setminus B_{1/2-\tau^{2/3}}(0)$ with initial condition $u(\tau,x)\ge 0$ for $x\in D$ and boundary condition $u(t,x)\ge \tfrac 35$ on $(\tau,\infty)\times \partial D$ satisfies $u(2\tau,x)>\tfrac 12$ for all $x\in B_{1/2+\tau^{2/3}}(0)$.  \eqref{9.3} shows that there is $M$ such that $v^M$ satisfies these conditions, and it follows that 
\beq \lb{9.4}
w(t,x)=p(x) \qquad \text{for $(t,x)\in (2\tau,\infty)\times B_{1/2+\tau^{2/3}}(0)$}.
\eeq

We can repeat this argument with \eqref{9.4} as a starting point instead of \eqref{9.3} and eventually obtain for all integers $n\ge \tau^{-2/3}$,
\beq \lb{9.5}
w(t,x)=p(x) \qquad \text{for $(t,x)\in (2n\tau,\infty)\times A_{n\tau^{2/3}-1/2}$},
\eeq
where $A_{a}:=\bigcup_{b\in[-a,a]} B_1(b,0,\dots,0)\subseteq \bbR^{d}$ (we need $A_{n\tau^{2/3}-1/2}$ instead of $B_{n\tau^{2/3}+1/2}(0)$ because $p(x)>\tfrac 12$ holds only when $|(x_2,\dots,x_d)|<C$, for some $C>1$) .  One can in fact show that $w(t,\cdot)=p(\cdot)$ for all $t>0$ but we will not need this.  
If we now choose $\tau>0$ so that there exists an integer $n\in((2c+1)\tau^{-2/3},(2\tau)^{-1})$, then \eqref{9.5} yields
\[
v^M(1,\cdot)\ge \frac 23 \chi_{A_{2c}}(\cdot),
\]
for some $M\ge M_0$.  Iterating this, we obtain for $m\in\bbN$,
\beq \lb{9.6}
v^M(m,\cdot)\ge \frac 23 \chi_{A_{2cm}}(\cdot).
\eeq

So let us take $f:=f^M$ and any $u_0$ as in Theorem \ref{T.1.2}(ii)  (without loss let $t_0=0$).  It follows from $c\ge 1$, $p\le\tfrac 34$, and Lemma \ref{L.9.1} with $z=0$ that
\beq \lb{9.7}
\sup_{t>0, \, x_1>ct+4} u(t,x) \le \frac 9{10}.
\eeq
If $u\to 1$ locally uniformly on $\bbR^d$ as $t\to\infty$, then $u(t_1,\cdot)\ge \tfrac 23 \chi_{B_{1/2}(0)}(\cdot)$ for some $t_1>0$, and so $u(t_1+m,2cm)\ge \frac 23$ for $m\in\bbN$ by \eqref{9.6}.   It follows from this and \eqref{9.7} that all claims in (C) are false for $u$.  This also holds when $u\not\to 1$ locally uniformly on $\bbR^d$ as $t\to\infty$, because then Lemma \ref{L.2.1} shows $\sup_{(t,x)\in (1,\infty)\times\bbR^d} u(t,x)<1$ (and $u\not\to 0$ uniformly by the hypothesis).

\section{Proof of Theorem \ref{T.1.6}} \lb{S7}

(i)  Having  Remark 2 after  Theorem \ref{T.1.11}, this is now rather standard.  Let $U,s_0,$ and small $\eps'>0$ be as in \eqref{11.3}, with $R'$ large enough so that $w_0(x):=W(|x|)$ satisfies \eqref{4.2} and the solution $w$ of \eqref{1.1} with $f_0$ in place of $f$ and $w(0,x):=w_0(x)$ spreads in the sense $w\to \tht_1$ locally uniformly as $t\to\infty$ \cite{AW}.  If now $u_n\in[0,u^+]$ is the solution of \eqref{1.1} on $(0,\infty)\times\bbR^d$ with $u_n(0,x)=w_0(x-ne_1)$, 
then $(u_n)_t\ge 0$; and the proof of Lemma \ref{L.11.1}, along with $f\ge f_0$ and the comparison principle, shows that $u_n\to u^+$ locally uniformly as $t\to\infty$.

If $t_n$ is the first time such that $u_n(t_n,0)=\tht_0$, shift $u_n$ in time by $-t_n$ so that now it solves \eqref{1.1} on $(-t_n,\infty)\times\bbR^d$ and $u_n(0,0)=\tht_0$.  Obviously $\lim_{n\to\infty} t_n=\infty$ because $u\le u^+\le 1$, $f(x,u)\le Ku$, and the comparison principle yield on $(-t_n,\infty)\times\bbR^d$,
\[
u_n(t,x)\le e^{-\sqrt K(x_1+n-s_0-R'-2\sqrt K (t+t_n))}.
\]
So by parabolic regularity, there is a sub-sequence along which $u_n$ and their spatio-temporal first and spatial second derivatives  converge locally uniformly to some solution $u$ of   \eqref{1.1} on $\bbR\times\bbR^d$.  Obviously $0\le u\le u^+$ and $u_t\ge 0$ (then  all three inequalities are strict due to the strong maximum principle), and since all the $u_n$ satisfy the Remark  after Theorem \ref{T.1.11} with the same $\ell_\eps,T_\eps$ (and $-t_n$ in place of $t_0$),  $u$ has a bounded width.   Theorem \ref{T.1.5}(ii) now shows that $u$ is a transition solution because bounded width and Lemma \ref{L.11.1} yield a positive global mean speed of $u$, finishing the proof.

(ii)
We will only {\it sketch} the proof, since the mechanics of the workings of the counter-example which we  construct are more important than the detailed proof.  The latter would only add tedious technical details, obscuring the main ideas.
Let us also only consider the case $d=2$ because the general case is identical, with annuli below replaced by shells.  


To find $f$ such that there is no transition solution with doubly-bounded width for \eqref{1.1} (and thus also no transition front), it is sufficient to take some ignition $f_0$ and let $f$ be equal to $\beta f_0(u)$ outside the union of the discs $B_n:=B_n(n^3e_1)$ (for some $\beta\gg1$), and $f(x,u)=f_0(u)$ inside each $B_{n-1}(n^3e_1)$ (with a smooth transition between the two on $B_{n}(n^3e_1)\setminus B_{n-1}(n^3e_1)$).  If $u$ is a transition solution for \eqref{1.1} with a bounded width, let $t_n$ be the first time when $\sup_{x\in B_n}u(t_n,x)=\tfrac 1{10}$ (i.e., when the reaction zone of $u$ ``reaches'' $B_n$).  Since $\beta\gg1$, the reaction will spread all over $A_n:=B_{2n}(n^3e_1)\setminus B_n(n^3e_1)$ before it spreads to $B_{n/2}(x_n)$, as described in the introduction (see below for more details).  So at the (later) time $s_n$ when $\inf_{x\in B_{n}} u(s_n,x)=\tfrac 12$, we will also have $\inf_{x\in A_n} u(s_n,x)\ge \tfrac 12$.  It follows that $L^{u,\eps}\ge n$ for all $n$ and $\eps\in(\tfrac 12,1)$.  Hence $u$ does not have a doubly-bounded width.

We will need to use a more involved construction to obtain $\inf_{x\in\bbR^2} u(t,x)>0$ for any $t\in\bbR$ and any $u$ from (ii).
Let  $f_0(u)=(2u-1)(1-u)\chi_{[1/2,1]}(u)$ and let $R$ be such that if $u_t=\Delta u + f_0(u)$ on $(0,\infty)\times\bbR^2$ and $u(0,\cdot)\ge \tfrac 34 \chi_{B_R(0)}(\cdot)$, then $u\to 1$ locally uniformly as $t\to\infty$.  By Lemma \ref{L.2.1}, such $u$ also satisfies
\beq \lb{10.2}
\lim_{t\to\infty} \inf_{|x|\le ct} u(t,x) = 1
\eeq
for any $c<c_0$, with $c_0>0$ the spreading speed for $f_0$ (and we have $c_0\le 2\sqrt{\|f_0(u)/u\|_\infty}\le 2$).

Let $\beta>1$  be such that if $u_t=\Delta u + \beta f_0(u)$ on $(0,\infty)\times\bbR\times[-2,2]$ with Dirichlet boundary conditions and $u(0,\cdot)\ge \tfrac 34 \chi_{B_1(0)}(\cdot)$, then
\beq \lb{10.1}
\lim_{t\to\infty} \inf_{x\in [-100t,100t]\times[-1,1]} u(t,x) \ge \frac 45.
\eeq
That is, the reaction with strength $\beta$ spreads along a strip with a cold boundary at speed at least 100.  It is not difficult to show that this holds for large enough $\beta$.  

Next 
let $f(x,u)=a(|x|)f_0(u)$, where $a:[0,\infty)\to [1,\beta]$ is smooth, Lipschitz with the constant $\beta$, with $a(r)=\beta$ if $|r-2^{n}|\le 3$ for some $n\ge 3$, and with $a(r)=1$ if $|r-2^{ n}|\ge 4$ for each $n\ge 3$.  That is, the reaction is large on a sequence of annuli with uniformly bounded widths and exponentially growing radii, and small elsewhere.  We obviously have $f\in F(f_0,\beta,\tfrac 14,\zeta,\eta)$ for any $\zeta,\eta>0$ because $\alpha_f(\cdot;\zeta)>\tht_0$ for any $\zeta>0$. 

Then pick $\eps_0\in(0,\tfrac 12)$ such that if $u\in[0,1]$ solves \eqref{1.1}, \eqref{1.2}, then $\inf_{y\in B_R(x)} u(t,x)\ge \tfrac 34$ whenever $t\ge t_0+1$ and $u(t,x)\ge 1-\eps_0$ (which exists by parabolic regularity) and also such that the unique traveling front for $u_t = u_{xx}+f_0(u)$ connecting $\eps_0$ and 1 has speed $c_{\eps_0}< 1.1c_0$ (which is possible because  $\lim_{\eps\to 0}c_\eps=c_0$).

Assume now that $u\not\equiv 0,1$ is a bounded entire  solution for \eqref{1.1} with bounded width.  By Theorem \ref{T.1.5}(i) we have $u\in(0,1)$, so Lemma \ref{L.2.1} yields a positive global mean speed of  $u$. Then Theorem \ref{T.1.5}(ii) shows that $u$ is a transition solution with $u_t>0$.

Let $t_0$ be the first time such that $u(t_0,0)=\tfrac 12$ and for any large $n$, let $t_n$ be the first time such that $\sup_{x\in B_{2^n}(0)} u(t_n,x) = \eps_0$.  Then the maximum principle shows that there is $x_n\in\partial B_{2^n}(0)$ with $u(t_n,x_n)=\eps_0$.  Since $L^{u,\eps_0}<\infty$, our choice of $\eps_0$ and $R$  shows that there is $T>0$ such that $\inf_{x\in B_1(x_n)}u(t_n+T,x)\ge \tfrac 34$ for each $n$.  It then follows from  \eqref{10.1} and $100>20\pi$ that  
\beq \lb{10.4}
\inf_{x\in B_{2^n+1}(0)\setminus B_{2^n-1}(0)} u \left( t_n+T + \frac {2^n}{20},x \right)\ge \frac 34
\eeq
for  all large $n$.  From this and \eqref{10.2} it follows that for all large $n$,
\beq \lb{10.3}
\inf_{x\in B_{2^n}(0)\setminus B_{2^{n-1}}(0)} u \left( t_n+T + \frac {2^n}{20} + \frac {2^{n-1}}{0.9c_0},x \right)\ge \frac 34.
\eeq

At the same time, $\sup_{x\in B_{2^n}(0)} u(t_n,x) = \eps_0$ and $c_{\eps_0}< 1.1c_0$ show that $u(t, 0)< \tfrac 12$ for $t\le t_n + 2^n(1.1c_0)^{-1}$ if $n$ is large, because the reaction can propagate radially no faster than at speed $c_{\eps_0}$ on any wide  annulus where $a(|x|)=1$, provided $u\le\eps_0$ initially (this is similar to the upper bound on the propagation speed in Lemma \ref{L.2.1a}, and also uses the fact that  the annuli on which $a(|x|)> 1$ have widths $\le 4$, so they shorten the time to reach the origin only by an amount proportional to $n$).   So $t_n + 2^n(1.1c_0)^{-1}\le t_0$ for all large $n$, and if we let 
\[
s_n:= t_n+T + \frac{2^n}{20} + \frac{2^{n-1}}{0.9c_0} \quad \left(\le t_n + \frac{2^n}{1.1c_0} \text{ if $n$ is large because $c_0\le 2$}\right),
\]
then we obtain $s_n\le t_0$ for all large $n$.  But then \eqref{10.3} and $u_t>0$ show for all large $n$,
\[
\inf_{x\in B_{2^n}(0)\setminus B_{2^{n-1}}(0) }u(t_0,x)\ge \frac 34.
\]
The result now follows from Lemma \ref{L.2.1}.

\smallskip
{\it Remark.}  It is an interesting question whether for the reaction in (ii), all entire solutions $u\in(0,1)$ satisfy $\lim_{t\to\infty} \inf_{x\in\bbR^2} u(t,x)=1$.

\begin{proof}[A rough sketch of the proof that the claim involving \eqref{1.00} is false]
We construct here an example involving front-like solutions in $\bbR^2$ (essentially the same idea works  for spark-like solutions as well as for all dimensions $d\ge 2$).  The full proof of it working as claimed would be quite technical, but the following clearly illustrates the main idea.

For some rapidly growing $b_n\to\infty$,  define
\[
A:=\bigcup_{n\ge 1} A_n := \bigcup_{n\ge 1} \big[ \{x\,|\, 0\le x_1 \le b_n \text{ and } |x_2|=b_n \} \cup \{x\,|\, x_1=b_n \text{ and } |x_2|\le b_n \} \big] \subseteq\bbR^2. 
\]
We then let $f(x,u)=a(d(x,A))f_0(u)$, where $f_0$ is the ignition reaction from part (ii) of the above proof and $a:[0,\infty)\to[1,\beta]$ is smooth, with $a(s)=\beta$ for $s\le 1$ and $a(x)=1$ for $s\ge 2$.  Here $\beta\gg 1$ will be chosen later.

We also let $s_0\gg 1$ and $w:\bbR\to[0,1]$ be smooth and such that $w(s)=0$ for $s\ge 1$ and $w(s)=1$ for $s\le 0$.  We then define $v_0(x):=w(x_1)$ and let $u_0(x)\in[w(x_1),w(x_1-2s_0)]$ be smooth and such that $u_0(x)=w(x_1)$ for $x_2\le -1$ and $u_0(x)=w(x_1-2s_0)$ for $x_2\ge 1$.  Finally, let $u,v$ solve \eqref{1.1} on $(0,\infty)\times\bbR^2$ with $u(0,\cdot)=u_0$ and $v(0,\cdot)=v_0$, and let $t_n$ be the first time such that $u(t_n,b_n,0)=\tfrac 12$.    

It is obvious that $u,v$ satisfy the hypothesis of the claim involving \eqref{1.00} because $u_0\ge v_0$ and also $u_0\le v(T,\cdot)$ for some $T$. So if  \eqref{1.00} holds, we must have  
\beq\lb{10.5}
\lim_{n\to\infty} [u(t_n,b_n,r)-u(t_n,b_n,-r)]=0 \qquad \text{for any $r\in\bbR$}
\eeq
 because $v$ is obviously even in $x_2$.  However, if we take $\beta\gg 1$ and  sufficiently rapidly growing $b_n$, then the reaction zone of $u$ spreads towards $(b_n,0)$ along the two ``arms'' of $A_n$ much faster than through anywhere else, and that propagation is virtually unaffected by the other ``arms''.  This and the definition of $u_0$ means that the reaction zone moving towards $(b_n,0)$ along the upper arm of $A_n$ is distance $\sim 2s_0$ ahead of the one arriving along the lower arm.  This means that if $s_0$ is chosen sufficiently large, depending on $\beta$ (but not on $b_n$), then $\liminf_{n\to\infty} u(t_n,b_n,s_0)\ge \tfrac 34$ and  $\limsup_{n\to\infty} u(t_n,b_n,-s_0)\le \tfrac 14$. But this means $\liminf_{n\to\infty} [u(t_n,b_n,s_0)-u(t_n,b_n,-s_0)]>0$, a contradiction with \eqref{10.5}.
\end{proof}

{\it Remarks.} 1.  This example can easily be adjusted to $v$ being a transition solution with a bounded width such that $v(t,x)=V(x_1-c_0t)$ for $t\ll -1$, where $c_0$ is the front/spreading speed and $V$ the traveling front profile for $f_0$.
\smallskip

2.  If $u,v$ are not required to be front-like (or spark-like), conuter-examples to \eqref{1.00} can be constructed even for homogeneous reactions and dimensions $d\ge 2$.




\end{document}